\documentclass[12pt]{amsart}
\usepackage[utf8]{inputenc}
\usepackage{amssymb,amsbsy,amsmath,amscd,amsthm,amsfonts,MnSymbol,mathrsfs}
\usepackage{natbib}
\usepackage[colorlinks=true,
            citecolor=blue,
            linkcolor=black,
            urlcolor=blue]{hyperref}
\usepackage{graphicx}
\usepackage{verbatim}
\usepackage{float}
\usepackage{bbm}
\usepackage{multirow}
\usepackage{enumerate}
\usepackage{xcolor}
\usepackage[normalem]{ulem}
\usepackage{pdfsync}
\usepackage{mathtools}
 \usepackage{booktabs}       
\usepackage{natbib}
 \usepackage{algorithm}
 \usepackage{algorithmic}
 \usepackage{subcaption} 
 \usepackage{caption}    
\usepackage{etoolbox}
\makeatletter
\patchcmd{\thebibliography}{\list}{\list}{}{}
\makeatother

\theoremstyle{plain}
\newtheorem{exmp}{Example}
\newtheorem{theorem}{Theorem}[section]
\newtheorem{proposition}[theorem]{Proposition}
\newtheorem{lemma}[theorem]{Lemma}
\newtheorem{corollary}[theorem]{Corollary}
\theoremstyle{definition}
\newtheorem{definition}[theorem]{Definition}
\newtheorem{assumption}[theorem]{Assumption}

\newcommand\numberthis{\addtocounter{equation}{1}\tag{\theequation}}
\newcommand{\V}{\mathrm{Var}}
\newcommand{\tr }{\mathsf{T}}

\newcommand{\algc}[1]{\hfill$\triangleright$~\textit{#1}}
\newcommand{\afterctrlc}[1]{\STATE $\!\triangleright$~\textit{#1}}

\renewcommand{\baselinestretch}{1.24}

\title{Differentially Private Inference for Longitudinal Linear Regression}

\newif\ifuniqueAffiliation
\uniqueAffiliationtrue

\title[Differentially Private Inference for Longitudinal Linear Regression]{Differentially Private Inference for Longitudinal Linear Regression}
\author{Getoar Sopa}\thanks{Department of Statistics, Columbia University.}
\author{Marco Avella Medina}\thanks{
This research was partly
supported by  NSF grants DMS-2310973 (Avella Medina) and DMS-2413828 (Rush).}
\author{Cynthia Rush}

\begin{document}

\setcounter{page}{1} 
\begingroup
\renewcommand{\baselinestretch}{1.24}\selectfont
\begin{abstract}
  Differential Privacy (DP) provides a rigorous framework for releasing statistics while protecting  individual information present in a dataset. Although substantial progress has been made on differentially private linear regression, existing methods almost exclusively address the item-level DP setting, where each user contributes a single observation. Many scientific and economic applications instead involve longitudinal or panel data, in which each user contributes multiple dependent observations. In these settings, item-level DP offers inadequate protection, and user-level DP - shielding an individual’s entire trajectory - is the appropriate privacy notion.
We develop a comprehensive framework for estimation and inference in longitudinal linear regression under user-level DP. We propose a user-level private regression estimator based on aggregating local regressions, and we establish finite-sample guarantees and asymptotic normality under short-range dependence.  For inference, we develop a 
privatized, bias-corrected covariance estimator that is automatically heteroskedasticity- and autocorrelation-consistent. 
These results provide the first unified framework for practical user-level DP estimation and  inference in longitudinal linear regression under dependence, with strong theoretical guarantees and promising empirical performance.
\end{abstract}
\endgroup

\maketitle

\section{Introduction}
Differential Privacy (DP) is a general framework that facilitates worst-case privacy guarantees for users whose data is employed in calculating published statistics \citep{dwork2006}. While initially a mostly theoretical construction, DP has seen widespread adoption in public, commercial, and research applications \citep{abowd2018us, cormode2018privacy, liu2024survey}. The fundamental idea behind DP is that properly calibrated noise \textit{must} be injected into non-trivial statistics to ensure private releases. The paramount challenge in constructing DP statistics, then, is constructing randomization mechanisms in such a way as to provide meaningful privacy guarantees while minimizing the loss of statistical efficiency. 

Perhaps the foremost  tool in applied statistics is linear regression. It is therefore not surprising that DP estimation and inference for linear regression has been the subject of much work \citep{sheffet2019old,barrientos2019differentially, cai2021cost, alabi2020differentially, milionis2022differentially,  alabi2023differentially, amineasy,cai2023private,barrientos2024feasibility}. Studies of linear regression under DP constraints have so far been constructed in the \textit{item-level} DP setting, where each user contributes exactly one datapoint, and the target of the DP randomization is to mask the contribution of any one user's datapoint. There are many applications, however, where every user contributes multiple datapoints to the linear regression problem, such as in medical experiments \citep{flynn2009effects,  church2010effects}, social science studies 
\citep{sobel2012does, mewes2021experiences}, and microeconomic data analyses \citep{jacobson1993earnings, autor2014trade}.  We call this \textit{longitudinal} or \textit{panel} linear regression. Due to the temporal dependence usually observed in such datasets, iid data models are not appropriate in these applications. 

Moreover, in  statistical applications such as longitudinal linear regression, item-level DP does not provide satisfactory privacy guarantees. Intuitively, we do not consider a user's contribution to a statistic to be private unless all of that user's datapoints are masked simultaneously. Many recent works in DP have addressed this problem by incorporating a framework called user-level (or person-level) Differential Privacy (uDP) that protects against worst-case changes in a single user's entire data contribution \citep{levy:sun:amin:kale:kulesza:mohri:suresh2021,cummings:feldman:mcmillan:talwar2022,ghazietal2023,kent2024rate,agarwal:kamath:majud:mouzakis:silver:ullman2025}.
While uDP is also a natural privacy framework for longitudinal data,   to our knowledge all existing work on uDP assumes iid data, which as mentioned above is not usually a good model for longitudinal regression. 
An exception considering both uDP and dependency is recent work by one of the authors \citep{roth:avellamedina2025} that analyzes  uDP mean estimation rates under dependence characterized by log-Sobolev inequalities. However, that paper does not  address inference so the procedures therein are primarily of theoretical interest rather than motivated by practical considerations. 

In this work, we study estimation and inference in the linear regression model under user-level DP constraints and temporal dependence. 
The main contributions of this paper are summarized as follows:
\begin{enumerate}[(a)]
    \item \textbf{Adaptive uDP mean estimation:} We propose a novel algorithm for performing mean estimation  using minimal knowledge of the data generating process. This is especially important in the longitudinal regression setting for two main reasons. First, the algorithm is a practical one that can be applied out-of-box 
    to real regression problems with limited apriori knowledge. Second, longitudinal data often have significant intertemporal dependence, which makes conventional tuning strategies for clip-and-noise  algorithms not directly applicable. We show empirically that our algorithm achieves very strong performance for moderate sample sizes. Our methods extend ideas of \cite{karwa:vadhan2018,biswas:dong:kamath:ullman2020} to general, possibly dependent data. 
    \item \textbf{uDP Linear regression estimation:} We propose a user-level private estimator for the longitudinal linear regression coefficients based on an average of  linear regression estimates computed from each user's data. We argue that such an estimator often outperforms estimators based on privatizing summary statistics in the uDP setting, and prove finite sample and asymptotic convergence of the private estimator under short-term dependence in users' data. 
    \item \textbf{uDP Inference:} We propose user-level private estimators for the asymptotic covariance of the private OLS estimator based on a privatized sample covariance estimator with a finite sample correction term. We show that this estimator is consistent and automatically heteroskedasticity and autocorrelation consistent, and demonstrate its accuracy in small samples. Our construction shares similarities with cluster-robust standard errors \citep{liang1986longitudinal,arellano1987computing}, but is adapted to our private distributed setting.  
    \item \textbf{uDP Difference-between-groups estimator:} We propose estimators for the difference in coefficient between two groups (for instance, treatment and control), and the asymptotic covariance thereof, and prove corresponding convergence guarantees. 
    This contribution is particularly interesting in practice as such variables are commonly encountered in applications in the form of categorical variables such as treatment status, sex, or  ethnicity, just to name a few.
    
    \item \textbf{Concentration results under dependence:} We obtained various new concentration inequalities under strong mixing as a side product of our analysis of OLS  with longitudinal data. In particular, a key intermediate step in our construction generalizes the finite sample analysis of  \cite{hsu:kakade:zhang2012} to autocorrelated settings. For this we relied on the work of \cite{merlevede:peligrad:rio2011}. Our results and techniques also bear some resemblance to recent work on regression under short-term dependence \citep{fan2016penalized}. We believe our concentration results are interesting on their own right and are likely to find many more useful applications in future work. 
\end{enumerate}
The remainder of this article is structured as follows. In Section \ref{sec:preliminaries} we introduce the necessary background from differential privacy and probability theory to present our main results. In Section \ref{sec:dpmean} we give the core mean estimation algorithm that underpins our longitudinal estimators. The algorithm returns a DP mean estimator that will be instantiated in appropriate ways in later sections to provide uDP estimators with strong statistical guarantees. In Section \ref{sec:dpregression} we present our main private longitudinal linear regression estimator. This estimator is complemented in Section \ref{sec:inf} with a private estimator of its asymptotic covariance. We use the estimators  to construct private approximate  pivots for consistent inference. We illustrate the practicality of our methods in numerical examples with modest sample sizes using simulated data in Section \ref{sec:simulations} and with the Childhood Asthma Management Program (CAMP) teaching dataset \citep{childhood2000long} in Section \ref{sec:real_data}. The Appendix contains all the proofs.

\section{Preliminaries}
\label{sec:preliminaries}

For a natural number $K \in \mathbb{N}$, we let $[K]$ denote the set $\{1, \dots, K\}$. For a vector $v \in \mathbb{R}^p$, we write $\|v\| = \sqrt{v^\tr v}$. For a real matrix $A \in \mathbb{R}^{p \times q}$,  we write $\|A\| = \max_{v \in \mathbb{R}^d: \|v\| = 1} \|Av\|$, and we let $\lambda_{\min}(A)$ and $\lambda_{\max}(A)$ denote the smallest and largest eigenvalue of $A$, respectively. We denote the space of real symmetric $d \times d$ matrices by $\mathcal{M}_{symm}^{d \times d}$.   We write $f(x) = O(g(x))$ if $\exists C>0$ such that $f(\cdot ) \leq C(g(\cdot ))$ for any admissible arguments of $f$ and $g$. Similarly, we write $f(x) = \Omega(g(x))$ if $\exists c>0$ such that $f(\cdot) \geq cg(\cdot)$ for any admissible argument of $f$ and $g$, and $f(x) \asymp g(x)$ if $f(x) = O(g(x))$ and $f(x) = \Omega(g(x))$. Moreover, we write $f(x) = o(g(x))$ if $\frac{f(x)}{g(x)} \to 0$ as $x \to \infty$. Finally, for an event $E$ we let $\mathbbm{1}\{E\}$ or $\mathbbm{1}_E$ denote the indicator function over the event $E$. 

\subsection{Gaussian Differential Privacy}
The measure of privacy that we maintain in this paper is Gaussian Differential Privacy (GDP), which provides a naturally interpretable framework and several desirable properties \citep{dong2022gaussian}. To define GDP, we must first introduce the notion of a \textit{trade-off function}.
\begin{definition}
    Let $P$ and $Q$ be two probability distributions defined on the same measurable space $(\mathcal{X}, \mathcal{F})$. A (potentially randomized) rejection rule  is any $\mathcal{F}$-measurable function $\phi : \mathcal{X} \mapsto [0,1]$. 
    The trade-off function $T(P,Q): [0,1] \mapsto [0,1]$ is defined as
    $
    T(P,Q)(\alpha) = \inf \{1 - \mathbb{E}_Q[\phi]:  \mathbb{E}_P[\phi] \leq \alpha\},
    $
    where the infimum is taken over the set of rejection rules $\phi$. 
\end{definition}
Intuitively, the trade-off function is a measure of the difficulty of distinguishing between distributions $P$ and $Q$ based on statistical hypothesis testing: for any given $\alpha \in [0,1]$, if the probability of a false rejection under $P$ is bounded by $\alpha$, it quantifies the smallest possible probability of a false rejection under $Q$. GDP assesses the privacy of an algorithm by comparing the difficulty in distinguishing between the output generated by two neighboring datasets against the difficulty in distinguishing two Gaussians. To that end, for a fixed privacy parameter $\mu>0$, define $
G_\mu := T\left(N(0,1), N(\mu,1)  \right)$,
which satisfies 
$
G_\mu(\alpha) =\Phi(\Phi^{-1}(1-\alpha)-\mu)$ for $\alpha \in [0,1]$,  
where $\Phi(\cdot)$ denotes the standard normal CDF. 
To characterize user-level Gaussian Differential Privacy (uGDP), we first introduce the notion of neighboring datasets in the uDP framework. Here, a dataset is a matrix $X^n \in \mathbb{R}^{nT \times d}$ containing $T$ measurements that are $d$-dimensional, for each of $n$ users.
\begin{definition}
\label{def:dataset}
    Let $X^n = (X_1^\tr, \dots, X_n^\tr)^\tr  $ with $X_i \in \mathbb{R}^{T \times d}$  and $X^{'n} = (X_1'^\tr, \dots X_n'^\tr)$ be an equally sized dataset. $X^n$ and $X^{'n}$ are user-level neighboring if $d_H(X^{n}, X^{'n}):= \sum_{i=1}^n \mathbbm{1}\{X_i \neq X_i' \} = 1 $ where $\mathbbm{1}\{X_i \neq X_i' \} = 0$ only if all elements of the matrices $X_i$ and $X'_i$ are equal.
\end{definition}
We emphasize that two datasets are user-level neighbors if they differ in the data of exactly one of $n$ users, where it is allowed for all $T$ entries of that user to differ between the datasets. 
\begin{definition}
     A randomized statistic $ h$ satisfies $\mu$-uGDP if  $T( h(X^n),  h(X^{'n}))(\alpha) \geq G_\mu(\alpha)$ for \textit{all} two user-level neighboring datasets $X^n$ and $X^{'n}$ and
    for all $\alpha \in[0,1].$
\end{definition}
Hence, a randomized statistic satisfies $\mu$-uGDP if distinguishing between the output generated by any two neighboring inputs is at least as difficult as distinguishing between a $N(0,1)$ and $N(\mu,1)$ based on a single draw. Thus, this notion of differential privacy is easily understood in the hypothesis testing framework \citep{wasserman2010statistical}. We emphasize that the randomness considered in the trade-off function comes from the random nature of the statistic $h$ rather than the data $X^n$, which for this purpose is considered fixed. Notice that as $\mu$ shrinks, the statistic $h$ becomes more private as distinguishing becomes harder.
The canonical way of ensuring an algorithm satisfies GDP is to add properly calibrated Gaussian noise. To formalize this, we first introduce the notion of global sensitivity of an algorithm.
\begin{definition}
   Let $h:\mathcal{X}^n\to \mathbb{R}^d$ be a statistic. Then, the user-level global sensitivity of $h$ is defined as
    $
    \text{GS}(h) = \sup_
    {d_H({X^n}, {X}^{'n}) = 1} \|{h}({X}^n) - {h}({X}^{'n})\|.$
\end{definition}
This allows us to introduce the Gaussian mechanism and its privacy guarantee.
\begin{proposition}[Theorem 1 of \citealt{dong2022gaussian}]
\label{prop:gaussianmechanism}
   Let $h$ be a statistic with user-level global sensitivity $\text{GS}({h})$. Then, the randomized statistic $\tilde{h}(X^n) =h(X^n)+ \frac{\text{GS}(h)}{\mu}Z$, with $Z \sim N(0,I_d)$, satisfies $\mu$-uGDP.
\end{proposition}
Finally, we make use of the following two important properties of GDP that allow us to derive the privacy guarantees of more complicated algorithms.
\begin{proposition}[Proposition 4 of \citealt{dong2022gaussian}]
\label{prop:postprocessing}
    Let the randomized statistic ${h}: \mathcal{X}^n \mapsto \mathcal{Y}_1$ satisfy $\mu$-uGDP. Furthermore, let ${g} : \mathcal{Y}_1 \mapsto \mathcal{Y}_2$ be a (possibly random) map from the output of the statistic ${h}$ onto the space $\mathcal{Y}_2$. Then ${g} \circ {h}: \mathcal{X}^n \mapsto \mathcal{Y}_2$ also satisfies $\mu$-uGDP.
\end{proposition}
Intuitively, the preceding proposition, which in the DP literature is referred to as the \textit{post-processing property}, says that processing the output of private statistic cannot degrade its privacy guarantee. 
\begin{proposition}[Corollary 2 of \citealt{dong2022gaussian}]\label{prop:composition}
    For $j \in [N]$, let ${h}_j$ be randomized statistics that satisfy $\mu_j$-uGDP. Then their $N$-fold composition satisfies $\sqrt{\sum_{j=1}^N \mu_j^2}$-uGDP.
\end{proposition}
The above results is known as the \textit{composition property} of GDP, and makes it possible to tightly quantify the privacy loss incurred by releasing multiple privatized statistics that possibly take previously privatized statistics as input. For example, if the randomized statistic $h: \mathcal{X}^n \mapsto \mathcal{Y}_1$ satisfies $\mu/\sqrt{2}$-GDP and the randomized statistic $g: \mathcal{X}^{n-1} \times \mathcal{Y}_1 \mapsto \mathcal{Y}_2$ satisfies $\mu/\sqrt{2}$-GDP, then jointly outputting $h(X^n)$ and $g(X^n, h(X^n))$ satisfies $\sqrt{\mu^2 / 2 + \mu^2 / 2} = \mu$-GDP. Propositions \ref{prop:postprocessing} and \ref{prop:composition} allow us to construct intricate private algorithms using simple building blocks and still provide meaningful privacy guarantees.  

\subsection{Probability theory}
We review some basic probabilistic notions that recur throughout this work, beginning with sub-Gaussian random vectors.
\begin{definition}\label{def:subgaussian}
    A random vector $X \in \mathbb{R}^d$ is $\sigma_X^2$-sub-Gaussian if for all $u \in \mathbb{S}^{d-1}$ and $\lambda \in \mathbb{R}$, it satisfies
    $\mathbb{E}[e^{\lambda u^\tr(X - \mathbb{E}[X]) }] \leq \exp({\sigma_X^2 \lambda^2}/{2}).$
\end{definition}
Next, we define the types of stochastic processes that are used in this paper. Definitions \ref{def:stochprocess}-\ref{def:causalarma} below are standard, adapted from \cite{brockwell1991time}.
\begin{definition}\label{def:stochprocess}
    A real-valued stochastic process $(\varepsilon_t)_{t \in \mathbb{Z}}$ is a collection of random variables defined on a joint probability space $(\mathcal{X}, \mathcal{F}, \mathbb{P})$. The process $(\varepsilon_t)_{t \in \mathbb{Z}}$ is weakly stationary if $\mathbb{E}[\varepsilon_t^2] <\infty$ 
    and $\mathbb{E}[\varepsilon_t] = m$ for all $t \in \mathbb{Z}$, and if furthermore $\gamma(t,s) := \text{Cov}(\varepsilon_t, \varepsilon_s) = \gamma(t+r,s+r)$ for all $t,s,r \in \mathbb{Z}$.
\end{definition}

We further define the concept of strongly mixing sequences. Let $(X_t)_{t \in \mathbb{Z}}$ be a process in $\mathbb{R}^d$ and let $\mathcal{F}_a^b = \sigma(X_a, \dots, X_b)$ be the sigma algebra induced by $X_a$ through $X_b$. 

\begin{definition}\label{def:strongmixing}
    The $\alpha$-mixing coefficients of a stochastic process $(X_t)_{t\in \mathbb{Z}}$ are defined as 
    $
    \alpha(k) = \sup_{t \in \mathbb{Z}} \sup_{A \in \mathcal{F}_{-\infty}^t, \, B \in \mathcal{F}_{t + k}^\infty} |\mathbb{P}(A \cap B) - \mathbb{P}(A)\mathbb{P}(B)|,
    $
    and we say that the process is strongly mixing if $\alpha(k) \to 0$ as $k \to \infty$. Furthermore, a process is exponentially strongly mixing if $\exists C,c>0$ such that $\forall k\geq1$ we have $\alpha(k) \leq Ce^{-ck}$.
\end{definition}
Finally, all of our motivating examples are in the class of ARMA processes, and in particular in causal ARMA processes. Intuitively, causality guarantees that the effect of errors or initial values decays as time progresses. These definitions straightforwardly extend to multivariate time series.
\begin{definition}\label{def:arma}
    The weakly stationary stochastic process $(\varepsilon_t)_{t \in \mathbb{Z}}$ is an ARMA(p,q) process if it can be represented as
    $
    \varepsilon_t - \sum_{k=1}^p \phi_k \varepsilon_{t-k} = r_t + \sum_{l=1}^q \theta_l r_{t-l},
    $
    where $r_t$ is mean zero and $\text{Cov}(r_t, r_s) = \mathbbm{1}_{t=s}\sigma^2$ for all $t,s \in \mathbb{Z}$. Defining the polynomials $\phi(z) := 1 - \sum_{k=1}^p \phi_k z^k$ and $\theta(z) := 1 + \sum_{l=1}^q \theta_l z^l$, the ARMA process can be represented as
    $\phi(B)\varepsilon_t = \theta(B) r_t,$
    where $B$ is the backshift operator. We assume that $\phi(\cdot)$ and $\theta(\cdot)$ have no common zeros.
\end{definition}
\begin{definition}\label{def:causalarma}
    The ARMA(p,q) process $(\varepsilon_t)_{t \in \mathbb{Z}}$ as defined in \ref{def:arma}  is said to be causal iff all zeros of $\phi(z)$ (defined in Definition~\ref{def:arma}) lie strictly outside the unit circle. In this case, there exist absolutely summable coefficients $(\psi_j)_{j \geq 0}$ such that
    $\varepsilon_t = \sum_{j \geq 0} \psi_j r_{t-j}$ for all $t\in\mathbb{Z}.$
\end{definition}

\section{Adaptive differentially private mean estimation}
\label{sec:dpmean}
Our main workhorse throughout this paper is  Algorithm \ref{alg:trimestTemp}, which is an adaptive GDP mean-estimation algorithm. It is conceptually similar to existing ideas in the literature such as the Coinpress algorithm of \citep{biswas:dong:kamath:ullman2020} and precursors \citep{karwa:vadhan2018, kamath:li:singhal:ullman2019}, with a major difference being that Algorithm \ref{alg:trimestTemp} does not require significant structural assumptions or knowledge of the underlying data generating process. In the next section, we will see that it can be instantiated to construct uGDP estimators with strong statistical guarantees.
\begin{algorithm}
\small
\begingroup
\renewcommand{\baselinestretch}{1}\selectfont
\caption{DP Trimming-based  mean estimation \hfill \textsc{DPTrimMean}$((D_i)_{i=1}^n,\mu,B,R,\xi)$}
\label{alg:trimestTemp}
\begin{algorithmic}[1]
\STATE \textbf{Input:} Data $(D_i)_{i=1}^n$, privacy parameter $\mu$, radius $B$, 
max rounds $R$, failure probability $\xi$\
\STATE \textbf{Set:} $\tau \gets n - \frac{2}{\mu}{\sqrt{2R\log\frac{{4R}}{\xi}}}$ and  $n_{LB} \gets \max\left(2\tau - n,1\right)$ 
\STATE \textbf{Init:} $\displaystyle \hat{m}_{-2} \gets 0$, $\hat{m}_{-1} \gets 0$, $r \gets 0$, $r^* \gets -1$ 
\WHILE{$r \le R$ and $r^* = -1$}
    \STATE $\omega_i \gets \mathbbm{1}\{\|D_i-\hat m_{r-1}\|\le B/2^r\}$, \qquad
           $\omega \gets \sum_{i=1}^n \omega_i + \tfrac{2\sqrt{R}}{\mu}Z_r$, \; $Z_r\sim\mathcal{N}(0,1)$ 
    \IF{$\omega < \tau$}
        \STATE $S \gets \{ i : \|D_i- {\hat m_{r-2}}\| < B/2^{r-1} \}$, \qquad $n_{\tilde{S}} \gets \max(|S|, n_{LB})$, \qquad $\mathcal{C} \gets \sqrt{2 - \frac{r}{R}}$
        \STATE $\hat m_{\mathsf{DP}} \gets {\hat m_{r-2}} + \frac{1}{n_{\tilde{S}}}\sum_{i\in S}(D_i- {\hat m_{r-2}}) + \frac{2\sqrt{2}B}{2^{r-1}\mathcal{C}{\mu} n_{LB}}N_{*}$, \; $N_{*}\sim\mathcal{N}(0,I)$ \label{line:mhatmhatminus2}
        \STATE $r^* \gets r-1$, \qquad $B^* \gets \frac{2\sqrt{2}B}{2^{r^*}\mathcal{C} \mu n_{LB}}$
    \ELSIF{$r=R$}
        \STATE $S \gets \{ i : \|D_i-\hat m_{R-1}\| < B/2^{R} \}$, \qquad  $n_{\tilde{S}} \gets \max(|S|, n_{LB})$
        \STATE $\hat m_{\mathsf{DP}} \gets \hat m_{R-1} + \frac{1}{n_{\tilde{S}}}\sum_{i\in S}(D_i-\hat m_{R-1}) + \frac{2\sqrt{2}B}{2^{R}\mu n_{LB}}N_{*}$, \; $N_{*}\sim\mathcal{N}(0,I)$  
        \STATE $r^* \gets R$, \qquad $B^* \gets \frac{2\sqrt{2}B}{2^{r^*} \mu n_{LB}}$
    \ELSE  
        \STATE $S \gets \{ i : \|D_i-\hat m_{r-1}\| < B/2^{r} \}$, \qquad $n_{\tilde{S}} \gets \max(|S|, n_{LB})$ \label{step:iterateS}
        \STATE $\hat m_r \gets \hat m_{r-1} + \frac{1}{n_{\tilde{S}}}\sum_{i\in S}(D_i-\hat m_{r-1}) + \frac{4B\sqrt{R}}{2^{r}\mu n_{LB}}N_r$, \; $N_r\sim\mathcal{N}(0,I)$  \label{step:iteratem}
        \STATE $r \gets r+1$ 
    \ENDIF
\ENDWHILE
\STATE \textbf{Output:} $(\hat{m}_{\mathsf{DP}}, \hat{m}_{r^*-1}, r^*, B^*)$
\end{algorithmic}
\endgroup
\end{algorithm}

Our mean-estimation algorithm works as follows. We start with a ball of large radius, centered without loss of generality at $\mathbf{0}$.
The algorithm then iteratively privately tests whether enough user data falls inside a new ball centered at the previous private mean estimate, with a radius that is a constant factor smaller than the previous radius. If yes, the algorithm estimates a new private mean using the data inside the refined ball, which has a reduced sensitivity compared to the previous ball. It repeats this process until a private test indicates that too many datapoints fall outside a newly proposed ball, at which point a final private estimate is computed over the datapoints inside the final ball that passed the private test. In this sense, our algorithm can also be interpreted as an automated, iterative Propose-Test-Release implementation \citep{dwork:lei2009,avellamedina:brunel2020}. In our instantiation of the algorithm, we configure our private test threshold in such a way that the algorithm will with high probability not terminate in any given iteration as long as all users fall inside that iteration's ball. We note that if we are only interested in estimating the mean of our input data, Algorithm~\ref{alg:trimestTemp} could output $\hat{m}_{\mathsf{DP}}$, the final mean estimate, only. However, Algorithm~\ref{alg:trimestTemp} also outputs the quantities $r^*$, $B^*$ and $\hat{m}_{r^* - 1}$, which are the algorithm's termination iteration, final noise scale, and final ball center, respectively. These additional output values will be used later for privately estimating covariance matrices in Section \ref{sec:inf}.

The main practical advantage of our algorithm compared to existing methods is that it adapts in a data driven way to the level of concentration present in the users' data. 
Thus, the scale of truncation and noise addition is neither too large nor too small, which the algorithm achieves through private testing instead of requiring structural assumptions such as Gaussian tails or within constant factor knowledge of the covariance, like in \cite{karwa:vadhan2018,kamath:li:singhal:ullman2019,biswas:dong:kamath:ullman2020}.
In this construction, in addition to deciding a priori on a privacy budget $\mu$, one also has to provide three tuning parameters corresponding to a  bound $B$ on the data, the maximum number of refinement rounds $R$ and a failure probability $\xi$ that we want our estimator to achieve in finite samples. Since the radius of the projection ball is explored in an exponentially shrinking grid, in practice one can choose large conservative values of $B$ for the initial radius.\footnote{In fact, if $B$ is too small in the first iteration, one could increase the parameters $B$, $R$, update $\mu\gets \mu(1-1/(2\sqrt{R}))$, and restart the algorithm. For simplicity we do not consider this.}

We illustrate our algorithm with a simple example, where $D_i \overset{iid}{\sim} N([50, 50]^{\tr  }, 10I_2)$, for $i = 1, \dots, 300$. We initialize the algorithm with privacy parameter $\mu = 1$, radius $B = 10^{14}$, maximum number of rounds $R = 50$, and failure probability $\xi = 0.05$. Note that $B$ is set much higher than necessary and is large enough to accommodate near unbounded data. The algorithm terminated in $r^* = 43$ iterations. While the initial noisy estimate is $\hat{m}_0 = [2.3 \cdot 10^{13}, 8.09 \cdot 10^{12}]^{\tr  }$, the final output $\hat{m}_{\mathsf{DP}} =[49.93, 50.22]^{\tr  }$ is close to the truth $[50, 50]^{\tr  }$. Moreover, the final radius used is $\frac{B}{2^{r^*}}= 10^{14}/2^{43} = 11.37$, which is very close to the theoretical uniform concentration radius. To see this, notice that $\|D_i - [50, 50]^\tr\|^2 /10\sim \chi_2^2$, so 
\[
\mathbb{P}\Big(\max_{i = 1, \dots, 300} \|D_i - [50,50]^\tr\| \leq \sqrt{20\log\frac{1}{1-(1-\xi)^{1/300}}} \,\Big) = 1-\xi,
\]
and for $\xi = 0.05$, this radius is equal to approximately $13.17$. Our radius is slightly smaller as we allow a small number of users to fall outside the projection ball. However, our algorithm was able to find this radius privately, {without prior knowledge of the data distribution}. Figure \ref{fig:dpillustration} displays the last five steps of the algorithm.

The main result regarding Algorithm \ref{alg:trimestTemp} is stated in Theorem \ref{thm:trimest}. It concerns deterministic data and the randomness comes entirely from the randomness injected by the algorithm. The proof can be found in Appendix \ref{app:trimest}.

\begin{theorem}
\label{thm:trimest}
    Algorithm \ref{alg:trimestTemp} satisfies $\mu$-GDP. Furthermore, assume there exist $B, K_1, K_2>0$ and a constant $m$ such that
    $\max_{i \in [n]} \|D_i\| \leq B$,  $\|\bar{D} - m\| \leq K_1$ and  $\max_{i \in [n]}\|{D}_i - m\| \leq K_2$, where $\bar{D} := \frac{1}{n}\sum_{i=1}^n D_i$.
    If $R \geq C \log \frac{B}{K_2}$ and $n \geq C'(\frac{1}{\mu}\sqrt{R(\log\frac{R}{\xi}+ d)})$  for universal constants $C, C'>0$, then Algorithm \ref{alg:trimestTemp} satisfies, with probability at least $1-\xi$,
    \[
    \|\hat{m}_{\mathsf{DP}} - m\| = O\Big(K_1 +  \frac{K_2}{n \mu}\sqrt{R \log \frac{R}{\xi}} + \frac{K_2}{n \mu}\sqrt{d + \log \frac{{1}}{\xi}}  \Big) = O\Big(K_1 +  \frac{K_2}{n \mu}\sqrt{d + R \log \frac{R}{\xi}} \Big).
    \]
\end{theorem}
The above results can easily be modified to accommodate random data by exploiting concentration results and analyzing the algorithm on a high-probability event where the user data are well concentrated. In the following, we illustrate this with a prototypical iid sub-Gaussian uDP mean estimation problem. We first note the following standard concentration result for this setting.

\begin{lemma}\label{lemma:subgaussianmeanconc}
    Let $w_{i,t} \overset{iid}{\sim} \text{subG}(\sigma^2)$ for $i \in [n],t \in [T]$ and $m := \mathbb{E}[w_{i,t}] \in \mathbb{R}^d $. Define $(\bar{w}_i)_{i=1}^n$ with $\bar{w}_i := \frac{1}{T} \sum_{t=1}^T w_{i,t}$. Then, for absolute constant $C>0$, with probability at least $1-\xi$,
    \[
    \max_{i \in [n]} \|\bar{w}_i - m\| \leq C\sigma\sqrt{\frac{1}{T}\Big(d + \log \frac{n}{\xi}\Big)} \quad \text{ and } \quad \Big\|\frac{1}{n}\sum_{i=1}^n \bar{w}_i - m\Big\| \leq C\sigma\sqrt{\frac{1}{nT}\Big(d + \log \frac{1}{\xi}\Big)}.    \]
\end{lemma}

We may now apply Theorem \ref{thm:trimest} to the problem of user-level privately estimating $m$ by analyzing Algorithm \ref{alg:trimestTemp} on the $1-\xi$ probability event of Lemma \ref{lemma:subgaussianmeanconc}. In particular, this allows us to take $K_1 = O(\sqrt{\frac{1}{nT}(d + \log \frac{1}{\xi})})$ and $K_2 = O(\sqrt{\frac{1}{T}(d + \log \frac{n}{\xi})})$ in Theorem \ref{thm:trimest}. This immediately yields the following corollary. 
\begin{corollary}\label{corollary:subgaussianmeanestimation}   Let $w_{i,t} \overset{iid}{\sim} \text{subG}(\sigma^2)$ for $i \in [n],t \in [T]$ and $m := \mathbb{E}[w_{i,t}] \in \mathbb{R}^d $. Applying Algorithm \ref{alg:trimestTemp}  with privacy parameter $\mu$ to $(\bar{w}_i)_{i=1}^n$ where $\bar{w}_i := \frac{1}{T} \sum_{t=1}^T w_{i,t}$ yields a $\mu$-uGDP estimate $\hat{m}_{\mathsf{DP}}$. Moreover, given the scaling conditions of Theorem \ref{thm:trimest}, with probability at least $1-\xi$,
    \[
    \|\hat{m}_{\mathsf{DP}} - m\| = O\Big(\sigma \sqrt{\frac{1}{nT}\Big(d + \log \frac{1}{\xi}\Big)} + \sigma\sqrt{\frac{1}{n^2\mu^2T}\Big(R \log \frac{R}{\xi}+d \Big)\Big(d + \log \frac{n}{\xi}\Big)}   \Big).
    \]
\end{corollary}
Here the $R \log \frac{R}{\xi}$ term is the cost of our private adaptive search for a radius of correct order. In particular, `oracle' procedures where the exact concentration radius is known apriori, such as Algorithm 2 of \cite{levy:sun:amin:kale:kulesza:mohri:suresh2021}, yields a similar error bound, with $R \log \frac{R}{\xi}$ replaced by $\log \frac{1}{\xi}$.  In Section~\ref{sec:dpregression}, we apply these tools to the problem of user-level privately estimating the coefficients of a linear regression problem. 
\begin{figure}[t]
\centering
\begin{minipage}[c]{0.4\linewidth}
  \centering
  \includegraphics[width=\linewidth]{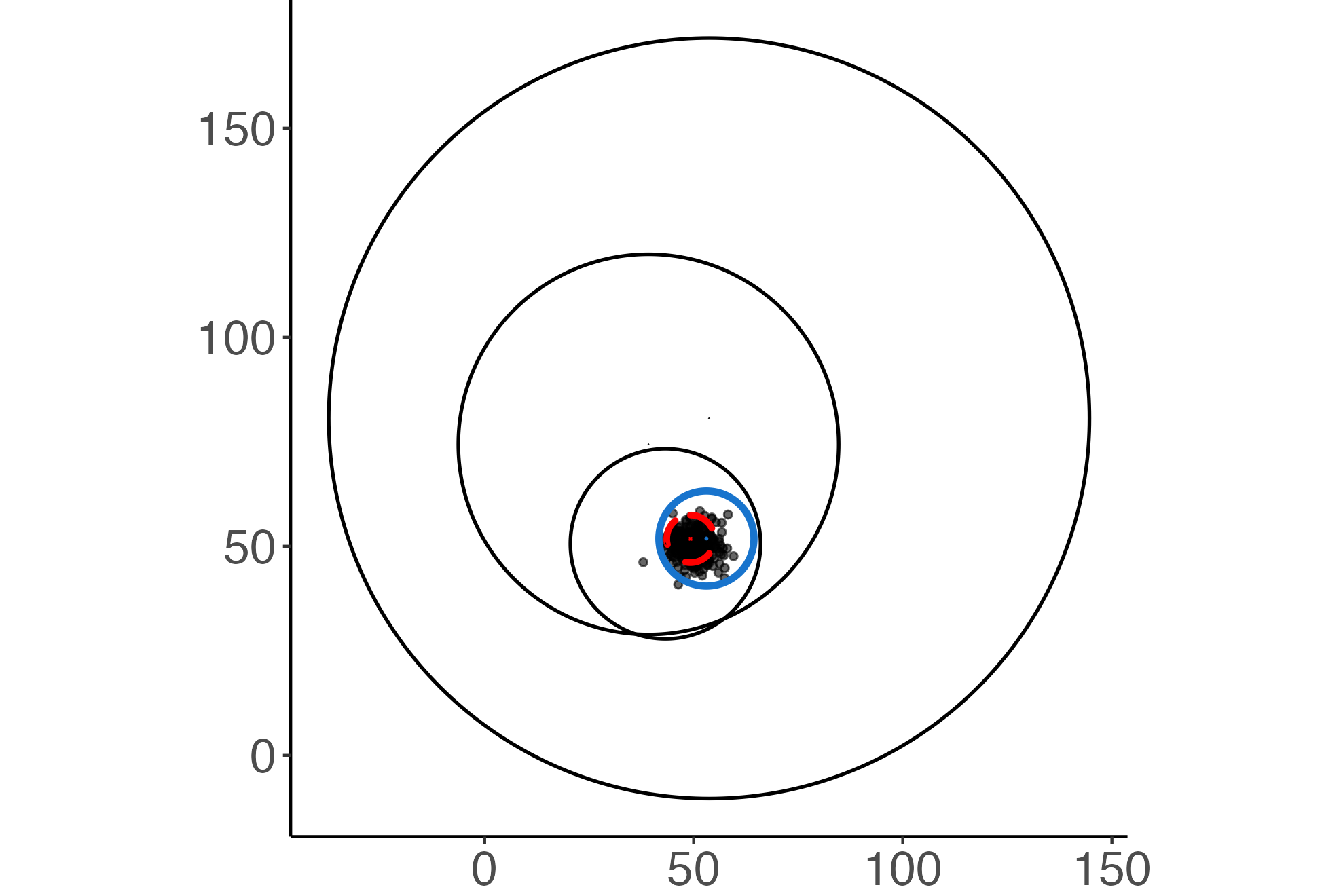}
\end{minipage}\hspace{0\linewidth}
\begin{minipage}[c]{0.58\linewidth}
  \caption{\small An illustration of the final five iterations of Algorithm \ref{alg:trimestTemp}.  The dashed, red circle represents the final attempted and rejected mean refinement, whereas the blue circle is the final projection ball.}
  \label{fig:dpillustration}
\end{minipage}
\end{figure}
\subsection{Discussion}
Algorithm \ref{alg:trimestTemp} privately identifies the optimal projection radius by iteratively running a sequence of sensitivity one queries, and terminating when one of these queries falls below a threshold. This bears resemblance to the AboveThreshold framework (e.g., Algorithm 1 in \citealt{dwork:roth2014}), which additionally adds noise to the threshold in the $\epsilon$-DP framework. 
By doing so, adding noise with constant order variance suffices to achieve a fixed privacy guarantee, regardless of the amount of queries that are potentially executed.  Similar manipulations might be possible in the $\mu$-GDP setting and could potentially lead to an improvement of the bound in Theorem \ref{thm:trimest}, but a thorough exploration of this goes beyond the scope of this paper. 
\section{Differentially Private longitudinal linear regression estimation}
\label{sec:dpregression}
\subsection{Set-up and assumptions}
The first, most basic longitudinal linear regression model that we examine in the framework of uGDP is 
\begin{align}
\label{eq:basemodel}
    y_{i,t} = \beta^\tr   x_{i,t} + \varepsilon_{i,t},
\end{align}
where $i \in [n]$ denotes the user index, $t \in [T]$ denotes the time index, and $x_{i,t}$ is a collection of time-varying covariates, with the possible exception of a constant term to accommodate an intercept.\footnote{Our results do not change if each user has a unique intercept, but the guarantees will only apply to non-intercept coefficients.} Our coefficient of interest is $\beta$. In the following, we often use $Y \in \mathbb{R}^{nT}$ to refer to the vector of stacked $y_{i,t}$, use $X \in \mathbb{R}^{nT \times d}$ to refer to the matrix of stacked $x_{i,t}$, and use $\varepsilon \in \mathbb{R}^{nT}$ to refer to the vector of stacked $\varepsilon_{i,t}$. We let $Y_i$, $X_i$, and $\varepsilon_i$ refer to the respective blocks contributed by user $i$. We require the following regularity conditions on the covariate and error distributions.
\begin{assumption}
    \label{assumptions:initialcovariates}(Exponentially mixing design)
    The covariates satisfy   $\Gamma_i^{-\frac{1}{2}}x_{i,t} \sim \text{subG}(\sigma_x^2)$, where $\mathbb{E}[x_{i,t}x_{i,t}^\tr] = \Gamma_i \in \mathbb{R}^{d \times d}$ with  $d = O(1)$ and $0 < c_{\Gamma} \leq \lambda_{\min}(\Gamma_i) \leq \lambda_{\max} (\Gamma_i) \leq C_\Gamma < \infty$. Furthermore, we assume that $(x_{i,t})_{t\in\mathbb{Z}}$ is strongly mixing $ \forall i \in [n]$ with $\alpha_{x_i}(k) \leq C \exp\{-ck \}$ for constants $C,c>0$. Finally, we assume that $x_i$ is independent of $x_j$ for $i \neq j$.
\end{assumption}
\begin{exmp}\label{example:covariateassumption}
    Let $x_{i,t} = (1; z_{i,t}^\tr   )^\tr  $ for all $i \in [n], t \in [T] $. Assumption \ref{assumptions:initialcovariates} is met if $(z_{i,t})_{t\in\mathbb{Z}}$, for each $i \in [n]$, obeys the causal multivariate ARMA(p,q) process
    $z_{i,t} - m = \sum_{j=0}^\infty \Psi_je_{i, t-j},$
    where $\lambda_{\min}(\sum_{j=0}^\infty\Psi_j \Psi_j^\tr) >0$, and where we have the following conditions on $e_{i,t} \overset{iid}{\sim}\text{subG}(\sigma_x^2)$: they admit a density with respect to the Lebesgue measure, $\mathbb{E}[e_{i,t}] = 0$, $\V(e_{i,t}) =: \Sigma_e$, and $0 < \lambda_{\min}(\Sigma_e) \leq \lambda_{\max}(\Sigma_e)<\infty$. For further details, see Appendix \ref{subsec:covariateassumption}.
\end{exmp}

\begin{assumption}
    \label{assumptions:initialerrors}
    (Strongly mixing errors) For all $i \in [n]$ and $ T\in \mathbb{N}$, the errors satisfy $\mathbb{E}[\varepsilon_i] = 0$ and $\V(\varepsilon_i)=:\Sigma_T \in \mathbb{R}^{T \times T}$ with  $\lambda_{\min}(\Sigma_T)\geq c_\Sigma>0$ where $(\varepsilon_i, X_i)$ is independent of $(\varepsilon_j, X_j)$ for $i\neq j$. Furthermore, $\varepsilon_i \mid X_i \sim \text{subG}(\sigma_\varepsilon^2)$ for $i \in [n]$ where $\sigma_\varepsilon^2$ does not depend on $T$. Finally, $\mathbb{E}[\varepsilon \mid X] = 0$,  both $(\varepsilon_{i,t})_{t \in \mathbb{Z}}$ and  $(x_{i,t} \varepsilon_{i,t})_{t \in \mathbb{Z}}$ are  weakly stationary, and  $(x_{i,t}, \varepsilon_{i,t})_{t \in \mathbb{Z}}$ is strongly mixing with mixing coefficients $\alpha_{\epsilon_i, x_i}(k) \leq Ck^{-3-q}$ for some $q>0$ and $C>0$.
\end{assumption}
In our asymptotic analyses, we consider sequences $(n,T) \to \infty$. Hence, $\varepsilon_i$ implicitly depends on $T$, but we suppress this dependence for simplicity, and in the remainder of this work we usually do the same for $\Sigma_T$, denoting it simply as $\Sigma$. We note that since it is assumed $\varepsilon_i \mid X_i\sim \text{subG}(\sigma_\varepsilon^2)$ for all $T\geq 1$, it follows that $\|\Sigma\| \leq \sigma_\varepsilon^2$ uniformly in $T$. 
\begin{exmp}\label{example:ARassumptions}
    Assumption \ref{assumptions:initialerrors} is met if $(\varepsilon_i)_{i=1}^n$ are   generated, for each $i \in [n]$, from a causal AR(p) process
    $\varepsilon_{i,t} = \sum_{j=0}^\infty \psi_j r_{i,t-j},$
    where $r_{i,t} \overset{iid}{\sim} \text{subG}(\sigma_r^2)$ have distributions absolutely continuous with respect to the Lebesgue measure. In particular,  we have $\sigma_\varepsilon^2 \leq \sigma_r^2 \sum_{j\geq 0} \psi_j^2$. For further details, see Appendix \ref{subsec:exampleARAssumption}.
\end{exmp}

\subsection{uDP Regression estimation}\label{subsec:regestimation}

Assumptions \ref{assumptions:initialcovariates} and \ref{assumptions:initialerrors}  ensure that one can construct user specific $\sqrt{T}$-consistent estimates of $\beta$ as 
\begin{equation}
\label{eq:betai}
\hat{\beta}_i := 
    X_i^\dagger Y_i,
\end{equation} 
where $\dagger$ denotes the Moore-Penrose pseudoinverse, which has the property $A^\dagger = (A^\tr A)^\dagger A^\tr$ for any real matrix $A \in \mathbb{R}^{m\times p}$ and $B^\dagger = B^{-1}$ for invertible matrices $B \in \mathbb{R}^{m \times m}$. Hence, when $\text{Rank}(X_i) = d$, we have $\hat{\beta}_i = (X_i^\tr X_i)^{-1} X_i^\tr Y_i$, the usual local least squares estimator. Furthermore, the estimator in \eqref{eq:betai} has the property that it is the unique minimum $\ell_2$-norm least squares solution regardless of the rank of $X_i$.  Since \textsc{DPTrimMean} takes numerical inputs, we use the Moore-Penrose local least squares solution so as not to reveal rank deficiency in case $\text{Rank}(X_i) < d$. Using the local estimates of \eqref{eq:betai},  ultimately the center can construct a distributed $\sqrt{nT}$-consistent, but non-private, estimate with the user-level sample mean:
\begin{equation}\label{eq:betahat}
    \bar{\beta} := \frac{1}{n}\sum_{i=1}^n \hat{\beta}_i.
\end{equation}
Thus, a natural user-level private analog of $\bar{\beta}$, and that which we employ in this work, is one that instead \emph{privately} computes the user-level sample mean of the local user regression estimates. 
We take a moment to reflect on what makes this a good candidate estimator, and why in many settings it may be preferred over a DP analog of the aggregated least squares estimator. In item-level DP linear regression, it is commonplace to separately calculate private summary statistics 
$
\widehat{(X^\tr X)}_{\mathsf{DP}}$ and $\widehat{(X^\tr Y)}_{\mathsf{DP}}$ 
and then compose these to obtain
$
\hat{\beta}_{\mathsf{SSDP}} := (\widehat{(X^\tr X)}_{\mathsf{DP}})^{-1} \widehat{(X^\tr Y)}_{\mathsf{DP}}$, 
see for example \cite{wang2018revisiting, sheffet2019old,alabi2020differentially}.
However, uniform concentration around a common mean - which, as evident from Theorem \ref{thm:trimest}, is desirable in uDP estimation - is unlikely to hold for the summary statistics of the regression coefficient. This can make the averaging estimator a more appealing candidate for privatization. Figure \ref{fig:sumstatill} serves as a simple example. 
\begin{figure}[t]
\centering
\begin{minipage}[c]{0.4\linewidth}
  \centering
  \includegraphics[width=\linewidth]{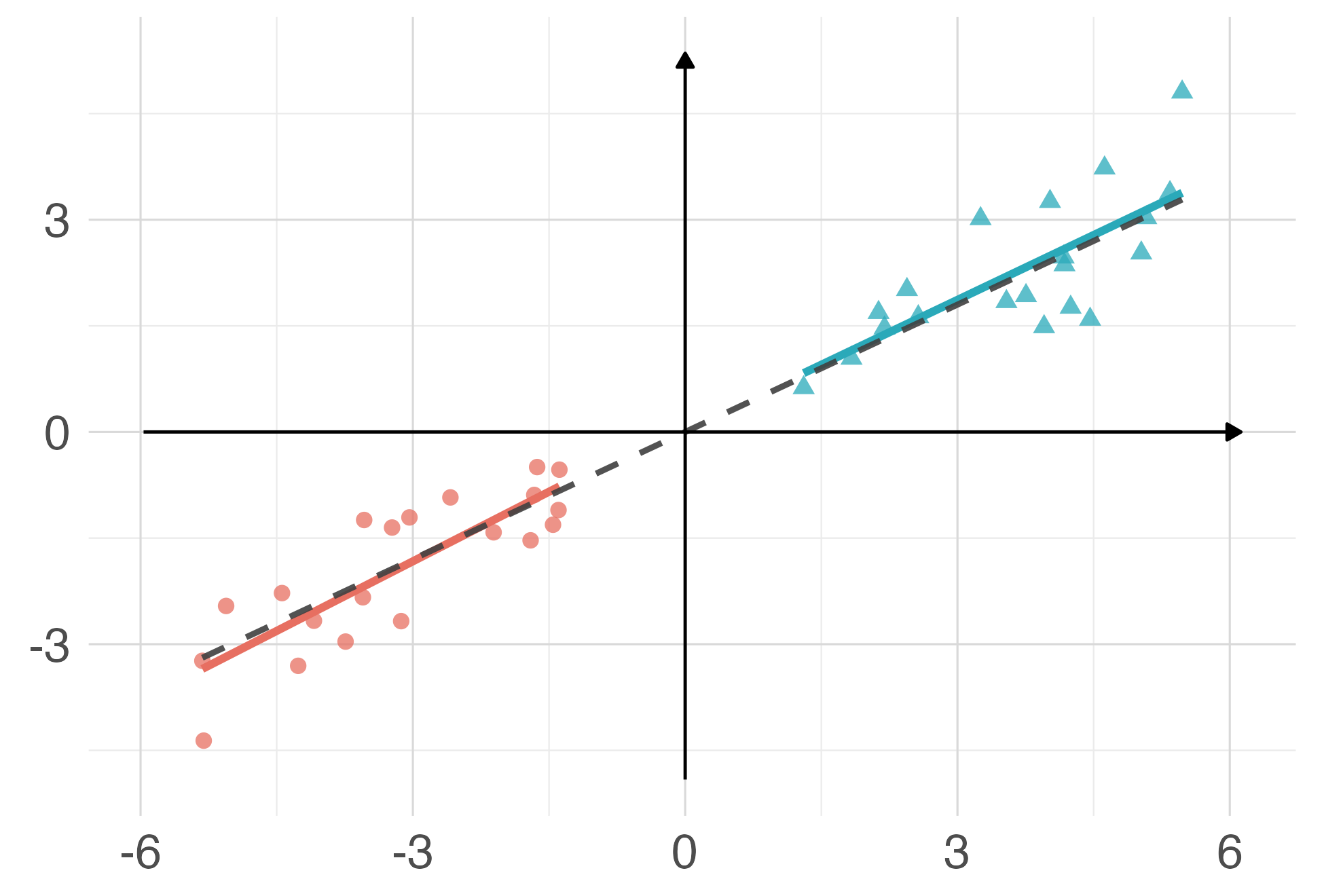}
\end{minipage}\hspace{0\linewidth}
\begin{minipage}[c]{0.58\linewidth}
\small
  \caption{{\small The two users' Gram matrices do not concentrate uniformly. However, their local regression fits concentrate tightly around the common regression coefficients.}}
  \label{fig:sumstatill}
\end{minipage}
\end{figure}
Furthermore, due to the linear nature of the linear regression problem, and unlike in most M-estimation problems, each local estimate $\hat{\beta}_i$ is an unbiased estimate when $X_i$ is full rank, making averaging local estimates a valid alternative as there is no  $O(\frac{1}{T})$ bias term that persists after averaging.

We now turn towards privatizing our candidate estimate \eqref{eq:betahat} by applying the techniques introduced in Section \ref{sec:dpmean}. In particular, we have the following result concerning the concentration of these estimates around the true regression coefficient $\beta$, which are the same as in the iid case as derived in \cite[Theorem 11]{hsu:kakade:zhang2012}, up to a slightly inflated requirement on $T$. As in \cite{hsu:kakade:zhang2012}, the proof relies on writing the difference $\hat{\beta}_i - \beta$ as a ``bounded'' linear transformation of the sub-Gaussian  $\varepsilon_i$ on a high-probability event that the covariates' empirical second moment matrices are well behaved. The main difference with their result is that, in our work, the covariates are temporally dependent and we must control many Gram matrices simultaneously.  The proof of the lemma is in Appendix \ref{app:betasconcentration}.
\begin{lemma}\label{lem:betasconcentration}
    Admit Assumptions \ref{assumptions:initialcovariates}-\ref{assumptions:initialerrors}, and let $\hat{\beta}_i$ be as in~\eqref{eq:betai}and $\bar{\beta}$ as in~\eqref{eq:betahat}. Then, with probability at least $1-\xi$,
    \begin{align*}
    \max_{i \in [n]}\|\hat{\beta}_i - \beta\| &= O\Big(\sigma_\varepsilon\max_{i \in [n]}\|\Gamma_i^{-\frac{1}{2}}\|\sqrt{\frac{1}{T}\big(d + \log \frac{n}{\xi}\big)} \Big), \\
    \|\bar{\beta} - \beta\| &= O\Big(\sigma_\varepsilon  \sqrt{\frac{1}{nT} \big(\frac{1}{n} \sum_{i=1}^n\|\Gamma_i^{-\frac{1}{2}}\|^2\big)\big(d + \log \frac{1}{\xi}\big)} \Big),
    \end{align*}
    when $T \geq C(d +\log \frac{nT}{\xi} )^2$, with $C>0$ a constant that depends on  $O(1)$ population parameters.
\end{lemma}
%
%
Hence, our preceding discussion and the one in Section \ref{sec:dpmean} suggest the following private estimator of $\beta$. First, each user $i \in [n]$ estimates their local regression coefficients, $\hat{\beta}_i$. Then, we provide $(\hat{\beta}_i)_{i=1}^n$ as input to Algorithm \ref{alg:trimestTemp}, and denote the resulting output as $\hat{\beta}_{\mathsf{DP}}$. By construction of Algorithm \ref{alg:trimestTemp}, this estimate will thus satisfy $\mu$-uGDP.
Before we state our utility guarantees for the estimator $\hat{\beta}_{\mathsf{DP}}$, we now introduce some quantities that will serve as thresholds for our results to hold, where in the following $B_{\min}$ and $R_{\min}$ relate to the inputs $B$ and $R$ of Algorithm \ref{alg:trimestTemp}. For some constants $C_1, C_2, C_3, C_4>0$ that may depend on $O(1)$ population parameters, and $g(n, T, \xi, d):=C_2 (d + \log \frac{nT}{\xi})^2$:
\begin{equation}
\begin{split}
\label{eq:min_vals}
&n_{\min}(R, \mu, \xi, d) := C_1\frac{1}{\mu} \sqrt{R\Big(\log \frac{R}{\xi} + d\Big)},  \quad  T_{\min}(n, T, \xi, d) :=\min \{T : T\geq g(n, T, \xi, d)\}, \\
&R_{\min}(T,B) := C_3 \log {B\sqrt{T 
}}
, \quad \text{ and } \quad B_{\min}(\beta, n,T, \xi, d) := \|\beta\| + {C_4}\sqrt{\frac{1}{T
}\big(d + \log \frac{n}{\xi}\big)}.
\end{split}
\end{equation}
In the following, we usually suppress the dependence of these functions on their arguments. 

These ingredients suffice to provide high-probability error bounds on our resulting private estimator of $\beta$. However, to derive its asymptotic distribution, we require that the users' Gram matrices are uniformly well behaved as formalized next.
\begin{assumption}
\label{assumption:UI}
    For a given $k\geq1$, there exist a $\delta>0$ and $T_0 > d$ such that $X_i^\tr X_i$ is invertible  almost surely for all $i \geq 1$ and
    $\sup_{i \geq 1, T\geq T_0} \mathbb{E}[\|(\frac{1}{T} X_i^\tr X_i)^{-1}\|^{k+ \delta}] <\infty.$
\end{assumption}
An easily verifiable sufficient condition under which Assumption \ref{assumption:UI} holds is presented in the following lemma. The proof can be found in Appendix \ref{app:suffUI}.

\begin{lemma}\label{lem:suffUI}
    Assumption \ref{assumption:UI} is satisfied if, for any $v \in \mathbb{S}^{d-1}$, denoting $w_{i,t} := v^\tr x_{i,t}$, the density $f_{W_i}$ of $W_i=(w_{i,1}, \dots w_{i,T})$ is such  that
    $
    \sup_{i\geq 1, w\in \mathbb{R}^T}f_{W_i}(w) \leq K^T
    $
    for some constant $K>0$ that is independent of $n$, $T$ and $v$, and we take $T_0 \geq 2(d + k + \delta')$ for any $\delta'>0$. 
\end{lemma}

Intuitively, this sufficient condition simply requires that the covariates have a density that does not explode too quickly as $T$ grows, as is the case for example under very strong dependence. Below we provide a simple example that satisfies this condition.

\begin{exmp}\label{example:UI}
    If $(x_{i,t})_{t\geq 1}$ are generated by a causal multivariate Gaussian AR(q) 
    and the innovations have covariance matrix $\Sigma_z \succ \mathbf{0}$, then the process satisfies the condition of Lemma \ref{lem:suffUI}; hence, Assumption \ref{assumption:UI}. For further details, see Appendix \ref{subsec:UIexample}.
\end{exmp}

Finally, our results will make use of the following growth condition for asymptotic normality of the estimator:
\begin{equation}\label{eq:growthconditions}
    \frac{1}{n\mu^2}\Big( R \log \frac{R}{\xi} \Big) \Big(\log \frac{n}{\xi} \Big) = o(1), \quad \text{ and} \quad \xi = o(1).
\end{equation}
We are now ready to state the privacy and statistical guarantees of our uDP regression estimator, presented in Theorem~\ref{thm:currentconv}.   We remark that the high probability error bound in Theorem~\ref{thm:currentconv}(b) follows directly from the deterministic guarantees of Algorithm \ref{alg:trimestTemp}, presented in Theorem \ref{thm:trimest}, combined with the concentrations statements of Lemma \ref{lem:betasconcentration}.
The proof is given in Appendix \ref{app:thm:currentconv}.

\begin{theorem}
\label{thm:currentconv}
Let $\hat{\beta}_i$ be as in \eqref{eq:betai} and let $\hat{\beta}_{\mathsf{DP}} = \hyperref[alg:trimestTemp]{\textsc{DPTrimMean}}((\hat{\beta}_i)_{i=1}^n, \mu, B, R, \xi/2)$.

\textbf{(a) Privacy.} 
$\hat{\beta}_{\mathsf{DP}}$ satisfies $\mu$-uGDP.

\textbf{(b) Error Bound.}
Under Assumptions \ref{assumptions:initialcovariates}-\ref{assumptions:initialerrors} and if $B \geq B_{\min}$, $R \geq R_{\min}$, $n \geq n_{\min}$, and $T \geq T_{\min}$ using the values from \eqref{eq:min_vals}, with probability at least $1-\xi$,
\[
\|\hat{\beta}_{\mathsf{DP}} - \beta\|
= O\Big(  \sigma_\varepsilon 
\Big[\sqrt{\frac{\frac{1}{n}\sum_{i=1}^n \|\Gamma_i^{-\frac{1}{2}}\|^2 (d + \log \tfrac{1}{\xi})}{nT} } 
+ \max_{i \in [n]} \|\Gamma_i^{-\frac{1}{2}}\|\sqrt{\frac{(R\log\frac{R}{\xi} + d )(d + \log \frac{n}{\xi})}{n^2\mu^2T}} 
 \Big] \Big).
\]

\textbf{(c) Asymptotic Normality.}
With $\Lambda_i := \sum_{\ell=-\infty}^\infty \mathbb{E}[x_{i,0} x_{i,\ell}^\tr\varepsilon_{i,0}\varepsilon_{i,\ell}]$, assume the existence of $V:=\lim_{n \to \infty} \frac{1}{n} \sum_{i=1}^n \Gamma_i^{-1} \Lambda_i \Gamma_i^{-1}$. If the assumptions in (b) continue to hold under \eqref{eq:growthconditions}, and we admit Assumption \ref{assumption:UI} with $k=2$, then
$
        \sqrt{nT}(\hat{\beta}_{\mathsf{DP}} - \beta) \overset{d}{\to} N(0,  V)$ as $(n,T) \to \infty.
$
 
\end{theorem}
{We remark that the first term on the right side of the error bound in Theorem \ref{thm:currentconv}(b) is the usual statistical error rate for estimating $\beta$ with $\frac{1}{n}\sum_{i=1}^n \hat\beta_i$, whereas  the second term is the privacy error, or the ``cost of privacy". Under \eqref{eq:growthconditions}, this cost of privacy becomes negligible as $n \to \infty$, no matter the relative growth of $n$ and $T$.}

Furthermore, in the result of Theorem \ref{thm:currentconv}(c), notice that if all users' covariates have the same first and second order moments, i.e.\ $\Gamma_i = \Gamma$ and $\Lambda_i = \Lambda$ for common matrices $\Gamma$, $\Lambda$, then the asymptotic covariance satisfies $V = \Gamma^{-1} \Lambda \Gamma^{-1}$. This equals the asymptotic covariance of the usual ordinary least squares estimator $\hat{\beta}_{\mathsf{OLS}}$, and therefore Theorem \ref{thm:currentconv} implies that $\hat{\beta}_{\mathsf{DP}}$ does not asymptotically lose in efficiency compared to the non-private OLS estimator in this case. As a direct consequence, in settings where the alternative estimator $\hat{\beta}_{\mathsf{SSDP}}$ has privacy noise decaying with $T$, our estimator does as well asymptotically.

\subsection{uDP Two-group regression estimation}\label{subsec:twogroupreg}
Besides the linear regression model introduced in (\ref{eq:basemodel}), a model of interest is one that also incorporates time-invariant covariates. Since each user only has one observation per time-invariant variable, it is not obvious that the effect of this variable can be user-level privately estimated in such a way that the privacy error decays with the panel length $T$. 
In the following, we focus on the simple case in which two groups (e.g.\ control and treatment) have a different intercept and/or slope: 
\begin{equation}\label{eq:diffcoeffs}
    y_{i,t} = (\beta  + \theta z_i)^\tr   x_{i,t}  + \varepsilon_{i,t},
\end{equation}
where $z_i \in \{0, 1\}$ indicates group membership. We note that $\theta$ could in principle be sparse, so that if $x_{i,t}$ has an intercept term, model (\ref{eq:diffcoeffs}) includes as as special case the model $y_{i,t} = \beta^\tr x_{i,t} + \theta z_i + \varepsilon_{i,t}$, for $\theta \in \mathbb{R}$.
To facilitate our analysis, we assume that  $n_1 := \sum_{i=1}^nz_i$ and $n_0:=n-n_1$ are of the same order.
\begin{assumption}
\label{assumption:ndiffreq}
    We have that $n_0, n_1 = \Omega(n)$ and $\frac{n_1}{n}\overset{}{\to} p \in (0,1)$ as $n\to \infty$.
\end{assumption}
If the group membership variable $z_i$ were \textit{public}, then one could estimate $\theta$ by applying Algorithm \ref{alg:trimestTemp} to the users in Groups $1$ and $2$ separately, and then differencing the respective estimates. If the group membership variable $z_i$ is itself sensitive, then this approach fails to provide privacy guarantees. Algorithm \ref{alg:diffmeanest} modifies Algorithm \ref{alg:trimestTemp} to privately differentiate between users with $z_i = 1$ and $z_i = 0$, and applying this algorithm separately to the users in the two groups yields the following privacy and utility guarantees.
\begin{theorem}
\label{thm:diffconv}
Let $\hat{\beta}_i$ be as in~\eqref{eq:betai} and let  $\hat{\beta}_{1, \mathsf{DP}} = \hyperref[alg:trimestTemp]{\textsc{DPTreatMean}}((D_i, z_i)_{i=1}^n,\mu/\sqrt{2},B,R,\xi/4)$  and $\hat{\beta}_{0, \mathsf{DP}} = \hyperref[alg:trimestTemp]{\textsc{DPTreatMean}}((D_i, 1-z_i)_{i=1}^n,\mu/\sqrt{2},B,R,\xi/4)$.
Define $\hat{\theta}_{\mathsf{DP}} := \hat{\beta}_{1, \mathsf{DP}} - \hat{\beta}_{0, \mathsf{DP}}$.

\textbf{(a) Privacy.} 
$\hat{\theta}_{\mathsf{DP}}$ satisfies $\mu$-uGDP.

\textbf{(b) Error Bound.}
Under Assumptions \ref{assumptions:initialcovariates}-\ref{assumptions:initialerrors}, \ref{assumption:ndiffreq}, and if $B \geq B_{\min} + \|\theta\|$, $R \geq R_{\min}$, $n \geq n_{\min}$, and $T\geq T_{\min}$,
then with probability at least $1-\xi$, letting $M_j := \frac{\sum_{i: z_i = j}\|\Gamma_i^{-\frac{1}{2}}\|^2}{\sum_{i:z_i = j} 1}$,
\begin{align*}
&\|\hat{\theta}_{\mathsf{DP}} - \theta\|
= O\Big(  \sigma_\varepsilon\Big[\sqrt{\frac{(M_0 + M_1)(d + \log \frac{1}{\xi})}{nT}} +  \max_{1 \leq i \leq n}\|\Gamma_i^{-\frac{1}{2}}\| \sqrt{\frac{(R\log\frac{R}{\xi} + d )(d + \log \frac{n}{\xi})}{n^2\mu^2T}} \Big]\Big).
\end{align*}

\textbf{(c) Asymptotic Normality.}
If in addition we admit Assumption \ref{assumption:UI} with $k = 2$ and condition \eqref{eq:growthconditions}, then if $V_1=\lim_{n\to \infty} \frac{1}{p^2n} \sum_{i:z_i = 1} \Gamma_i^{-1} \Lambda_i \Gamma_i^{-1}$ and $V_0=\lim_{n\to \infty} \frac{1}{(1-p)^2n} \sum_{i:z_i = 0} \Gamma_i^{-1} \Lambda_i \Gamma_i^{-1}$ exist, we have
$
\sqrt{nT}( \hat{\theta}_{\mathsf{DP}} - \theta ) \overset{d}{\to} N(0,V_1+V_0  )$ as  $(n,T) \to \infty.
$
\end{theorem}
The proof is in Appendix \ref{app:thm:diffconv}.
We remark that this estimator achieves a privacy error, i.e., the second term in the error bound of Theorem~\ref{thm:diffconv}(b), decaying with $T$, while this is not possible when privatizing summary statistics to construct a private estimate of $\beta$. To see this, consider the simple case $y_{i,t} = \beta^\tr x_{i,t} + \theta z_{i} + \varepsilon_{i,t}$, for which $\frac{1}{nT} X^\tr X = \begin{bmatrix}
    \frac{1}{nT} \sum_{i,t} x_{i,t} x_{i,t}^\tr & \frac{1}{nT} \sum_{i,t} z_i x_{i,t} \\
  \frac{1}{nT}\sum_{i,t} z_i x_{i,t}^\tr & \frac{1}{n} \sum_{i=1}^n z_i
\end{bmatrix}$. 
It is clear that only the top left block's statistic can give rise to the uniform concentration necessary to attain a privacy error decaying with $T$. 
\section{Differentially Private HAC inference}
\label{sec:inf}
Inference in linear regression problems with observations across time usually requires the estimation of a Heteroskedasticity and Autocorrelation Consistent (HAC) covariance matrix, due to the probable intertemporal dependence. In longitudinal regression, this is often resolved by using the so-called cluster robust  covariance matrix estimator, which estimates the asymptotic covariance of $\hat{\beta}_{\mathsf{OLS}}$ as $Q^{-1} \hat{W} Q^{-1}$, with $Q = \frac{1}{n} \sum_{i=1}^n X_i^\tr X_i$ and $\hat{W} = \frac{1}{n} \sum_{i=1}^n X_i^\tr \hat{\varepsilon}_i \hat{\varepsilon}_i^\tr X_i$ \citep{liang1986longitudinal, arellano1987computing, wooldridge2003cluster}. It has been shown in \cite{hansen2007asymptotic} that this estimator is $\sqrt{n}$-consistent in the $(n,T) \to \infty$ regime. Unfortunately, this estimate of the covariance is not directly applicable in our setting because (i) our estimator, $\hat{\beta}_{\mathsf{DP}}$, has a different asymptotic covariance (see Theorem \ref{thm:currentconv}) and (ii) we cannot readily control the sensitivity of $\frac{1}{T} X_i^\tr X_i$ and $\frac{1}{T} X_i^\tr \hat{\varepsilon}_i \hat{\varepsilon}_i^\tr X_i$, which is a critical step for private noise calibration. 

Hence, we propose a simple private estimator of the asymptotic covariance of our private point estimators that is also HAC and $\sqrt{n}$-consistent in our setting. Rather than calculating a sandwich-type estimator, we propose a private analog of the sample covariance estimator across $\hat{\beta}_i$, that is, a sample average over $(\hat{\beta}_i - \hat{\beta}_{\mathsf{DP}})(\hat{\beta}_i - \hat{\beta}_{\mathsf{DP}})^\tr \approx (X_i^\tr X_i)^{-1} X_i^\tr \varepsilon_i \varepsilon_i^\tr X_i (X_i^\tr X_i)^{-1}$. Thus, our estimator takes an average across users to sidestep heteroskedasticity and autocorrelation issues. This is similar in spirit to the cluster robust standard errors, but additionally takes into account dispersion across $\hat{\beta}_i$. 

\begin{algorithm}
\small
\caption{uGDP Covariance matrix estimation}
\begin{algorithmic}[1]
\label{alg:covestimation1}
\STATE \textbf{Input:} User-level local estimates $\hat{\beta}_i$, the output  $(\hat{\beta}_{\mathsf{DP}}, \hat{m}_{r^*-1}, r^*, B^*)$ of Algorithm \ref{alg:trimestTemp}, as well as the Algorithm \ref{alg:trimestTemp}  quantities $n_{LB}$ and initial radius $B$, and covariance privacy parameter $\mu_{\mathsf{var}}$.
\STATE Let $\delta_{\mathsf{DP}} \gets \|\hat{\beta}_{\mathsf{DP}} - \hat{m}_{r^*-1}\|$, $\kappa \gets \frac{B}{2^{r^*}} + \delta_{\mathsf{DP}}$,  $S \gets \{i \in [n]: \|\hat{\beta}_i - \hat{m}_{r^*-1}\| \leq \frac{B}{2^{r^*}}\}$, $n_{\tilde{S}} \gets \max(|S|, n_{LB})$\label{alg:linewithS}
\STATE Let 
        \[
        \tilde{V}(\hat{\beta}_{\mathsf{DP}})_{\mathsf{DP}} \gets \frac{1}{n_{\tilde{S}}^2}\sum_{i \in S }(\hat{\beta}_i - \hat{\beta}_{\mathsf{DP}})(\hat{\beta}_i - \hat{\beta}_{\mathsf{DP}})^\tr  
        + (B^*)^2I
        + W,
        \]
        where $W_{i,i} \overset{iid}{\sim} N(0, \frac{16\kappa^4}{n_{LB}^4\mu_{\mathsf{var}}^2} )$ and  $W_{i,j} \overset{iid}{\sim} N(0,\frac{8\kappa^4}{n_{LB}^4\mu_{\mathsf{var}}^2} )$ for $j >i$ with $W_{i,j} = W_{j,i}$ for $j <i$.  \label{alg:cov_est}
\STATE Let $\hat{V}(\hat{\beta}_{\mathsf{DP}})_{\mathsf{DP}}$ be the projection of $\tilde{V}(\hat{\beta}_{\mathsf{DP}})_{\mathsf{DP}}$ onto the cone $\mathcal{M}_{symm}^{d \times d}$. 
\STATE \textbf{Output:} $\hat{V}(\hat{\beta}_{\mathsf{DP}})_{\mathsf{DP}}$
\end{algorithmic}
\end{algorithm}

In particular, we apply  Algorithm \ref{alg:covestimation1} to the problem of estimating the covariance of $\hat{\beta}_{\mathsf{DP}}$. We note that our covariance estimator in Line~\ref{alg:cov_est} of Algorithm~\ref{alg:covestimation1} includes a small sample correction term $(B^*)^2I$ as explained in the next subsection, where $B^*$ is the output radius of Algorithm \ref{alg:trimestTemp}.
Since Algorithm \ref{alg:covestimation1} is essentially a private sample covariance estimator of the local estimates of the users that were included in the estimate of $\hat{\beta}_{\mathsf{DP}}$, the algorithm requires one to have already privately estimated $\hat{\beta}_{\mathsf{DP}}$.  

\begin{theorem}
\label{thm:covariancerobust}
Let $\hat{\beta}_{\mathsf{DP}}$ be estimated with $\mu_{\mathsf{est}}$-uGDP. Let $(\hat{\beta}_{i})_{i=1}^n$ and $\hat{\beta}_{\mathsf{DP}}$ be the input to Algorithm \ref{alg:covestimation1} with privacy level $\mu_{\mathsf{var}}$ and denote the resulting estimate by $\hat{V}(\hat{\beta}_{\mathsf{DP}})_{\mathsf{DP}}$.

\textbf{(a) Privacy.} 
$(\hat{\beta}_{\mathsf{DP}}, \hat{V}(\hat{\beta}_{\mathsf{DP}})_{\mathsf{DP}})$ satisfies $\sqrt{\mu_{\mathsf{est}}^2 + \mu_{\mathsf{var}}^2}$-uGDP.

\textbf{(b) Convergence.}
If 
all conditions of Theorem \ref{thm:currentconv} are met,
then $ (\hat{V}(\hat{\beta}_{\mathsf{DP}})_{\mathsf{DP}})^{-\frac{1}{2}} (\hat{\beta}_{\mathsf{DP}} - \beta) \overset{d}{\to} N(0,I_d)$ as $(n,T) \to \infty$.
\end{theorem}

The combination of Theorems \ref{thm:currentconv} and \ref{thm:covariancerobust} allows us to perform asymptotically valid inference about $\beta$. Since we have private estimators that are asymptotically normal and private estimators of their asymptotic variance, we have access to private asymptotic pivots. By post-processing, we can therefore release multiple confidence intervals based on these private approximate pivots without incurring further privacy losses. The following corollary formalizes the guarantees of private confidence intervals constructed using this approach.

\begin{corollary}\label{corollaryCI}
    Admit the conditions of Theorem \ref{thm:covariancerobust}. Then, as $(n,T) \to \infty$
    \[\mathbb{P}\Big((\hat{\beta}_{\mathsf{DP}})_j - z_{1-\alpha/2}\sqrt{(\hat{V}(\hat{\beta}_{\mathsf{DP}})_{\mathsf{DP}})_{j,j}} \leq \beta_j \leq   (\hat{\beta}_{\mathsf{DP}})_j + z_{1-\alpha/2}\sqrt{(\hat{V}(\hat{\beta}_{\mathsf{DP}})_{\mathsf{DP}})_{j,j}} \Big) \overset{p}{\to} 1-\alpha,
    \]
     and these confidence intervals satisfy $\sqrt{\mu_{\mathsf{est}}^2 + \mu_{\mathsf{var}}^2}$-uGDP for all $\alpha\in(0,1)$. 
\end{corollary}

Moreover, one is often interested in testing hypotheses of the form
$
H_0 : R\beta = r,
$
where $R \in \mathbb{R}^{q \times d}$ is a rank $q$ matrix and $q$ is the number of linearly independent restrictions imposed on $\beta$ under $H_0$. Since our estimate is asymptotically normal and we have a consistent estimate of the asymptotic covariance matrix, we can plug these into a Wald test to obtain a private hypothesis test. This is formalized in the following corollary.
\begin{corollary}\label{corollaryTest}
     Admit the conditions of Theorem \ref{thm:covariancerobust}. Then, under $H_0$, as $(n,T) \to \infty$
     $$\hat{W}_{\mathsf{DP}}:= (R\hat{\beta}_{\mathsf{DP}} -r)^\tr(R \hat{V}(\hat{\beta}_{\mathsf{DP}})_{\mathsf{DP}} R^\tr)^{-1} (R\hat{\beta}_{\mathsf{DP}} - r) \overset{d}{\to} \chi_q^2,$$
      and these tests satisfy $\sqrt{\mu_{\mathsf{est}}^2 + \mu_{\mathsf{var}}^2}$-uGDP for all $R \in \mathbb{R}^{q \times d}$ and $r \in \mathbb{R}^q$.
\end{corollary}

\subsection{Small sample correction }
In practice it is useful to consider a finite sample correction similar to the one introduced in \cite{avellamedina:bradshaw:loh2023}.  Indeed, the ``base'' of the estimator is an estimator that with high probability takes the form
\begin{align*}
\hat{V} &:= \frac{1}{|S|^2}\sum_{i\in S} (\hat{\beta}_{i} - \hat{\beta}_{\mathsf{DP}})(\hat{\beta}_{i} - \hat{\beta}_{\mathsf{DP}})^\tr = \frac{1}{|S|^2}\sum_{ i \in S}(\hat{\beta}_{i} - \bar{\beta}_{S})(\hat{\beta}_{i} - \bar{\beta}_{S})^\tr + \frac{(B^*)^2}{|S| }Z Z^\tr,
\end{align*}
where $S$ contains the set of users whose data was used in computing $\hat{\beta}_{\mathsf{DP}}$, i.e.\ those who were not trimmed. This has expectation $
\mathbb{E}[\hat{V}_{}] \approx \V(\bar{\beta}_{S}) + \frac{(B^*)^2}{|S| }I$. On the other hand, with high probability, $\hat{\beta}_{\mathsf{DP}} = \frac{1}{|S|}\sum_{i \in S} \hat{\beta}_{i} + B^* Z$,
with $Z \sim N(0,I_d)$, such that $\V(\hat{\beta}_{\mathsf{DP}}) \approx \V(\bar{\beta}_{S}) + (B^*)^2I$,
and hence we conclude that 
$
\V(\hat{\beta}_{\mathsf{DP}}) - \mathbb{E}[\hat{V}_{}] \approx (1 - \frac{1}{|S|})(B^*)^2I \approx (B^*)^2 I.
$
As such, we employ this quantity as a bias correction term in our covariance estimation procedure. 

\subsection{Inference in the two-group model}
With a minor change, it is possible to extend this algorithm and its guarantees to perform inference in the two-group model of  \eqref{eq:diffcoeffs}. Indeed, one may separately estimate $\hat{\beta}_{0, \mathsf{DP}}$ and $\hat{\beta}_{1,\mathsf{DP}}$ as in Section \ref{subsec:twogroupreg}, and separately apply these to Algorithm \ref{alg:covestimationGrouped}, a slightly modified version of Algorithm \ref{alg:covestimation1}, to obtain private covariance estimators $\hat{V}(\hat{\beta}_{0,\mathsf{DP}})_{\mathsf{DP}}$ and $\hat{V}(\hat{\beta}_{1,\mathsf{DP}})_{\mathsf{DP}}$. By independence across users, and hence groups, $\hat{V}(\hat{\theta}_{\mathsf{DP}})_{\mathsf{DP}} := \hat{V}(\hat{\beta}_{0,\mathsf{DP}})_{\mathsf{DP}} + \hat{V}(\hat{\beta}_{1,\mathsf{DP}})_{\mathsf{DP}}$  yields a consistent estimate under the conditions of Theorem \ref{thm:diffconv}. For a precise theorem statement, we refer the reader to Appendix \ref{Appendix:TwoGroupCov}.

\subsection{Discussion}
In the event of iid errors and equal covariate second moments, it is simple to use Algorithm \ref{alg:trimestTemp} to privately estimate $\frac{1}{T}X^\tr X$ and $\sigma^2$, providing a $\sqrt{nT}$-consistent estimate of the asymptotic covariance of $\hat{\beta}_{\mathsf{DP}}$. More interestingly, it could be possible to estimate a private analog of the HAC estimation method of \cite{andrews1991heteroskedasticity} with the tools and mechanics used throughout our paper. We conjecture this could lead to a $\sqrt{nT}$-consistent estimator for the covariance matrix under appropriate regularity conditions. A formal exploration of this problems goes beyond the scope of this paper and is left for future work. 

Moreover, we remark that our tail assumptions in Assumptions \ref{assumptions:initialcovariates} and \ref{assumptions:initialerrors} can be relaxed in several ways, at the cost of weaker finite sample error bounds. In particular, the errors are only required to have finite moments up to a certain degree, so that even polynomial tails can be permissible. 
Similarly, the tails of the covariate distribution can be instead assumed to be sub-Weibull throughout as in, for example, 
\cite{wong2020lasso}, but can not be weakened beyond that without developing new concentration results, since our control of the Gram matrix relies on the concentration results in \cite{merlevede:peligrad:rio2011} for exponentially strongly mixing sequence of sub-Weibulls.

\section{Simulations}
\label{sec:simulations}
\subsection{ARMA Estimation}
We consider the model
$
y_{i,t} = \beta^\tr x_{i,t} + \varepsilon_{i,t},
$
with $d = 4$, where we draw $\beta_j \overset{iid}{\sim} \text{Unif}(-20, 20)$. In this model $(x_{i,t}, \varepsilon_{i,t})_{t=1}^T$ are cross-sectionally iid, and we generate $m_i \overset{iid}{\sim} N(0, 9I)$ and $x_{i,t} = m_i + 0.5 (x_{i,t-1} - m_i) + \eta_{i,t}$ with $\eta_{i,t} \overset{iid}{\sim} N(0, I)$. Finally, the errors follow the model $\varepsilon_{i,t} = 0.5\varepsilon_{i,t-1} +0.5 r_{i,t-1}+ r_{i,t} $, where $r_{i,t} \overset{iid}{\sim} N(0, 1)$.
We investigate the performance of various estimators of $\beta$ in this model, namely: the DP estimator $\hat{\beta}_{\mathsf{DP}}$ as in Theorem \ref{thm:currentconv} with $\mu = 1$-uGDP, the DP estimator $\hat{\beta}_{\mathsf{DP}}$ with $\mu=\infty$-uGDP (i.e., no privacy), and the OLS estimator $\hat{\beta}_{\mathsf{OLS}}=(X^\tr X)^{-1} X^\tr Y$. Table \ref{table:RMSE_one_neat} displays the empirical scaled RMSE, $\sqrt{nT}\|\hat{\beta} - \beta\|$, of these estimators across 1000 independent replications, and Figure \ref{fig:rmse_comparison} displays the empirical \textit{unscaled} RMSE, i.e.\ $\|\hat{\beta} - \beta\|$, across these replications. We used $B = 100$, $R = 10$, and $\xi = 10 ^{-5}$ in all instantiations of Algorithm \ref{alg:trimestTemp}. 
\begin{table}[ht]
\centering
\caption{Empirical value of $\sqrt{nT}\|\hat{\beta}-\beta\|$ for various estimators of $\beta$.}
\label{table:RMSE_one_neat}
\begingroup
\setlength{\tabcolsep}{10pt}
\renewcommand{\baselinestretch}{1}\selectfont
\begin{tabular}{l ccc ccc ccc}
\toprule
& \multicolumn{3}{c}{DP-OLS ($\mu=1$)}
& \multicolumn{3}{c}{DP-OLS ($\mu=\infty$)}
& \multicolumn{3}{c}{OLS ($\mu = \infty$)} \\
\cmidrule(lr){2-4}\cmidrule(lr){5-7}\cmidrule(lr){8-10}
$n\backslash T$ & 10 & 40 & 160 & 10 & 40 & 160 & 10 & 40 & 160 \\
\midrule
300  & 18.91 & 21.65 & 24.03 & 14.49 & 17.19 & 18.93 & 12.16 & 15.65 & 16.33 \\
600  & 15.79 & 19.13 & 20.93 & 14.20 & 17.20 & 18.77 & 12.16 & 15.16 & 16.45 \\
1200 & 14.72 & 18.15 & 19.94 & 14.27 & 17.39 & 18.81 & 11.89 & 15.27 & 16.16 \\
2400 & 14.61 & 17.63 & 19.30 & 14.53 & 17.06 & 18.66 & 11.94 & 15.28 & 16.22 \\
\bottomrule
\end{tabular}
\endgroup
\end{table}

Recall that by Theorem \ref{thm:currentconv}, 
$\hat{\beta}_{\mathsf{DP}}$ 
satisfies the following error bound, for $T$ sufficiently large:
\[
\sqrt{nT}\|\hat{\beta}_{\mathsf{DP}} - \beta\| \lesssim C_{\mathsf{stat}} + \frac{C_{\mathsf{priv}}}{\sqrt{n}}, \quad 
\]
where $C_\mathsf{stat}$ captures the usual normalized statistical error of estimating $\beta$ using $\frac{1}{n}\sum_{i=1}^n \hat{\beta}_i$, whereas $\frac{1}{\sqrt{n}} C_\mathsf{priv}$ represents the normalized cost of privacy.
In particular, the cost of privacy term $\frac{1}{\sqrt{n}} C_\mathsf{priv}$ captures the excess error of taking  privacy level $\mu=1$ with respect to taking $\mu = \infty$ in Theorem \ref{thm:currentconv}. This aligns with our findings in Table \ref{table:RMSE_one_neat} and Figure \ref{fig:rmse_comparison}, where we observe that the RMSE of $\hat{\beta}_{\mathsf{DP}}$ with $\mu = 1$ converges to the RMSE of $\hat{\beta}_{\mathsf{DP}}$ with $\mu = \infty$ as $n$ increases, regardless of the value of $T$. Moreover, we observe that in all methods - private and non-private - the gain from increasing $n$ is larger than the gain of increasing $T$. This, perhaps, can be explained by the fact that increasing $T$ introduces new correlated data, whereas increasing $n$ introduces a new independent sequence of data, and the latter is more beneficial in finite samples. 
\begin{figure}[ht]
    \centering
    \begin{subfigure}[b]{0.32\textwidth}
        \centering
        \includegraphics[width=\textwidth]{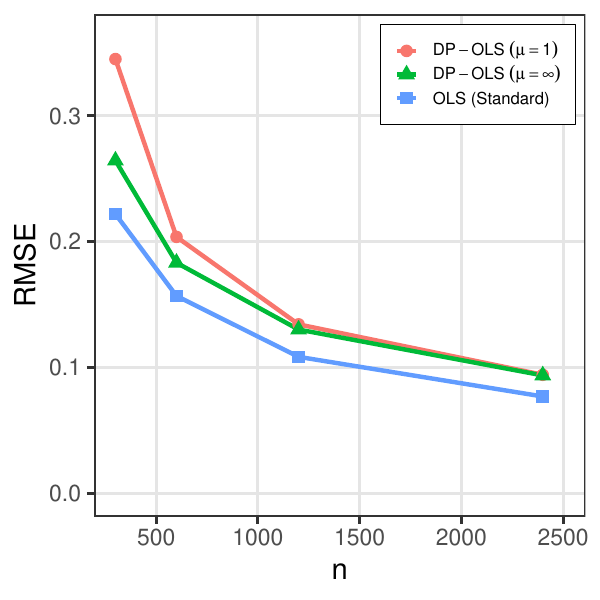}
        \caption{$T = 10$}
        \label{fig:T10}
    \end{subfigure}
    \hfill
    \begin{subfigure}[b]{0.32\textwidth}
        \centering
        \includegraphics[width=\textwidth]{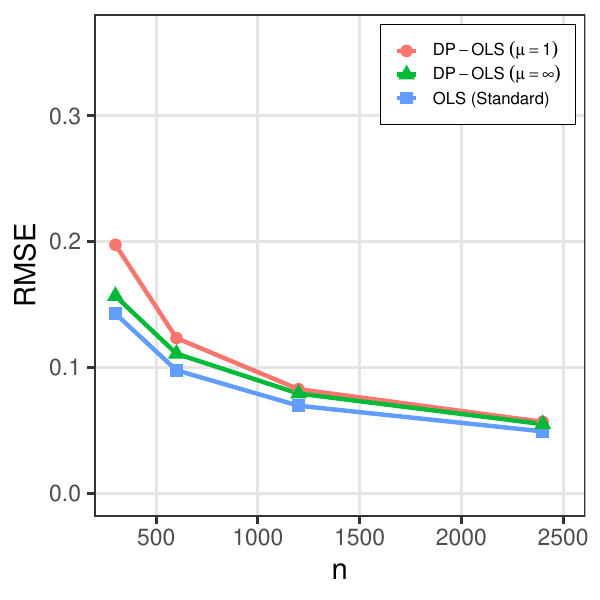}
        \caption{$T = 40$}
        \label{fig:T40}
    \end{subfigure}
    \hfill
    \begin{subfigure}[b]{0.32\textwidth}
        \centering
        \includegraphics[width=\textwidth]{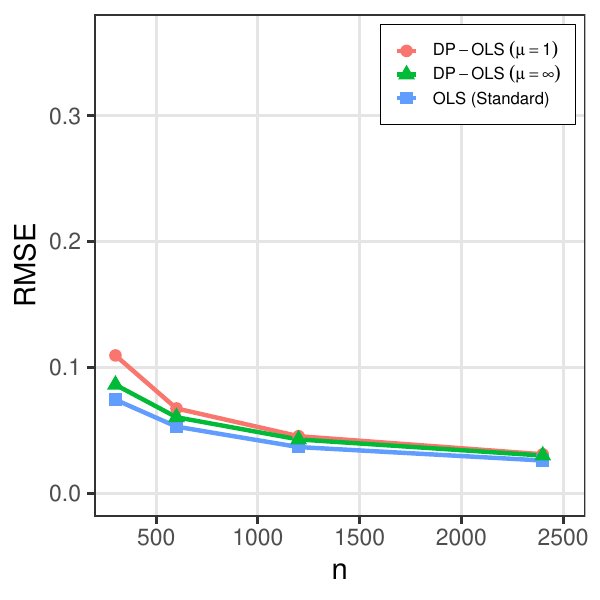}
        \caption{$T = 160$}
        \label{fig:T160}
    \end{subfigure}
    \caption{The unscaled RMSE for various $(n,T)$ combinations.}
    \label{fig:rmse_comparison}
\end{figure}

\subsection{ARMA Inference}
We take the model of the previous subsection, but now fix $T = 15$ and also consider $n = 4800$, and we assess the empirical coverage and width of the confidence intervals
$
(\hat{\beta}_j -  z_{0.975}\hat{V}_{j,j}, \hat{\beta}_j +  z_{0.975}\hat{V}_{j,j})
$ using the cluster covariance estimator as described in Algorithm \ref{alg:covestimation1}, which by Corollary \ref{corollaryCI} asymptotically achieve $95\%$ coverage.
We evaluate intervals based on Theorem \ref{thm:covariancerobust} that satisfy $\sqrt{2}$-uGDP (by taking $\mu_\mathsf{est} = \mu_\mathsf{var} = 1$) as well as $\infty$-uGDP (i.e., no privacy at all) As before, we take $B = 100$, $R = 10$ and $\xi = 10^{-5}$ in estimating $\hat{\beta}_{\mathsf{DP}}$. We compare these to the coverage levels achieved by the non-private pooled OLS estimator $\hat{\beta}_{\mathsf{OLS}} = (X^\tr X)^{-1} X^\tr Y$ with non-HAC covariance estimate $\hat{\sigma}^2 (X^\tr X)^{-1}$, which we call ``classical'', and based on cluster robust standard errors (e.g.,\ \cite{arellano1987computing}), which we call ``clustered''.   Table \ref{table:coverage_width} displays the coverage levels achieved by all four sets of estimators, as well as the average width of the constructed confidence intervals in parentheses.

The table is based on 10,000 replications per configuration, so that with $d = 4$ there are 40,000 intervals per configuration. The private intervals achieve a $95\%$ coverage rate for $n=4800$, at almost the same width as the non-private analog and not much wider than the confidence intervals constructed with the pooled OLS cluster robust standard errors.

\begin{table}[H]
\centering
\caption{Empirical coverage (and width) of the ARMA $95\%$ confidence intervals. 
\textit{Note: Widths are reported in parentheses below the coverage rates.}}
\label{table:coverage_width}
\begingroup
\renewcommand{\baselinestretch}{1}\selectfont
\begin{tabular}[t]{c c c c c}
\toprule
 & \multicolumn{2}{c}{DP-OLS} & \multicolumn{2}{c}{Pooled OLS} \\
\cmidrule(r){2-3} \cmidrule(l){4-5}
n & $\sqrt{2}$-uGDP & $\infty$-uGDP & Clustered & Classical \\
\midrule
\multirow{2}{*}{300} & 0.895 & 0.951 & 0.947 & 0.526 \\
                     & (0.524) & (0.456) & (0.413) & (0.151) \\
\addlinespace[0.2em] 
\multirow{2}{*}{600} & 0.936 & 0.949 & 0.948 & 0.525 \\
                     & (0.375) & (0.323) & (0.293) & (0.107) \\
\addlinespace[0.2em]
\multirow{2}{*}{1200}& 0.947 & 0.949 & 0.951 & 0.524 \\
                     & (0.246) & (0.228) & (0.207) & (0.075) \\
\addlinespace[0.2em]
\multirow{2}{*}{2400}& 0.949 & 0.950 & 0.950 & 0.525 \\
                     & (0.167) & (0.162) & (0.147) & (0.053) \\
\addlinespace[0.2em]
\multirow{2}{*}{4800}& 0.950 & 0.950 & 0.950 & 0.520 \\
                     & (0.116) & (0.114) & (0.104) & (0.038) \\
\bottomrule
\end{tabular}
\endgroup
\end{table}

\section{Data example}
\label{sec:real_data}
The data example we consider is the anonymized teaching version of the Childhood Asthma Management Program (CAMP) dataset \citep{childhood2000long}. The dataset contains lung capacity measurements for $n = 695$ children with asthma over the course of several doctor visits. All children were treated with albenol in this period. Moreover, roughly $2/5$ received no additional medication, and the remaining $3/5$ received some other experimental medication. 
The primary measurement is the Forced Expiratory Volume at one second (FEV). A question of interest is how this develops over time given administration of treatment. 
The first regression we consider is, then,
\begin{equation}
\label{DataExample:eq}
\mathsf{FEV}_{i,t} = \alpha + \alpha_i +  \beta \text{days}_{i,t} + \varepsilon_{i,t},
\end{equation}
where $\mathsf{FEV}_{i,t}$ measures child $i$'s FEV at the $t^{th}$ visit, and days$_{i,t}$ measures the number of days that have elapsed since the start of child $i$'s treatment at the time of their $t^{th}$ visit. The primary statistic of interest is $\beta$. We estimate this using two estimators. The first estimator  is a \textit{non-private} one-way Random Effects  estimator under the assumption that $\alpha_i$ are mean-zero iid with $\mathbb{E}[\alpha_i^2] = \sigma_\alpha^2$,  where the model is fit using the `plm' package in R \citep{croissant2008panel} with  the demeaning parameter obtained via the method of \cite{swamy1972exact}. For this estimator, we calculated the standard errors using the cluster robust covariance estimator as in \cite{arellano1987computing}. The second estimator is $\hat{\beta}_{\mathsf{DP}}$ as in Theorem \ref{thm:currentconv}, estimated using $\mu = 1$, $B = 1$, $R = 15$, and $\xi = 10^{-6}$. We use the estimator of Theorem \ref{thm:covariancerobust} with $\mu_{\mathsf{var}} = 1$ to obtain standard errors, so that $(\hat{\beta}_{\mathsf{DP}}, \hat{V}(\hat{\beta}_{\mathsf{DP}})_{\mathsf{DP}})$ satisfies $\sqrt{2}$-uGDP. Table~\ref{tab:DataExampleReg} contains the point estimate and standard error obtained with the two methods. Since the private estimate is necessarily random, the DP entries contain the \textit{interquartile range of the estimates across 1000 repetitions}. Both the non-private and all 1000 private z-scores strongly reject $H_0: \beta = 0$.
\begin{table}[ht]
\begingroup
\renewcommand{\baselinestretch}{1}\selectfont
\centering
\caption{Estimated regression coefficient for \eqref{DataExample:eq}} 
\label{tab:DataExampleReg}
\begin{tabular}{l c c}
\toprule
Estimator    & Estimated Coefficient  & SE    \\
\midrule
{$\hat{\beta}_{\mathsf{RE}}$} & {$0.000640$} & {$8.687 \cdot 10^{-6}$} \\ 
$\hat{\beta}_{\mathsf{DP}}$  & $[0.000639,   \; 0.000645]$ & $[8.192 \cdot 10^{-6}, \; 8.638 \cdot 10^{-6}]$  \\ 
\bottomrule
\end{tabular}
\endgroup
\end{table}
Note furthermore that this is an unbalanced panel with patients receiving between zero and seventeen follow-up visits, with eight users having fewer than three datapoints. Explicitly accounting for such an imbalanced panel could lead to greater efficiency. 

We proceed with the primary question of interest in this study, which is whether additional medication has an effect on FEV. The relevant regression equation in this case is
\begin{equation}
\label{eq:differenceregression}
\mathsf{FEV}_{i,t} = \alpha + z_i \rho + \alpha_i+  (\beta + z_i\theta)\text{days}_{i,t} + \varepsilon_{i,t} =: \varphi^\tr x_{i,t} + \varepsilon_{i,t},
\end{equation}
where $z_{i}$ equals $1$ if patient $i$ received additional medication and $0$ if patient $i$ was part of the control group. The quantity of interest is $\theta$, which measures the increase in slope coefficient for the patients that received additional medication.
Our strategy will be to privately estimate the slope  among the users in the treatment group (corresponding to $\beta + \theta$) using Algorithm \ref{alg:diffmeanest}, doing the same for the users in the control group, and then taking differences to obtain an estimate for $\theta$.  We similarly run Algorithm \ref{alg:covestimationGrouped} twice and sum to obtain an estimate of the variance. Since we run Algorithm \ref{alg:diffmeanest} and \ref{alg:covestimationGrouped} twice  with $\frac{1}{\sqrt{2}}$-uGDP, the final privacy guarantee of our results is $\sqrt{2}$-uGDP. Furthermore,  our analyses are performed on small samples, with the placebo group having $n_1 =275$, and the non-placebo group having $n_0 = 420$, so we use a lower $R = 10$. The result of this analysis may be found in Table \ref{tab:DataExampleDiff}. We note that $\hat{\theta}_\mathsf{RE}$ is not statistically significant at the $\alpha = 0.05$ level, having a p-value of $0.248$, whereas the private $z$-test rejected $H_0: \theta = 0$ in 65 out of 1000 repetitions. 
\begin{table}[ht]
\begingroup
\renewcommand{\baselinestretch}{1}\selectfont
\centering
\caption{Estimated slope difference coefficient for \eqref{eq:differenceregression}}
\label{tab:DataExampleDiff}
\begin{tabular}{l c c}
\toprule
Estimator    & Estimated Coefficient &  SE     \\
\midrule
{$\hat{\theta}_{\mathsf{RE}}$} & {$-2.043 \cdot 10^{-5}$} &{ $1.770 \cdot 10^{-5}$} \\
$\hat{\theta}_{\mathsf{DP}}$  & $[-2.634 \cdot 10^{-6}, \;5.119 \cdot 10^{-5}]$ & $[4.095 \cdot 10^{-5}, \; 4.453 \cdot 10^{-5}]$  \\
\bottomrule
\end{tabular}
\endgroup
\end{table}

\subsection*{Privatizing the summary statistics}
Another approach to estimating \eqref{eq:differenceregression} is to calculate a uDP estimate of $X^\tr X$ and of $X^\tr Y$, and to compose into a linear regression estimator. In particular, for each user, let 
$$x_{i,t,1} = \text{days}_{i,t} - \overline{\text{days}}_i, \qquad x_{i,t,2} = z_i x_{i,t,1}, \qquad \text{ and } \qquad y_{i,t} = \mathsf{FEV}_{i,t} - \overline{\mathsf{FEV}}_i,$$
so that this yields a DP analog of the Fixed Effects estimator. We estimate $(\widehat{X^\tr X})_\mathsf{DP}$ and $(\widehat{X^\tr Y})_\mathsf{DP}$ using Algorithm \ref{alg:trimestTemp}, using the upper triangular part of $X_i^\tr X_i$ to obtain input in $\mathbb{R}^3$, and transforming the DP output into a matrix in $\mathbb{R}^{2 \times 2}$. For both quantities, we set the initial radius $B$ in Algorithm \ref{alg:trimestTemp} equal to the smallest possible $B$ such that the relation $\max_{i =1, \dots, 695} \|D_i\| \leq B$ holds for our dataset, and $R = 1$ (for various choices of $R$, no radius shrank more than once since there is no concentration around a common quantity for any variable), as well as $\mu = 1$ and $\xi = 10^{-5}$, to produce a final estimate $\hat{\varphi}_{\mathsf{SSOptimistic}} := (\widehat{X^\tr X})_{\mathsf{DP}}^{-1}\widehat{(X^\tr Y)}_{\mathsf{DP}}$. This estimator is not really DP as the calibration of $B$ and $R$ was done using the smallest possible values in the observed data as opposed to global quantities. The estimates across $1000$ repetitions may be found in Table~\ref{tab:DataExampleSumStats}.

\begin{table}[ht]
\begingroup
\renewcommand{\baselinestretch}{1}\selectfont
\centering
\caption{Estimated coefficients for \eqref{eq:differenceregression} using $\hat{\phi}_{\mathsf{SSOptimistic}}$}
\label{tab:DataExampleSumStats}
\begin{tabular}{l c}
\toprule
Coefficient    & Estimated Coefficient Interquartile Range   \\
\midrule
$\hat{\beta}_{\mathsf{SSOptimistic}}$  & [$0.000636$, $0.000672$]  \\
$\hat{\theta}_{\mathsf{SSOptimistic}}$  & [$-8.041 \cdot 10^{-5}$, $-4.869 \cdot 10^{-6}$]  \\
\bottomrule
\end{tabular}
\endgroup
\end{table}
Table \ref{tab:DataExampleSumStats} shows that even the optimistic estimator based on adding noise to the summary statistics produces noisier estimates than the private estimator based on aggregating local averages, whose estimated values are displayed in Tables \ref{tab:DataExampleReg} and \ref{tab:DataExampleDiff}. We remark that $T$ in this dataset is small, and that as discussed in Section \ref{subsec:regestimation}, the private estimator based on aggregating local averages has the advantage that the privacy error decays in $T$. Hence, we expect $\hat{\beta}_{DP}$ to compare more favorably to $\hat{\beta}_{\mathsf{SSOptimistic}}$ as the panel length increases, or as the panel becomes more balanced. 

\section{Discussion}

There is now a large literature establishing connections between robust statistics and differential privacy in order to develop new private algorithms.  See 
\citep{dwork:lei2009,avellamedina:brunel2020,avellamedina2020,avellamedina2021,liu:kong:oh2022, ramsay:jagannath:chenouri2022, avellamedina:bradshaw:loh2023,li:berrett:yu2023,hopkins:kamath:majid:narayanan2023,alabi:kothari:tankala:venkat:zhang2023,chhor:sentenac2023,yu:zhao:zhou2024} just to name a few. We note that while our core algorithms build on DP ideas of \cite{karwa:vadhan2018,biswas:dong:kamath:ullman2020}, the construction  intuitively implies some degree of robustness and bears some resemblance to the ideas advanced by \cite{gervini:yohai2002} for robust and efficient estimation. It would be interesting to formally characterize  the robustness properties of our methods. 

In addition, we remark on the efficiency of our linear regression estimator. Several works \citep{dobriban2021distributed, rosenblatt2016optimality} have investigated the efficiency of averaging estimators in distributed settings. \cite{dobriban2021distributed} show that whilst there is no loss of asymptotic efficiency in our setting if $\Gamma_i = \Gamma \; \forall i =1, \dots, n$, this is not true if there is covariate heterogeneity. Moreover, their results show that generally the simple averaging estimator (with equal weights, as in our work) is asymptotically less efficient than the optimally weighted averaging linear regression estimator. As each optimal weight is a parameter specific to a user, implementing such a weighting scheme differentially privately is non-trivial but work in this direction could result in more efficient private estimators.   

Furthermore, while our work addresses efficient estimation and inference in linear regression problems under uDP, it would be natural to extend this to generalized linear models and M-estimators. These  fundamental models have received some attention in the simpler setting of item-level DP \citep{lei2011differentially,cai2020cost, avellamedina2020, avellamedina2021, arora2022differentially,  avellamedina:bradshaw:loh2023}, with many works also addressing DP empirical risk minimization and stochastic convex optimization prediction tasks \citep{chaudhuri2011differentially, kifer2012private,bassily2014private, bassily2019private}. Recent papers have also addressed some of these problems in the uDP setting \citep{levy:sun:amin:kale:kulesza:mohri:suresh2021, bassily2023user, ghazietal2023, liu2024user}, but they did not consider inference, nor dependent data. 

Similarly, there is a large amount of work in the non-private longitudinal regression literature that we cannot incorporate trivially in our private setting. This includes testing for the existence and form of serial correlation \citep{baltagi1995testing, hong2004wavelet, baltagi2007testing}, cross-sectional correlation \citep{frees1995assessing, ng2006testing, moscone2009review, baltagi2016testing}, slope homogeneity \citep{pesaran2008testing}, and the existence of a unit root \citep{unit1999, unit2002, unit2007, unit2010}, among many other relevant procedures. Working out DP counterparts of all these methods would be of great practical relevance. 

Finally, in settings where the error covariance matrix $\Sigma$ can be consistently estimated, standard regression theory suggests that the estimators proposed in this work can be asymptotically improved upon by incorporating FGLS mechanics. We have obtained promising initial results in this direction, which will be the subject of a separate paper. 
\par
{
\begingroup
 \renewcommand{\baselinestretch}{1.1}\selectfont
\bibliographystyle{apalike}
\bibliography{references} 
\endgroup
}

\newpage
\appendix

\section{Auxiliary Results}
\label{app:aux_results}
\begingroup
\renewcommand{\baselinestretch}{1}\selectfont

In Section \ref{app:Aux_A}, we first review some well-known concentration inequalities and a couple of linear algebra lemmas that will be useful in the proofs of the main results. Section \ref{app:Aux_B} presents concentration inequalities under strong mixing conditions that constitute the main technical tools underpinning our theoretical analysis.

\subsection{Basic inequalities and concentration}
\label{app:Aux_A}
\begin{lemma}\cite[Proposition 2.5]{wainwright2019}\label{suplemma:sumSubGaussians} Let $X_i$ for $i \in [n]$ be independent $\sigma^2_i$-sub-Gaussian random variables with mean $m_i$. Then, for all $q\geq 0$,
\[
\mathbb{P}\Big(\sum_{i=1}^n (X_i - m_i) \geq q\Big) \leq \exp \Big(-\frac{q^2}{2\sum_{i=1}^n \sigma_i^2}\Big) \quad \text{
and } \quad 
\mathbb{P}\Big(\sum_{i=1}^n (X_i - m_i) \leq -q\Big) \leq \exp \Big(-\frac{q^2}{2\sum_{i=1}^n \sigma_i^2}\Big).
\]
\end{lemma}

\begin{lemma}\citep[Example 2.1 \& (2.9)]{wainwright2019}\label{suplemma:max_subG}
    Let $Z_r \overset{iid}{\sim} subG(1)$ with $\mathbb{E}[Z_r] = 0$ for $r \in [R]$. Then, for any $0 < \xi \leq 1$,
    \[
    \mathbb{P}\Big(\max_{r \in [R]} |Z_r| \leq \sqrt{2\log\frac{R}{\xi}}\Big) \geq 1 -  2\xi.
    \]
\end{lemma}

\begin{lemma}\cite[Example 2.28]{wainwright2019}\label{suplemma:maxnormals}
    Let $N_r \overset{iid}{\sim} N(0, \mathbb{I}_d)$ for $r \in [R]$. Then, for any $0 < \xi \leq1$,
    \[
    \mathbb{P}\Big( \max_{r \in [R]}\|N_r\| \leq \sqrt{d} + \sqrt{2 \log \frac{R}{\xi}}\Big) \geq 1 - \xi.
    \]
\end{lemma}

\begin{lemma}\label{lemma:epsilonNets}
    Define a function $g: [0,1] \mapsto \mathbb{R}_+$. Let $Z \in \mathbb{R}^d$ be a random vector such that for all $u \in \mathbb{S}^{d-1}$ and $\xi \in [0,1]$ we have $ \mathbb{P}(|u^\tr Z| > g(\xi)) \leq \xi$, then
    \[
   \mathbb{P}\left(\|Z\| > \frac{4}{3} g\Big(\frac{\xi}{9^d}\Big)\right) \leq \xi. 
    \]
    Similarly, if $W \in \mathbb{R}^{d \times d}$ is a symmetric, positive semi-definite random matrix such that for all $u \in \mathbb{S}^{d-1}$ and $\xi \in [0,1]$, we have $\mathbb{P}(|u^\tr W u | > g(\xi)) \leq \xi$, then
   \[
\mathbb{P}\left(\|W\| > 2g\Big(\frac{\xi}{9^d}\Big)\right) \leq \xi.
    \]
\end{lemma}
\begin{proof}
    This follows by a standard $\epsilon$-net argument. Namely, by \cite{vershynin2018high}'s Corollary 4.2.13, there exists a set $\mathcal{N}$ such that $\sup_{x \in \mathbb{S}^{d-1}, y \in \mathcal{N}}\|x -y\| \leq \frac{1}{4}$ and $|\mathcal{N}| \leq 9^d$.  Now, let $v = {Z}/{\|Z\|}$ and $v_0 \in \mathcal{N}$ be such that $\|v - v_0\| \leq \frac{1}{4}$. Then $|v_0^\tr Z| \geq  \sup_{v\in\mathbb{S}^{d-1}}|v^\tr Z| - |(v - v_0)^\tr Z|  \geq (1 - \frac{1}{4})\|Z\| $, so that $\sup_{y \in \mathcal{N}}|y^\tr Z| \geq \frac{3}{4} \|Z\|$. Hence,
    \begin{align*}
        \mathbb{P}\Big(\|Z\| > \frac{4}{3}g(\xi)\Big) \leq \mathbb{P}\Big(\sup_{y \in \mathcal{N}} |y^\tr Z| > g(\xi)\Big)\leq |\mathcal{N}| \, \xi,
    \end{align*}
    so that by taking $\xi' = {\xi}/{9^d}$, we obtain
    \[
     \mathbb{P}\Big(\|Z\| > \frac{4}{3}g(\xi')\Big)  =  \mathbb{P}\Big(\|Z\| > \frac{4}{3}g\Big(\frac{\xi}{9^d}\Big)\Big) \leq |\mathcal{N}|\, \frac{\xi}{9^d} \leq \xi.
    \]
    The proof in the matrix case is similar, noting that $\|W\| = \sup_{u \in \mathbb{S}^{d-1}} u^\tr W u$ and that for two unit vectors $u, u_0$ such that $\| u - u_0\| \leq \frac{1}{4}$, we have $u^\tr W u = u_0^\tr Wu_0 + (u - u_0)^\tr W u_0 + u^\tr W(u - u_0)$, and hence $|u^\tr W u - u_0 ^\tr W u_0| \leq \frac{1}{2}\|W\|$.
\end{proof}

\begin{lemma}\cite[Exercise 7.3.5]{vershynin2018high}\label{auxres:GOEnorm}
    Let $W \in \mathbb{R}^{d \times d}$ be a random matrix such that $W_{i,i} \overset{iid}{\sim} N(0, 2\sigma^2)$ for all $i \in [d]$ and $W_{i,j} \overset{iid}{\sim} N(0, \sigma^2)$ for all $i >j$, with $W_{i,j} = W_{j,i}$ when $i < j$. Then $ \|W\| = O(\sigma\sqrt{d + \log \frac{1}{\xi}})$ with probability at least $1-\xi$.
\end{lemma}

\begin{lemma}
    \label{lem:spectrumbound}
    Let $A$ and $B$ be real symmetric invertible matrices such that $ \|B^{-\frac{1}{2}}AB^{-\frac{1}{2}} - I\| \leq \epsilon <1.$ Then, we have
    $\|B^{\frac{1}{2}}A^{-\frac{1}{2}}\| \leq (1-\epsilon)^{-1/2}.$
\end{lemma}
\begin{proof}
    Note that $\|B^\frac{1}{2}A^{-\frac{1}{2}}\| = \sqrt{\|B^\frac{1}{2} A^{-\frac{1}{2}}\|^2} = \sqrt{\|B^\frac{1}{2}A^{-1} B^\frac{1}{2}\|}$. Furthermore, by assumption, we have that $(1 - \epsilon)I \preceq B^{-\frac{1}{2}}AB^{-\frac{1}{2}}  \preceq (1+\epsilon)I,$ and so taking inverses and accordingly reversing the ordering, we obtain
    \[
    \frac{1}{{1-\epsilon}}I \succeq B^\frac{1}{2}A^{-1}B^\frac{1}{2} \succeq \frac{1}{{1 +\epsilon}}I.
    \]
    Thus, we have
    \[\|B^\frac{1}{2}A^{-\frac{1}{2}}\| = \sqrt{\|B^\frac{1}{2}A^{-1}B^\frac{1}{2}\|} \leq \sqrt{\frac{1}{1-\epsilon}} = \frac{1}{\sqrt{1-\epsilon}}.\]
\end{proof}

\begin{lemma}\label{lem:subGmgfBOUND}
    Let $Z \in \mathbb{R}$ be mean-zero, $\sigma^2$-subGaussian. Then 
    $
    \mathbb{E}[\exp(\frac{Z^2}{6\sigma^2})] \leq 2$.
\end{lemma}
\begin{proof}
    Let $a = \frac{1}{6\sigma^2}$. Then,
    \begin{align*}
    \mathbb{E}[\exp(aZ^2)] &= \int_0^\infty \mathbb{P}(\exp(aZ^2) > u)du = 1 +\int_1^\infty \mathbb{P}(\exp(aZ^2) > u)du  \\&= 1 + \int_0^\infty 2at e^{at^2} \mathbb{P}(|Z| > t)dt \leq 1 + 4a \int_0^\infty t e^{-(\frac{1}{2\sigma^2}  - a)t^2}dt = 1 + \frac{2a}{\frac{1}{2\sigma^2} - a},
    \end{align*}
    through a change of variables $u = e^{at^2}$ and an application of the tail bound of Lemma \ref{suplemma:sumSubGaussians}, where the right side of the tail bound in Lemma \ref{suplemma:sumSubGaussians} incurs an additional factor $2$ since we take a two-tailed bound. Now, plugging in $a = \frac{1}{6\sigma^2}$, we obtain our result $\mathbb{E}[\exp(\frac{Z^2}{6\sigma^2}) ] \leq 2.$
\end{proof}
\begin{lemma}\label{lem:squaresubgaussiansubexp}
    Let the random variables $X, Y \in \mathbb{R}$ be respectively $\sigma_X^2$ and $\sigma_Y^2$-sub-Gaussian, with $|\mathbb{E}[X]| \leq c\sigma_X$ and $|\mathbb{E}[Y] | \leq c\sigma_Y$ for some constant $c>0$, then $\mathbb{P}(|XY - \mathbb{E}[XY] | > q) \leq 2\exp(-\frac{q}{C\sigma_X \sigma_Y})$
    for some constant $C>0$ that only depends on $c$.
\end{lemma}
\begin{proof}
    Let $m_X := \mathbb{E}[X]$, $m_Y := \mathbb{E}[Y]$, $\tilde{X} = X - m_X$, $\tilde{Y} = Y - m_Y$. Then $XY - \mathbb{E}[XY] = A + B + C$, where $A := \tilde{X} \tilde{Y} - \mathbb{E}[\tilde{X} \tilde{Y}]$, $B := m_X \tilde{Y}$, $C = m_Y \tilde{X}$. We first prove the result for $m_X, m_Y \neq 0$. It is straightforward to see that, 
    \begin{align}\label{eq:squaresubGproofA}
    \mathbb{P}(|A| > q/3) &\leq 2e^{-\frac{q}{36\sigma_X \sigma_Y}}, \\
    \label{eq:squaresubGproofB}
    \mathbb{P}(|B| > q/3) &\leq 2e^{-\frac{q}{18|m_X| \sigma_Y}}, \\
    \label{eq:squaresubGproofC}
    \mathbb{P}(|C| > q/3) &\leq 2e^{-\frac{q}{18|m_Y| \sigma_X}}.
    \end{align}
    In fact \eqref{eq:squaresubGproofB}-\eqref{eq:squaresubGproofC} follow simply by Lemma \ref{suplemma:sumSubGaussians}. Thus, we prove \eqref{eq:squaresubGproofA}. We have for any $u,v \in \mathbb{R}$, $2|uv| \leq u^2 + v^2$, so
\[
\Big|\frac{\tilde{X} \tilde{Y}}{\sigma_X\sigma_Y}\Big| \leq \frac{1}{2}\Big(\frac{\tilde{X}^2}{\sigma_X^2} + \frac{\tilde{Y}^2}{\sigma_Y^2}\Big).
\]
Taking expectations and Cauchy-Schwarz we get, using Lemma \ref{lem:subGmgfBOUND}, 
\[
\mathbb{E}\Big[\exp\Big(\frac{|\tilde{X} \tilde{Y}|}{6\sigma_X \sigma_Y}\Big)\Big] \leq 2,
\]
and hence, $\mathbb{P}(|\tilde{X} \tilde{Y}| >q) \leq 2\exp(-\frac{q}{6\sigma_X \sigma_Y})$ for all $q>0$. Moreover, by Cauchy-Schwarz,  $|\mathbb{E}[\tilde{X} \tilde{Y}]| \leq \sigma_X \sigma_Y.$
Now, first consider $q > 2 |\mathbb{E}[\tilde{X}\tilde{Y}]|$. Then $q - |\mathbb{E}[\tilde{X}\tilde{Y}]| \geq q/2$ and
\[
\mathbb{P}(|\tilde{X}\tilde{Y} - \mathbb{E}[\tilde{X} \tilde{Y}]| >q ) \leq \mathbb{P}(|\tilde{X} \tilde{Y}| > q/2) \leq 2\exp(-\frac{q}{12\sigma_X\sigma_Y}).
\]
Now let $0 < q \leq 2|\mathbb{E}[\tilde{X} \tilde{Y}]|$. We have trivially  
\[
2 \exp\Big(-\frac{q}{12\sigma_X \sigma_Y}\Big) \geq 2\exp\Big(-\frac{2|\mathbb{E}[\tilde{X} \tilde{Y}]|}{12\sigma_X \sigma_Y}\Big) \geq 2\exp\Big(-\frac{2\sigma_X \sigma_Y}{12\sigma_X \sigma_Y}\Big) = 2e^{-1/6} \geq 1;
\]
hence, since every probability is also trivially bounded by $1$, we  conclude that for all $q>0$, 
\[
\mathbb{P}(| \tilde{X} \tilde{Y} - \mathbb{E}\tilde{X} \tilde{Y}| > q) \leq 2\exp(-\frac{q}{12\sigma_X \sigma_Y}).
\]
By taking $q/3$, this proves \eqref{eq:squaresubGproofA}. Now, by a union bound over the events in \eqref{eq:squaresubGproofA}-\eqref{eq:squaresubGproofC},
    \begin{align*}
    \mathbb{P}(|XY - \mathbb{E}[XY] | >q ) \leq 2e^{-\frac{q}{36\sigma_X \sigma_Y}} + 2e^{-\frac{q}{18|m_X| \sigma_Y}} + 2e^{-\frac{q}{18|m_Y| \sigma_X}} &\leq 2e^{-\frac{q}{36\sigma_X \sigma_Y}} + 4e^{-\frac{q}{18c\sigma_X \sigma_Y}} \\
    &\leq 6e^{-\frac{q}{c'\sigma_X\sigma_Y}},
    \end{align*}
    where $c' = \max(36, 18c)$.

    We have to obtain tail bounds of the form $\mathbb{P}(|XY - \mathbb{E}[XY] | > q) \leq 2\exp(-\frac{q}{C \sigma_X \sigma_Y}),$ and by our derivations and since probabilities are bounded by one, we thus need to show $\min (1, 6\exp(-\frac{q}{c'\sigma_X\sigma_Y})) \leq 2\exp(-\frac{q}{C \sigma_X \sigma_Y}).$ 
    This holds for $C = c' \frac{\log 6}{ \log 2}$, proving our result for the given $C>0$.

    Note that if either or both of $m_X$, $m_Y$ is zero, then the respectively probability bound (over $B$ or $C$ or both) is trivially zero so that the bound holds also in that case.
\end{proof}

\begin{lemma}\cite[Theorem 2]{von1965inequalities}\label{lemma:vonbahresseen}
    Let $Z_1, \dots, Z_n$ be a sequence of random variables satisfying $\mathbb{E}[Z_{m+1} \mid \sum_{i=1}^m Z_m] = 0$ a.s.\ for all $1 \leq m \leq n-1$, and $\mathbb{E}[|Z_i|^r] < \infty$ for all $i \in [n]$ and some $1 \leq r \leq 2$. Then, $\mathbb{E}[|\sum_{i=1}^n Z_i|^r] \leq  2\sum_{i=1}^n \mathbb{E}[|Z_i|^r].$
\end{lemma}

\subsection{Strong mixing and concentration}
\label{app:Aux_B}

\begin{lemma}\label{lem:measurablemixing}
    If $(z_t)_{t \in \mathbb{Z}}$ is strongly mixing with strong mixing coefficients $\alpha(k)$ and $g: \mathbb{R}^d \mapsto \mathbb{R}^p$ 
    is a measurable function, then $(g(z_t))_{t \in \mathbb{Z}}$ is also strongly mixing with strong mixing coefficients $\alpha_g(k) \leq \alpha(k)$.
\end{lemma}
\begin{proof}
    By Definition \ref{def:strongmixing}, we have
    \[
    \alpha_g(k) = \sup_{t \in \mathbb{Z}} \sup_{A \in \mathcal{G}_{-\infty}^t, B \in \mathcal{G}_{t +k}^{\infty}}  |\mathbb{P}(A \cap B) -\mathbb{P}(A) \mathbb{P}(B)|,
    \]
    where $\mathcal{G}_a^b = \sigma(g(X_a), \dots, g(X_b))$ is the sigma algebra generated by $g(X_a)$ through $g(X_b)$. Elementarily, we have $\mathcal{G}_a^b \subseteq \mathcal{F}_a^b,$
    where $\mathcal{F}_a^b$ is the corresponding sigma algebra generated by $X_a$ through $X_b$. Hence, by definition, we obtain
    \begin{align*}
    \alpha_g(k) &= \sup_{t \in \mathbb{Z}} \Big(\sup_{A \in \mathcal{G}_1^t, B \in \mathcal{G}_{t +k}^\infty} \big|\mathbb{P}(A \cap B) - \mathbb{P}(A) \mathbb{P}(B)\big|\Big) \\&\leq \sup_{t \in \mathbb{Z}}\Big(\sup_{A \in \mathcal{F}_1^t, B \in \mathcal{F}_{t +k}^\infty} \big|\mathbb{P}(A \cap B) - \mathbb{P}(A) \mathbb{P}(B)\big|\Big)   = \alpha(k)
    \end{align*}
\end{proof}

\begin{lemma}\cite[Chapter 1.2, Theorem 3; Davydov's inequality]{doukhan1995}\label{auxres:davydov}
    Let $(X_t)_{t \in \mathbb{Z}} $ be a stochastic process. Then, assuming $\mathbb{E}[\|X_s\|^p]<\infty$ and $\mathbb{E}[\|X_t\|^q]<\infty$ for $s,t \in \mathbb{Z}$,
    \[
    |\mathrm{Cov}(X_s,X_t)| \leq 8 \alpha(|s-t|)^{\frac{1}{r}}(\mathbb{E} [\|X_s\|^p])^{1/p} (\mathbb{E}[\|X_t\|^q])^{1/q},
    \]
    for any $p , q, r \geq 1$  such that $\frac{1}{p} + \frac{1}{q} + \frac{1}{r} = 1$, and where $\alpha(\cdot)$ is the $\alpha$-mixing coefficient of the process $(X_t)_{t \in \mathbb{Z}}$ as defined in Definition \ref{def:strongmixing}.
\end{lemma}

\begin{lemma}\cite[Chapter 1.4, Theorem 2]{doukhan1995}\label{auxres:momentstrongmixing}
    Let $\tau > 2$ and let $\mathcal{C}$ be a finite subset of $\mathbb{N}_+$. Furthermore, let $(Z_t)_{t \in \mathbb{Z}}$ be a sequence of centered strongly mixing random variables with strong mixing coefficients $\alpha(k)$ such that $\exists \epsilon>0$, $\exists c \in 2\mathbb{N}$ with $c \geq \tau$ for which
    $\sum_{r = 1}^\infty (r+1)^{c-2} \alpha(r)^{\epsilon / (c+ \epsilon)} <\infty$
    and $\max_{t \in \mathcal{C}} \mathbb E[|Z_t|^{\tau + \epsilon}]<\infty$. Then,
    \[
    \mathbb{E}|\sum_{t\in \mathcal{C}} Z_t|^\tau \leq C\max \Big(\sum_{t \in \mathcal{C}}(\mathbb{E}|Z_t|^{\tau + \epsilon})^{\tau / (\tau + \epsilon)},  \Big(\sum_{t \in \mathcal{C}}(\mathbb{E}|Z_t|^{2 + \epsilon})^{2 / (2 + \epsilon)} \Big)^{\tau/2} \Big) = O(|\mathcal{C}|^{\tau/2})
    \]
\end{lemma}

\begin{lemma}\cite[Theorem 1]{merlevede:peligrad:rio2011}\label{auxres:merlevederesult}
   Let $(Z_t)_{t \in \mathbb{Z}}$ be a sequence of real, centered random variables that satisfy $\sup_{t>0} \mathbb{P}(|Z_t | > q) \leq \exp\{1 - (q/b)^{\gamma_2} \}$ for some constants $\gamma_2, b>0 $ and all $q >0$, 
   and that furthermore have strong mixing coefficients $\alpha(k)$ that satisfy  $\alpha(k) \leq \exp\{-ck^{\gamma_1} \}$, for some constants $c, \gamma_1>0$.
    Then, denoting
    \[
    V:= \sup_{M>0} \sup_{t >0} \Big(\V\left(\tau_M(Z_t)\right) + 2\sum_{r >t} \big| \textrm{Cov}\left(\tau_M(Z_t), \tau_M(Z_r) \right) \big| \Big),
    \]
    with $\tau_M(x) := \max \{-M, \min(x, M) \}$, we have, for $T \geq 4$ and $\gamma<1$ defined by $1/\gamma = 1/\gamma_1 + 1/\gamma_2$, and positive constants $C_1, C_2, C_3, C_4$ that depend on $c, b, \gamma_1, \gamma_2$, for all $q >0$,
    \begin{align*}
    \mathbb{P}\Big(\sup_{R \leq T}|\sum_{t=1}^R Z_t| > q \Big) &\leq  T \exp \Big\{-\frac{q^\gamma}{C_1}\Big\} + \exp\Big\{-\frac{q^2}{C_2(1 + TV)}\Big\} + \exp\Big\{-\frac{q^2}{C_3 T}\Big(\frac{q^{\gamma(1-\gamma)}}{C_4(\log q)^\gamma} \Big)\Big\}.
    \end{align*}
\end{lemma}

\begin{lemma}\label{lem:expexpconc}
    Let $(Z_t)_{t \in \mathbb{Z}}$ be a sequence of centered strongly mixing random variables, with strong mixing coefficients $\alpha(k) \leq C\exp\{-ck\}$. Furthermore, if, for some $K>0$ and all $q>0$,
    we have $
    \sup_{t }\mathbb{P}(|Z_t| > q) \leq 2e^{-\frac{q}{K}},$
    then for any $T>0$ we have, for some constant $c_K>0$ that does not depend on $T$ or $q$,
    \begin{align*}
    \mathbb{P}\Big(\big|\frac{1}{T}\sum_{t=1}^TZ_t \big| > q\Big) \leq T\exp\Big(\frac{-(Tq)^\frac{1}{2}}{c_K}\Big) + \exp\Big(\frac{-(Tq)^2}{c_K(1 + TK^2)}\Big) + \exp\Big(\frac{-Tq^2}{c_K} \exp\big(\frac{(Tq)^{\frac{1}{4}}}{c_K (\log(Tq))^{\frac{1}{2}}}\big)  \Big).
    \end{align*}
    In particular, with probability at least $1-\xi$, and for a constant $C_K$ that depends on $K$,
    \[
    \Big|\frac{1}{T}\sum_{t=1}^TZ_t\Big| = O\Big(K\sqrt{\frac{1}{T}\log \frac{1}{\xi}}  + \frac{C_K}{T} \big(\log \frac{T}{\xi}\big)^2 \Big).
    \]
\end{lemma}
\begin{proof}
    This is a near immediate application of Lemma \ref{auxres:merlevederesult}, by noting that by  assumption $\gamma_1 = 1$ and $\gamma_2 = 1$ so that $\gamma = \frac{1}{2}$. Indeed,  $\gamma_2 = 1$ follows from
    \[
    \sup_{t>0} \mathbb{P}(|Z_t| >q) \leq 2e^{-\frac{q}{K}} \leq e e^{-\frac{q}{K}} = \exp\Big\{1 - \frac{q}{K}\Big\}.
    \]
    Finally then, to bound their quantity $V$, we have, for the truncation parameter $\tau_M(x) = \max\{\min\{x, M\} ,-M\} $,
    \begin{align*}
    V &:= \sup_{M>0} \sup_{t>0}\Big(\V(\tau_M(Z_t)) + 2\sum_{r>t} |\text{Cov}(\tau_M(Z_t), \tau_M(Z_r))|  \Big)  \\
    &\leq \sup_{t>0}\Big(\V(Z_t) + 2\sum_{r>t} |\text{Cov}(Z_t, Z_r)|  \Big) \leq  
    \sup_{t>0}\Big(\mathbb{E}[Z_t^2] +
    2C\mathbb{E}[Z_t^2] \sum_{k\geq 1}\alpha(k)^\frac{1}{2}\Big)
     = O(K^2) ,
    \end{align*}
    by applying the moment bounds of sub-exponential random variables and using Lemma \ref{auxres:davydov} (Davydov's inequality) to obtain the final inequality and our exponential strong mixing assumption to conclude that $ \sum_{k\geq 1}\alpha(k)^\frac{1}{2} = O(1)$.
\end{proof}
\begin{lemma}\label{auxres:GammaITExpectationBound}
    Let $x_{i,t}$ for $i \geq 1$, $t \in [T]$ be as in Assumption \ref{assumptions:initialcovariates} and let $(\varepsilon_{i,t})$ be as in Assumption \ref{assumptions:initialerrors}. Then, denoting $\Gamma_{i,T} = \frac{1}{T} X_i^\tr X_i$, for $T \geq 3$ and $ 2 \leq p <\infty$, we have that
   $\sup_{i \geq 1} \mathbb{E}[\|\Gamma_{i,T} - \Gamma_i\|^p] = O(T^{-p/2}),$
    and $\sup_{i\geq 1}\mathbb{E}[\|S_{i,T}\|^p] = O(1)$ denoting $S_{i,T} = \frac{1}{\sqrt{T}}X_i^\tr \varepsilon_i$.
\end{lemma}
\begin{proof}
        First notice that $\mathbb E[\|\Gamma_{i,T} - \Gamma_i\|^p] = \int_0^1 Q_i(\xi)^p d\xi,$
    where $Q_i(\xi) := \inf\{x \in \mathbb{R}: F_i(x)  \geq \xi \}$ is the quantile function, if $F_i(x) = \mathbb{P}(\|\Gamma_{i,T} - \Gamma_i\| \leq x)$.
    Thus, we begin by finding the tail probabilities of
    \begin{align*}
    \Big\|\frac{1}{T}\sum_{t=1}^{T} (x_{i,t} x_{i,t}^\tr - \mathbb E[x_{i,t} x_{i,t}^\tr] )\Big\|.
\end{align*}
Notice that $|\mathbb{E}[\frac{u^\tr x_{i,t}}{\sqrt{u^\tr \Gamma_i u}}]| \leq 1 $, and since $x_{i,t}$ is $\sigma_x^2$-sub-Gaussian and $\lambda_{\max}(\Gamma_i)\leq C_\Gamma<\infty$, we can apply Lemma \ref{lem:squaresubgaussiansubexp} to conclude that
\begin{align*}
  \mathbb{P}(| u^\tr x_{i,t} x_{i,t}^\tr u - \mathbb E [u^\tr x_{i,t} x_{i,t}^\tr u]| > q)  &  = \mathbb{P}\Big(\Big|\frac{u^\tr x_{i,t} x_{i,t}u - \mathbb{E}[u^\tr x_{i,t} x_{i,t}^\tr u]}{{u^\tr \Gamma_i u}}\Big| > \frac{q}{{u^\tr \Gamma_i u}} \Big)  \leq 2\exp\Big(\frac{-q}{C\sigma_x^2 C_{\Gamma}} \Big).
\end{align*}
Since $(u^\tr x_{i,t} x_{i,t}^\tr u)_{t \in \mathbb{Z}}$ is exponentially strongly mixing by Lemma \ref{lem:measurablemixing},
by applying Lemma \ref{lem:expexpconc} we obtain that with probability at least $1-\xi$, for any $u \in \mathbb{S}^{d-1}$,
\begin{align*}
w_{h}(u):= \Big|\frac{1}{T} \sum_{t=1}^{T} u^\tr x_{i,t} x_{i,t}^\tr u - u^\tr \mathbb{E} [x_{i,t} x_{i,t}^\tr ]u\Big| = O\Big(\sigma_x^2 C_\Gamma \sqrt{\frac{1}{T} \log \frac{1}{\xi}} + \frac{C_{\sigma_x, \Gamma}}{T}(\log \frac{T}{\xi})^2 \Big).
\end{align*}
Since $\|\frac{1}{T}\sum_{t=1}^{T} (x_{i,t} x_{i,t}^\tr - \Gamma_i)\| = \sup_{u \in \mathbb{S}^{d-1}} w_{h}(u)$,
by Lemma \ref{lemma:epsilonNets} we obtain, with probability at least $1-\xi$,
\begin{align}\label{eq: xitxitconcentration}
\|\Gamma_{i,T} - \Gamma_i \| \leq C\Big( \sigma_x^2 C_{\Gamma} \sqrt{\frac{1}{T}\big(d  + \log \frac{1}{\xi}\big)} + C_{\sigma_x, \Gamma} \frac{(d + \log \frac{T}{\xi})^2}{T}\Big) =:h(\xi),
\end{align}
    for some $C>0$. Hence, since $Q_i(1-\xi) \leq  h(\xi)$, we obtain
    \begin{align*}
    \mathbb{E} [\|\Gamma_{i,T} - \Gamma_i \| ^p] &\leq \int_0^1 h(\xi)^p d\xi \\
    &\leq 2^{p-1}C^p\Big(\int_{0}^1 \Big(\sigma_x^2 C_\Gamma\sqrt{\frac{1}{T}\big(d  + \log \frac{1}{\xi}\big)}\Big)^pd\xi + \int_0^1 \Big(\frac{C_{\sigma_x, \Gamma}}{T}(d + \log\frac{T}{\xi})^2\Big)^pd\xi  \Big),
    \end{align*}
    and for the first integral, substituting $\psi = \log \frac{1}{\xi}$, we obtain 
    \begin{align*}
        \int_{0}^1 \Big(\sigma_x^2 \sqrt{\frac{1}{T}\big(d  + \log \frac{1}{\xi}\big)}\Big)^pd\xi = \sigma_x^{2p} T^{-p/2}\int_0^\infty (d + \psi)^{p/2}e^{-\psi}d\psi  = \sigma_x^{2p}T^{-p/2}e^d\Gamma\Big(1+\frac{p}{2},d\Big),
    \end{align*}
    where $\Gamma(s,r) := \int_r^\infty t^{s-1} e^{-t} dt $ denotes the incomplete Gamma function. Thus, we have that the first term is $O(T^{-p/2})$, while for the second term we have, also substituting $\psi = \log \frac{1}{\xi}$,
    \begin{align*}
        \int_0^1 \Big(C_{\sigma_x, \Gamma}\frac{(d + \log\frac{T}{\xi})^2}{T}\Big)^pd\xi &= C_{\sigma_x,\Gamma}^p T^{-p}\int_0^\infty (d + \log T + \psi)^{2p}e^{-\psi}d\psi  \\&= C_{\sigma_x,\Gamma}^p T^{-p}e^{d + \log T}\Gamma(2p + 1, d + \log T) \leq C_{\sigma_x,\Gamma}^p T^{-p+1}e^{d}\Gamma(2p + 1, d),
    \end{align*}
    since the incomplete Gamma function is monotonically decreasing in its second argument, such that $\Gamma(2p + 1, d + \log T) \leq \Gamma(2p + 1, d)$ for $T \geq 3$. Hence, since none of these terms depend on $i$, for $p \geq 2$, we have that \begin{align}\label{eq:inLemmaGammiITMoment}
    \sup_{i \geq 1}\mathbb{E}[\|\Gamma_{i,T} - \Gamma_i\|^p] = O(T^{-p/2}).
    \end{align}
    Proceeding with the moments of the score $S_{i,T}$, first observe $\mathbb{E} [\|S_{i,T}\|^p] = \mathbb E\left[\mathbb{E}[\|S_{i,T}\|^p \mid X_i]\right]$. 
    We also note that using H\"older's inequality to pass from an $L_2$ to an $L_p$ norm, for $p>2$,  $r = p/2$ and $r' = \frac{p}{p-2}$ such that $\frac{1}{r} + \frac{1}{r'} = 1$, we have that 
    \begin{equation}
    \label{eq:Holder}
       \|v\|^p = \Big(\sum_{j=1}^d |v_j|^2 \cdot 1\Big)^{p/2} \leq \Big(\Big(\sum_{j=1}^d |v_j|^{2r} \Big)^{1/r} \Big(\sum_{j=1}^d 1 \Big)^{1/r'}\Big)^{p/2} =d^{p/2 - 1} \sum_{j=1}^d |v_j|^p,
    \end{equation}
    Now, conditional $X_i$, we have that $S_{i,T}$ is $\|\Gamma_{i,T}\|^\frac{1}{2}\sigma_\varepsilon$-sub-Gaussian, so that, with $(e_j)_{j=1}^d$ denoting the natural basis of $\mathbb{R}^d$, and using \eqref{eq:Holder} we see that
    \begin{align*}
        \mathbb{E}[\|S_{i,T}\|^p \mid X_i] &\leq d^{\frac{p}{2} - 1}\sum_{j = 1}^d \mathbb{E}[ |e_j^\tr S_{i,T}|^p \mid X_i] \\  
        &\leq d^\frac{p}{2} C_p \sigma_\varepsilon^p \|\Gamma_{i,T}\|^{p/2} p^{p/2} \leq d^\frac{p}{2} C_p \sigma_\varepsilon^p \left(\|\Gamma_i\| + \|\Gamma_{i,T} - \Gamma_i\|\right)^\frac{p}{2} p^{p/2} \\
        &\leq d^\frac{p}{2}C_p2^{\frac{p}{2}-1} \sigma_\varepsilon^p p^\frac{p}{2} \left(\|\Gamma_i\|^\frac{p}{2} + \|\Gamma_{i,T} - \Gamma_i\|^\frac{p}{2} \right),
    \end{align*}
    and, therefore, for $d,p = O(1)$, using \eqref{eq:inLemmaGammiITMoment} and the fact that $\|\Gamma_i\| \leq C_\Gamma$ uniformly over $i$,
    \begin{align*}
   \sup_{i\geq 1} \mathbb{E}[\|S_{i,T}\|^p]  \leq d^\frac{p}{2}C_p 2^{\frac{p}{2}-1} \sigma_\varepsilon^p p^\frac{p}{2} \Big(\sup_{i \geq 1}\|\Gamma_i\|^\frac{p}{2} + \sup_{i\geq1}\mathbb{E}[\|\Gamma_{i,T} - \Gamma_i\|^\frac{p}{2}] \Big) = O(1).
    \end{align*}
\end{proof}

\section{Proofs of Section \ref{sec:dpmean}}
\label{appendix:trimestalg}
In this appendix, we provide privacy guarantees for Algorithm~\ref{alg:trimestGeneral}, which is a slight generalization of Algorithm~\ref{alg:trimestTemp}. In particular, we can recover Algorithm~\ref{alg:trimestTemp} from Algorithm~\ref{alg:trimestGeneral} by taking $\mu_1 = \mu_2 = \frac{\mu}{2\sqrt{R}}$ and $\mu_3 = \frac{\mu}{\sqrt{2}}$. We also use this algorithm to derive the utility guarantees of Theorem \ref{thm:trimest}.
\begin{algorithm}
\caption{DP Trimming-Based Mean Estimation}
\label{alg:trimestGeneral}
\begin{algorithmic}[1]
\STATE \textbf{Input:} Data $\{D_i\}_{i=1}^n$, privacy parameters $(\mu_1,\mu_2,\mu_3)$, radius $B$, 
max rounds $R$, failure probability\ $\xi$\
\STATE \textbf{Set:} $\tau \gets n - \frac{1}{\mu_2}{\sqrt{2\log\frac{{4}R}{\xi}}}$ and  $n_{LB} \gets \max\left(2\tau - n,1\right)$ \algc{A h.p.\ lower bound on the \# of users in the final estimate} 
\STATE \textbf{Init:} $\hat{m}_{r-2} \gets 0$, $\hat{m}_{r-1} \gets 0$, $r \gets 0$, $r^* \gets 0$
\WHILE{$r \le R$ and $r^* = 0$}\afterctrlc{Shrink the radius \& check that enough users fall inside the new ball}\label{while1}
    \STATE $\omega_i \gets \mathbbm{1}\{\|D_i-\hat m_{r-1}\|\le B/2^r\}$, \quad
           $\omega \gets \sum_{i=1}^n \omega_i + \tfrac{1}{\mu_2}Z_r$, \quad $Z_r\sim\mathcal{N}(0,1)$ \label{line:ZrAlg}
    \IF{$\omega < \tau$}\label{line:countcheck}
        \afterctrlc{Too many users fall outside the ball: output the radius of the previous iteration}
        \STATE $S \gets \{ i : \|D_i-{\hat m_{r-2}}\| < B/2^{r-1} \}$, \quad $n_{\tilde{S}} \gets \max(|S|, n_{LB})$,  \quad  $\mathcal{C} \gets \frac{\sqrt{(R - r)\mu_1^2 + (R - r)\mu_2^2 + \mu_3^2}}{\mu_3}$
       \STATE $\hat m \gets {\hat m_{r-2}} + \frac{1}{n_{\tilde{S}}}\sum_{i\in S}(D_i-{\hat m_{r-2}}) + \frac{2B}{2^{r-1}\mathcal{C}\mu_3  n_{LB}}N_{*}$, \quad $N_{*}\sim\mathcal{N}(0,I)$  \label{line:earlyterminationm}
               \STATE $r^* \gets r-1$
    \ELSIF{$r=R$}
        \afterctrlc{Final round: output the smallest possible radius}
        \STATE $S \gets \{ i : \|D_i-\hat m_{R-1}\| < B/2^{R} \}$, \quad  $n_{\tilde{S}} \gets \max(|S|, n_{LB})$
        \STATE $\hat m \gets \hat m_{R-1} + \frac{1}{n_{\tilde{S}}}\sum_{i\in S}(D_i-\hat m_{R-1}) + \frac{2B}{2^{R}\mu_3 n_{LB}}N_*$, \quad $N_*\sim\mathcal{N}(0,I)$ \label{line:Rterminationm}  
                \STATE $r^* \gets R$

    \ELSE  
        \afterctrlc{Otherwise: update the center and shrink the radius}
        \STATE $S \gets \{ i : \|D_i-\hat m_{r-1}\| < B/2^{\,r} \}$, \quad $n_{\tilde{S}} \gets \max(|S|, n_{LB})$
        \STATE $\hat m_r \gets \hat m_{r-1} + \frac{1}{n_{\tilde{S}}}\sum_{i\in S}(D_i-\hat m_{r-1}) + \frac{2B}{2^{r}\mu_1 n_{LB}}N_r$, \quad $N_r\sim\mathcal{N}(0,I)$ \label{line:meanrefinement} 
        \STATE $r \gets r+1$ 
    \ENDIF
\ENDWHILE
\STATE \textbf{Output:} $(\hat{m}, \hat{m}_0, \dots, \hat{m}_{r^* -1} , \hat{m}_{r^*}\mathbf{1}(r^* \neq R),  r^*)$
\end{algorithmic}
\end{algorithm}

\begin{proposition}
\label{prop:trimprivacy}
    The combined output $(\hat{m}, \hat{m}_0, \dots, \hat{m}_{r^* -1} , \hat{m}_{r^*}\mathbf{1}(r^* \neq R),  r^*)$ of Algorithm \ref{alg:trimestGeneral} satisfies $ \sqrt{R\mu_1^2 + R\mu_2^2 + \mu_3^2 }$-GDP.
\end{proposition}
\begin{proof}    
Let's first consider the case where the algorithm terminates at $r=R$.    The combined output is a function of the user data only through the statistics $\sum_{i=1}^n \omega_i$ and $\frac{1}{n_{\tilde{S}}} \sum_{i \in {S}} (D_i - \hat{m}_{r-1})$.   The output uses at most $2R+1$ such statistics, and in particular it uses the data exactly $2R+1$ times if the algorithm terminates at $r = R$.  Up to $R$ times for the mean refinement (line \ref{line:meanrefinement}), up to $R$ times for the number of users check (line \ref{line:countcheck}), and once for the final output (line \ref{line:earlyterminationm} or \ref{line:Rterminationm}).
    By properties of the Gaussian mechanism, all that we have to prove in case the algorithm terminates at $r = R$ is that the sensitivity of the statistic in line \ref{line:countcheck} is $1$, the sensitivity of the statistic in lines \ref{line:earlyterminationm} and \ref{line:Rterminationm} are $\frac{2B}{2^{r-1}n_{LB}}$ and $\frac{2B}{2^RL}$ respectively, and finally that the sensitivity of the statistic in line \ref{line:meanrefinement} is $\frac{2B}{2^rL}$.

    The sensitivity of $\sum_{i=1}^n \omega_i$ is simple. The $\omega_i$ are indicators, one for each user, and changing a user's entire data row can change the indicator by at most $1$. Hence, the sensitivity of this statistic is $1$, and as a result each call to line \ref{line:countcheck} satisfies $\mu_2$-GDP.  Here we use the fact that $\hat{m}_{r-1}$ is separately privatized in the preceding iteration, so that by post-processing no further privacy is leaked.

    Finally, lines \ref{line:earlyterminationm}, \ref{line:Rterminationm}, and \ref{line:meanrefinement} all concern the sensitivity of $\frac{1}{n_{\tilde{S}}}\sum_{i \in S}(D_i - \hat{m}_{r-j})$, where $S$ is an indexing set over $[n]$ with $n_{\tilde{S}}= \max\{|S|, n_{LB}\}$, and $j \in \{1,2\}$, where $j = 2$ in line \ref{line:earlyterminationm} and $j=1$ in lines \ref{line:Rterminationm} and \ref{line:meanrefinement}. Again, $\hat{m}_{r-j}$ was separately privatized such that by post-processing we can regard it as a public quantity. 

    Without loss of generality, assume the data of user $1$ changes. There are four cases to consider. The first case is one where $1 \notin S$ before and after changing the first user's data. In this case, the sensitivity is zero. The second case is one where $1 \in S$ before and after changing the first user's data. In this case, $n_{\tilde{S}} = n_{\tilde{S}}'$, and the sensitivity is bounded as
    \begin{align*}
    \Big\|\frac{1}{n_{\tilde{S}}}\sum_{i \in S}(D_i  - \hat{m}_{r-j})- \frac{1}{n_{\tilde{S}}'}\sum_{i \in S'}(D_i' - \hat{m}_{r-j})\Big\| &\leq \frac{1}{n_{LB}}\Big\|D_i - \hat{m}_{r-j} - (D_i' - \hat{m}_{r-j})\Big\| \\
    &\leq \frac{1}{n_{LB}} \Big(\frac{B}{2^{r-(j-1)}} + \frac{B}{2^{r-(j-1)}}\Big)  = \frac{2B}{2^{r-(j-1)}n_{LB}}.
    \end{align*}
    The third case is where $1 \in S$ initially, but not after changing the first user's data. In this case $n_{\tilde{S}}' \leq n_{\tilde{S}}$. Since $n_{\tilde{S}}'$ is lower bounded by $n_{LB}$, the worst case change in this respect is if $n_{\tilde{S}} = n_{LB}+1$ and $n_{\tilde{S}}' = n_{LB}$. Then, the sensitivity is
    \begin{align*}
    \Big\|\frac{1}{n_{\tilde{S}}}\sum_{i \in S}(D_i - \hat{m}_{r-j}) - \frac{1}{n_{\tilde{S}}'}\sum_{i \in S'}(D_i - \hat{m}_{r-j}) \Big\| &= \Big\|\Big(\frac{1}{n_{\tilde{S}}} - \frac{1}{n_{\tilde{S}}'}\Big)\sum_{i \in S'}(D_i - \hat{m}_{r-j}) + \frac{1}{n_{\tilde{S}}}(D_1 - \hat{m}_{r-j})\Big\|  \\ &\leq \Big\|\frac{-1}{n_{\tilde{S}} n_{\tilde{S}}'}\sum_{i \in S'}(D_i - \hat{m}_{r-j})\Big\| +\Big\| \frac{1}{n_{\tilde{S}}'}(D_1 - \hat{m}_{r-j})\Big\| \\
    &\leq \frac{2B}{2^{r-(j-1)}n_{LB}}.
    \end{align*}

    Finally, the fourth case is where $1 \notin S$, but $1 \in S'$. Thus, $n_{\tilde{S}}' \geq n_{\tilde{S}} \geq n_{LB}$. Again, the worst case change is if $n_{\tilde{S}} = n_{LB}$, $n_{\tilde{S}}' = n_{LB} + 1$, so that we obtain
    \begin{align*}
    \Big\|\frac{1}{n_{\tilde{S}}}\sum_{i \in S}(D_i - \hat{m}_{r-j}) - \frac{1}{n_{\tilde{S}}'}\sum_{i \in S'}(D_i - \hat{m}_{r-j})\Big\|\leq \frac{2B}{2^{r-(j-1)}n_{LB}},
    \end{align*}
    since this is completely symmetrical to the previous case. Thus, the statistics in line \ref{line:earlyterminationm}, where we have $j = 2$, has sensitivity bounded by $\frac{2B}{2^{r-(2-1)}n_{LB}} = \frac{2B}{2^{r-1}n_{LB}}$, whereas the statistics in lines \ref{line:Rterminationm} and \ref{line:meanrefinement} have sensitivity bounded by $\frac{2B}{2^{r-(1-1)}n_{LB}} = \frac{2B}{2^{r}n_{LB}}$.

    As a result, by the Gaussian mechanism, each instance of line \ref{line:countcheck} satisfies $\mu_2$-GDP, each instance of line \ref{line:meanrefinement} satisfies $\mu_1$-GDP, and the final output (line \ref{line:earlyterminationm} or \ref{line:Rterminationm}) satisfies $\mu_3$-GDP. Hence, the combined output of all $\hat{m}_r$, private count checks, and the final output, satisfies $\mu$-GDP. Since our final outputs $(\hat{m}, r^*)$ are a function of the data only through these privatized statistics, by composition $(\hat{m}, r^*)$ satisfies $\sqrt{R\mu_1^2 + R\mu_2^2 + \mu_3^2}$-GDP.

    The above analysis covers the case where the algorithm terminates through reaching $r = R$. The other possibility is that the algorithm terminates through entering the if-statement of line \ref{line:countcheck}. In this case, the privacy budget reserved for the refinements and count checks of the remaining iterations remains unused and can therefore be recaptured for use in the final estimate. In particular, say that the algorithm terminates through the if-statement of line \ref{line:countcheck} with $r^*$. In this case, the algorithm performed $r^*+1$ center refinements, namely in iterations $0$ through $r^*$. Moreover, the algorithm performed $r^* + 1$ user count checks, namely in iterations $1$ through $r^* + 1$. Hence, if the final output is released with a Gaussian mechanism instantiation with privacy parameter $\mathcal{C}\mu_3$, then the total privacy budget expended will be
    \[
    \sqrt{(r^* + 1)\mu_1^2 + (r^* + 1)\mu_2^2 + \mathcal{C}^2 \mu_3^2},
    \]
    and we have to find the parameter $\mathcal{C}$ that allows us to ``recapture'' the unused privacy budget by solving
 \[
    \sqrt{(r^* + 1)\mu_1^2 + (r^* + 1)\mu_2^2 + \mathcal{C}^2 \mu_3^2} = \sqrt{R\mu_1^2 + R\mu_2^2 + \mu_3^2}.
    \]
    The solution to this equation is easily found to be
    \[
    \mathcal{C} = \frac{\sqrt{(R - r^* - 1)\mu_1^2 + (R - r^* - 1)\mu_2^2 + \mu_3^2}}{\mu_3} = \sqrt{(R - r^* - 1)\frac{\mu_1^2 + \mu_2^2}{\mu_3^2} + 1},
    \]
    so that we can inflate $\mu_3$ by this factor in the if-statement of line \ref{line:countcheck} and still achieve our desired privacy guarantee. Notice in particular that if $r^* = R-1$, then $\mathcal{C} = 1$ and there is no recapturing of privacy budget, but for any smaller value of $r^*$, as long as $\mu_1$ or $\mu_2 >0$, this will result in privacy budget recapture.

    \end{proof}

\subsection{Proof of Theorem \ref{thm:trimest}}
\label{app:trimest}
\begin{proof}
     By Proposition \ref{prop:trimprivacy}, the output of Algorithm \ref{alg:trimestGeneral}, namely, $(\hat{m}_0, \dots, \hat{m}_{r^*-1}, \hat{m}_{\mathsf{DP}}, r^*)$ if $r^* = R$ or $(\hat{m}_0, \dots, \hat{m}_{r^*}, \hat{m}_{\mathsf{DP}}, r^*)$ if $r^* < R$, is $\sqrt{R\mu_1^2 + R\mu_2^2 + \mu_3^2 }$-GDP. Moreover, Algorithm \ref{alg:trimestTemp} is an instantiation of Algorithm \ref{alg:trimestGeneral} with $\mu_1 = \mu_2 = \frac{\mu}{2\sqrt{R}}$ and $\mu_3 = \frac{\mu}{\sqrt{2}}$. 
     Hence, $(\hat{m}_{r^*-1}, \hat{m}_{\mathsf{DP}}, r^*)$ satisfies $\mu$-GDP.
     Furthermore, $B^*$ is a deterministic post-processing of $r^*$. Therefore, the output of Algorithm \ref{alg:trimestTemp} satisfies $\mu$-GDP. In the following, we prove the utility part of the theorem. When useful, we will sometimes use the labels $\mu_1 = \mu_2 = \frac{\mu}{2\sqrt{R}}$ and $\mu_3 = \frac{\mu}{\sqrt{2}}$.

Before proceeding with the proof, we define a few events that will be useful in establishing our results in what follows. Let $Z_r$ be the iid standard Gaussians from Line \ref{line:ZrAlg} of Algorithm \ref{alg:trimestGeneral}, and define the event $\mathcal{E}_1$ as
\begin{equation}\label{eq:trimestE1}
\mathcal{E}_1 = \Big\{\max_{r \in [R]} |Z_r| \leq \sqrt{2\log\frac{4R}{\xi}}\Big\},
\end{equation}
which holds with probability at least $1-\xi/2$ by Lemma \ref{suplemma:max_subG}. Let  $N_r$ be the iid $\mathcal{N}(0, I)$ $d$-dimensional standard Gaussians used in Line \ref{line:meanrefinement} of Algorithm \ref{alg:trimestGeneral} and define event $\mathcal{E}_2$ as
\begin{equation}\label{eq:trimestE2}
\mathcal{E}_2 =\Big\{\max_{r \in 0 , \dots, R-1} \|N_r\| \leq \sqrt{d} + \sqrt{2\log\frac{3(R+1)}{\xi}} \Big\},
\end{equation}
which holds with probability at least $1-\xi/3$ by Lemma \ref{suplemma:maxnormals}. Finally, let $N_*$ be the termination Gaussian noise used in Line \ref{line:earlyterminationm} or \ref{line:Rterminationm} in Algorithm \ref{alg:trimestGeneral} and define event $\mathcal{E}_3$ as 
\begin{equation}\label{eq:trimestE3}
    \mathcal{E}_3 = \|N_{*}\| \leq \sqrt{d} + \sqrt{2 \log \frac{6}{\xi}},
\end{equation}
which holds with probability at least $1-\xi/6$.
Then the event $\mathcal{E} = \mathcal{E}_1 \cap \mathcal{E}_2 \cap \mathcal{E}_3$ holds with probability at least $1-\xi$. We analyze the performance of our algorithm conditional on event $\mathcal{E}$. 
 In doing so, we will make use of the following auxiliary Lemma, which we prove after the proof of this Theorem.
     \begin{lemma}\label{lem:propertiesmean}
    Admit the Assumptions of Theorem \ref{thm:trimest}. On the event $\mathcal{E}$, constructed from~\eqref{eq:trimestE1}-\eqref{eq:trimestE3}, the following properties hold for the internal quantities of Algorithm \ref{alg:trimestTemp}:

    \textbf{Property 1:} $n_{LB} \geq n/2$.

    \textbf{Property 2:} $\frac{B}{2^{r^*}} < 2K_2$.

    \textbf{Property 3:} $\hat{m}_\mathsf{DP} = \frac{1}{|S|}\sum_{i \in {S}} D_i + B^* Z$ with $|S| \geq n_{LB}$ for $Z$ a standard normal random vector bounded on $\mathcal{E}$ and $B^* = O(\frac{K_2}{\mu n})$.
    
    \textbf{Property 4.}\label{gfdrgdrg} For some constant $C>0$ and any vector $v \in \mathbb{R}^d$, and the termination step 
    indexing set $S$ of users included in the final estimate $\hat{m}_{\mathsf{DP}}$,
    \[
    \Big\|\frac{1}{|S|} \sum_{i \in S}D_i - v \Big\| \leq \|\bar{D} - v\| + \frac{CK_2}{n\mu} \sqrt{R \log \frac{R}{\xi}},
    \]
    and consequently,
    \[
    \|\hat{m}_{\mathsf{DP}} - v \| \leq \|\bar{D} - v\| + \frac{CK_2}{n\mu}\sqrt{R \log \frac{R}{\xi}} + \frac{CK_2}{n\mu} \sqrt{d + \log \frac{1}{\xi}}.
    \]
\end{lemma}
It is now straightforward to prove the utility result of Theorem \ref{thm:trimest}. 
By Property 4 of Lemma \ref{lem:propertiesmean}, taking $v = m$ and using $\|\bar{D}- m\| \leq K_1$  by assumption, it immediately follows
\begin{align}\label{eq:finalmDPbound}
\|\hat{m}_{\mathsf{DP}} - m\|  \leq K_1 + \frac{CK_2}{n\mu}\sqrt{R \log \frac{R}{\xi}} + \frac{CK_2}{n\mu} \sqrt{d + \log \frac{1}{\xi}} = O\Big(K_1+ \frac{K_2}{n \mu}\sqrt{R \log \frac{R}{\xi}}  + \frac{K_2}{n \mu} \sqrt{d + \log \frac{1}{\xi}}   \Big).
\end{align}
\end{proof}

\subsubsection*{Proof of Lemma \ref{lem:propertiesmean}}
\begin{proof}
   We show that on the event $\mathcal{E}$, Properties 1-3 hold deterministically, starting with Property 1.

\textbf{Property 1.} Notice that $2\tau - n = n - \frac{4}{\mu}\sqrt{2R \log \frac{4R}{\xi}}$, so that if $n \geq \frac{8}{\mu} \sqrt{2R \log \frac{4R}{\xi}}$, which we assume, then $2\tau - n \geq n/2 \geq 1$, and hence $n_{LB} = 2\tau - n \geq n/2$.

\textbf{Property 2.} By $\mathcal{E}_1$, we know that if $\sum_{i=1}^n \omega_i = n$, then 
$$\omega \geq \sum_{i=1}^n \omega_i + \frac{2 \sqrt{R}}{\mu} \min_{r \in [R]} Z_r \geq n - \frac{2}{\mu}\sqrt{2R\log\frac{4R}{\xi}} = \tau.$$
 In other words, the algorithm will not terminate as long as $\sum_{i=1}^n \omega_i = n$.

 In what follows, we show that under $\mathcal{E}$, as long as $n \geq\frac{C}{\mu}\sqrt{R(\log\frac{R}{\xi}+ d)}$ for some absolute constant $C>0$, as assumed, the algorithm will not terminate at step $r$ if $B/2^r \geq 4K_2$. Hence, if the algorithm terminates early at step $r < R$, we know that $B/2^{r} < 4K_2$. In this case $r^* = r-1$ so $B/2^{r^*} = B/2^{r-1}  < 2K_2$.

 Now we prove this by induction. In particular, we show that if $\sum_{i=1}^n \omega_i = n$ in step $r-1 \neq R$, then $\sum_{i=1}^n \omega_i = n$ in step $r$, meaning the algorithm will not terminate, so long as $n \geq\frac{C}{\mu}\sqrt{R(\log\frac{R}{\xi}+ d)}$ for some absolute constant $C>0$ and $B/2^{r+1} \geq 2K_2$.

 In the zero-th iteration, by our assumption $\max_i \|D_i\| \leq B$, we have that
\[\sum_{i=1}^n \omega_i = \sum_{i=1}^n \mathbbm{1}\{\|D_i - \hat{m}_{-1}\| \leq B\} = \sum_{i=1}^n \mathbbm{1}\{\|D_i\| \leq B\}= n.\]
Now assume that $\sum \omega_i = n$ in step $r-1 \neq R$. In this case, the algorithm proceeds to line~\ref{step:iterateS} and we notice that $|S| = \sum \omega_i = n$, meaning $n_{\tilde{S}} = |S|$, from which it follows $\frac{1}{n_{\tilde{S}}}\sum_{i \in S}(D_i - \hat{m}_{r-2})  =  \frac{1}{n}\sum_{i =1}^n(D_i - \hat{m}_{r-2})$. Using $\bar{D} := \frac{1}{n} \sum_{i=1}^n D_i$, from line~\ref{step:iteratem}, we find
\[
\hat{m}_{r-1}  =  \hat{m}_{r-2} + \frac{1}{n}\sum_{i =1}^n(D_i - \hat{m}_{r-2}) + \frac{4B\sqrt{R}}{2^{r-1}\mu n_{LB}} N_{r-1} = \bar{D} + \frac{4B\sqrt{R}}{2^{r-1}\mu n_{LB}} N_{r-1}.
\]

Now we want to find conditions under which $\sum_{i=1}^n \omega_i = n$ in step $r$ as well. First notice that on event $\mathcal{E}_2$,
\begin{equation}
\begin{split}
\label{eq:Di_bound}
\|D_i - \hat{m}_{r-1}\| = \Big\|D_i - \bar{D} - \frac{4B\sqrt{R}}{2^{r-1}\mu n_{LB}}N_{r-1} \Big\| &\leq \|D_i - \bar{D}\| + \frac{4B\sqrt{R}}{2^{r-1}\mu n_{LB}}\|N_{r-1}\| \\
&\leq \|D_i - \bar{D}\| + \frac{4B\sqrt{R}}{2^{r-1}\mu n_{LB}}\Big(\sqrt{d} + \sqrt{2\log\frac{3(R+1)}{\xi}}\Big).
\end{split}
\end{equation}
Next, if \textbf{(A)} $\max_i\|D_i - \bar{D}\| \leq B/2^{r+1}$ and \textbf{(B)}
\[
\frac{4B\sqrt{R}}{2^{r-1}\mu n_{LB}}\Big(\sqrt{d} + \sqrt{2\log\frac{3(R+1)}{\xi}}\Big) \leq \frac{B}{2^{r+1}},
\]
then \eqref{eq:Di_bound} implies $\|D_i - \hat{m}_{r-1}\| \leq B/2^r$ for all $i$. From this, we see that in step $r$, we again have $\sum \omega_i = n$. Now we find conditions under which \textbf{(A)} and \textbf{(B)} are true.

First, for condition \textbf{(A)}, by assumption,
\[
\|D_i - \bar{D}\| = \|D_i - m - \bar{D} + m\| \leq \|D_i - m\| + \|\bar{D} - m\| \leq K_1 + K_2 \leq 2K_2;
\]
hence, $\max_i \|D_i - \bar{D}\| \leq 2K_2$. We notice that the final step uses $K_1 \leq K_2$, which is trivially true since $\max_i\|{D}_i - m\| \leq K_2$ implies $\|\bar{D} - m\| \leq K_2$. Therefore, $\max_i\|D_i - \bar{D}\| \leq B/2^{r+1}$ if $2K_2 \leq B/2^{r+1}$.

By property 2, we have that $n_{LB} = 2\tau - n$.  Then,
Condition \textbf{(B)} is satisfied if
\begin{align}\label{eq:alg1ReqonN}
\frac{16\sqrt{R}}{\mu}\Big(\sqrt{d} + \sqrt{2\log\frac{3(R+1)}{\xi}}\Big) \leq n_{LB} = 2 \tau - n = n - \frac{4}{\mu}\sqrt{2R\log\frac{4R}{\xi}},
\end{align}
which is true when $n \geq \frac{C}{\mu}\sqrt{R(\log\frac{R}{\xi}+ d)}$ for some absolute constant $C>0$, which holds by assumption.

Thus, we have established that if $n$ is large enough, the algorithm will not terminate early unless the radius is small enough: \begin{equation}\label{eq:boundedradius}
B/2^{r^*} < 2K_2;
\end{equation} hence, bounding the induced privacy error.  Moreover, we notice that by assumption 
$R \geq C\log({B}/{K_2}),$
for some absolute constant $C>0$; this ensures that in the case that the algorithm does not terminate early and hence $r^* = R$, we have $B/2^{r^*} = B/2^{R} \leq K_2.$

\textbf{Property 3.}  Notice that if we terminate at $r^*$, where either $r^* = R$ or $r^* < R$, the algorithm passed the $\omega$ check when the radius is equal to $B/2^{r^*}$ and the center is $\hat{m}_{r^*-1}$; we label this as $\omega^*$ and note that since $\omega^*$ passed the line check, we necessarily have $\omega^* \geq \tau$.

Furthermore, for convenience, in the following we let 
$$\tilde{\mathcal{C}} := \mathcal{C} \mathbb{I}\{r^* < R\} + \mathbb{I}\{r^* = R\} = \mathbb{I}\{r^* < R\} \sqrt{2 - \frac{r^* + 1}{R}} + \mathbb{I}\{r^* = R\} \geq 1,$$
to accommodate the two types of termination events.

Now we consider $S$ at the termination step and notice that in either case $r^* = R$ or $r^* = r-1 < R$, we have that $|S| + \frac{2\sqrt{R}}{\mu} Z_{r^*} = \omega^*$. Therefore, conditional on $\mathcal{E}_1$ we have
\begin{equation}
\label{eq:tildeScardinality}
|S| = \omega^* - \frac{2\sqrt{R}}{\mu} Z_{r^*} \geq \tau - \frac{2}{\mu}\sqrt{2R\log\frac{4R}{\xi}} = n - \frac{4}{\mu}\sqrt{2R\log\frac{4R}{\xi}} = 2 \tau - n = n_{LB}.
\end{equation}
As a direct consequence, on event $\mathcal{E}_1$, we have $n_{\tilde{S}}  = \max\{|S|, n_{LB}\}  =  |S|$. Since we have $\hat{m}_{\mathsf{DP}} = \hat{m}_{r^* - 1} + \frac{1}{n_{\tilde{S}}} \sum_{i \in S} (D_i - \hat{m}_{r^*-1}) + \frac{2B}{2^{r^*} \tilde{\mathcal{C}}\mu_3 n_{LB}}N_*$ with $n_{\tilde{S}} = |S|$ by \eqref{eq:tildeScardinality}, it follows that $\hat{m}_{\mathsf{DP}} = \frac{1}{|S|} \sum_{i \in S} D_i +\frac{2B}{2^{r^*} \tilde{\mathcal{C}}\mu_3 n_{LB}}N_*$. Finally, for $B^* = \frac{2B}{2^{r^*} \tilde{\mathcal{C}}\mu_3 n_{LB}}$, notice that   $B/2^{r^*} \leq 2K_2$ by Property 2,   $n_{LB} \geq n/2$ by Property 1, and by construction $\tilde{\mathcal{C}} \geq 1$ and $\mu_3 = \frac{\mu}{\sqrt{2}}$, and hence $B^* = O(\frac{K_2}{n\mu})$.

\textbf{Property 4.} 
Denoting our final output as $\hat{m}_{\mathsf{DP}}$ and analyzing our algorithm on the event $ \mathcal{E}_1 \cap \mathcal{E}_2$, we have
\begin{equation}
\begin{split}
\label{eq:mdp_bound}
\|\hat{m}_{\mathsf{DP}}  - v\| 
&= \Big \|\frac{1}{|S|}\sum_{i \in S}D_i +  \frac{2\sqrt{2} B}{2^{r^*}\tilde{\mathcal{C}}\mu n_{LB}}N_{*} - v \Big\| \leq  \Big\|\frac{1}{|S|}\sum_{i \in S}D_i - v  \Big \| + \frac{2\sqrt{2} B}{2^{r^*}\tilde{\mathcal{C}}\mu n_{LB}} \|N_{*}\|.
\end{split}
\end{equation}
In what follows, we study the two terms on the right side of \eqref{eq:mdp_bound} separately.

We first show that the second term on the right side of \eqref{eq:mdp_bound} is $O(\frac{K_2}{n \mu} \sqrt{d + \log \frac{1}{\xi}})$. We recall that by Property 3, $B^* = O(\frac{K_2}{n\mu})$.
Therefore, on event $\mathcal{E}$, we have
\begin{align}
\label{eq:second_bound}
\frac{2\sqrt{2} B}{2^{r^*}\mu n_{LB}} \|N_{*}\|
\leq \frac{C K_2}{ n \mu } \Big(\sqrt{d + \log \frac{1}{\xi}}\Big).
\end{align}

Now consider the first term on the right side of \eqref{eq:mdp_bound}.
Denote $\bar{D}_S := \frac{1}{|S|}\sum_{i \in S} D_i$ and $S^\mathsf{c} := [n] \char`\\ S$.
Now, since $n \bar{D} =|S|\bar{D}_S + |S^\mathsf{c}|\bar{D}_{S^\mathsf{c}}$, it follows that
\begin{equation}\label{eq:DBarDecomp}
\bar{D}_S = \frac{n}{|S|}\bar{D} - \frac{|S^\mathsf{c}|}{|S|}\bar{D}_{S^\mathsf{c}} = \bar{D} - \frac{|S^\mathsf{c}|}{|S|}\left(\bar{D}_{S^\mathsf{c}} - \bar{D}\right);
\end{equation}
hence, we have
\[
 \Big\|\frac{1}{|S|}\sum_{i \in S}D_i - v \Big\| =  \
\left\|\bar{D}_S - v \right\| = \Big\|\left(\bar{D} - v \right) - \frac{|S^\mathsf{c}|}{|S|}\left(\bar{D}_{S^\mathsf{c}} - \bar{D} \right) \Big\| \leq \left \|\bar{D} -v \right\| + \frac{|S^\mathsf{c}|}{|S|} \left\|\bar{D}_{S^\mathsf{c}} - \bar{D} \right\|.
\]
The first term is as desired, and for the second term, we use that
\begin{equation}
\label{eq:DBarDecomp2}
\left \|\bar{D}_{S^\mathsf{c}} - \bar{D} \right \| = \Big\|\frac{1}{|S^\mathsf{c}|}\sum_{i \in S^\mathsf{c}}D_i - \bar{D} \Big\| \leq \frac{1}{|S^\mathsf{c}|}\sum_{i \in S^\mathsf{c}} \left\|D_i - \bar{D} \right\| \leq \frac{1}{|S^\mathsf{c}|} \sum_{i \in S^\mathsf{c}}(\left\|D_i - m \right\| + \left\|\bar{D} - m \right\|) \leq  2K_2.
\end{equation}
Finally, to show the desired bound, we use that
\begin{equation}\label{eq:sizeSComplement}
\frac{|S^\mathsf{c}|}{|S|}  = \frac{n - |S|}{|S|} \leq \frac{n - n_{LB}}{n_{LB}} \leq \frac{\frac{8}{\mu}\sqrt{2R \log \frac{4R}{\xi}}}{n} =  O\Big( \frac{1}{n\mu}  \sqrt{R\log\frac{R}{\xi}} \Big),
\end{equation}
where the first inequality follows from  \eqref{eq:tildeScardinality} and the second inequality from Property 1.

It follows from~\eqref{eq:DBarDecomp}-\eqref{eq:sizeSComplement} that we have
\[
\Big\|\frac{1}{|S|} \sum_{i \in S} D_i - v \Big\| \leq \|\bar{D} - v\| +  \frac{CK_2}{n\mu}\sqrt{R \log \frac{R}{\xi}},
\]
for some constant $C>0$,
and consequently, using also \eqref{eq:second_bound},
\[
\|\hat{m}_{\mathsf{DP}} - v\| \leq \|\bar{D} - v\| + \frac{CK_2}{n\mu}\sqrt{R \log \frac{R}{\xi}} + \frac{CK_2}{n\mu} \sqrt{d + \log \frac{1}{\xi}}
\]
\end{proof}

\subsection{Proof of Lemma \ref{lemma:subgaussianmeanconc}}
\begin{proof}
    The statement about $\|\frac{1}{nT} \sum_{i=1}^n \sum_{t=1}^T w_{i,t} - m\|$ follows directly from Lemma \ref{suplemma:sumSubGaussians} in combination with Lemma \ref{lemma:epsilonNets}. Namely, for $u \in \mathbb{S}^{d-1}$, we have that $u^\tr(w_{i,t} - m)$ is $\sigma^2$-subGaussian. Hence, by Lemma \ref{suplemma:sumSubGaussians}, 
    \begin{align}\label{eq:vectorsubgaussian}
    \mathbb{P}\Big(\Big|u^\tr \Big[\frac{1}{nT}\sum_{i=1}^n \sum_{t=1}^T w_{i,t} - m\Big]\Big| > q\Big) \leq 2\exp \Big(-\frac{nTq^2}{2\sigma^2}\Big),
    \end{align}
    and hence, by setting the right side of \eqref{eq:vectorsubgaussian} equal to $\xi$ and solving for $q$,
    \[
    \mathbb{P}\Big(\Big|u^\tr \Big[\frac{1}{nT}\sum_{i=1}^n \sum_{t=1}^T w_{i,t} - m\Big]\Big| > \sigma\sqrt{\frac{2}{nT} \log \frac{2}{\xi}}\Big) \leq \xi.
    \]
    Then, by Lemma \ref{lemma:epsilonNets},
    \[
    \mathbb{P}\Big(\Big\|\frac{1}{nT}\sum_{i=1}^n \sum_{t=1}^T w_{i,t} - m \Big\| > \frac{4\sigma}{3}\sqrt{\frac{2}{nT} \Big(d \log 9 + \log \frac{2}{\xi}}\Big)\Big) \leq \xi,
    \]
    and our result for this case follows. For $\max_{i \in [n]} \|\frac{1}{T}\sum_{t=1}^T w_{i,t} - m\|$, notice that by following the preceding steps we can obtain, for any $i \in [n]$,
    \[
    \mathbb{P}\Big(\Big\|\frac{1}{T}\sum_{t=1}^T w_{i,t} - m \Big\| > \frac{4\sigma}{3}\sqrt{\frac{2}{T} \Big(d \log 9 + \log \frac{2n}{\xi}}\Big)\Big) \leq \xi/n,
    \]
    so that our result for this case follows by a union bound over $[n]$.
\end{proof}

\subsection{Proof of Corollary \ref{corollary:subgaussianmeanestimation}}
\begin{proof}
    Let $\mathcal{E}$ denote the event that the bounds of Lemma \ref{lemma:subgaussianmeanconc} hold with probability parameter $\xi' = \xi/2$. On $\mathcal{E}$, given the scaling conditions of Theorem \ref{thm:trimest},  Theorem \ref{thm:trimest} provides that with probability at least $1-\xi/2$,
    \[
    \|\hat{m}_{\mathsf{DP}} - m\| = O\Big(\sigma \sqrt{\frac{1}{nT} \Big(d + \log \frac{1}{\xi}\Big)} + \sigma\sqrt{\frac{1}{n^2\mu^2 T}\Big(d + \log \frac{n}{\xi}\Big)\Big(d + \log \frac{1}{\xi} + R\log\frac{R}{\xi}\Big)} \Big),
    \]
    such that by union bound over these two events that both hold with probability at least $1-\xi/2$, and by noticing that $\log \frac{1}{\xi} \leq R \log \frac{R}{\xi}$, we obtain, with probability at least $1-\xi$,
    \[
      \|\hat{m}_{\mathsf{DP}} - m\| = O\Big(\sigma \sqrt{\frac{1}{nT} \Big(d + \log \frac{1}{\xi}\Big)} + \sigma\sqrt{\frac{1}{n^2\mu^2 T}\Big(d + \log \frac{n}{\xi}\Big)\Big(d + \log \frac{1}{\xi} + R\log\frac{R}{\xi}\Big)} \Big).
    \]
\end{proof}

\section{Proofs of Section \ref{sec:dpregression}}
\label{appendix:section3proofs}
\subsection{Proof of Lemma \ref{lem:betasconcentration}}
\label{app:betasconcentration}
\begin{proof}
    We begin by deriving an event under which $\text{Rank}(X_i) = d$, so that $X_i ^\tr X_i$ is invertible and $\hat{\beta}_i = (X_i^\tr X_i)^{-1} X_i^\tr Y_i$. Recall that $\Gamma_i := \mathbb{E}[\frac{1}{T} X_i^\tr X_i]$,  and by Assumption~\ref{assumptions:initialcovariates} we have $\lambda_{\min}(\Gamma_i) \geq c_\Gamma >0$. Define the event $\mathcal{E}_i := \{\|\frac{1}{T} \Gamma_i^{-\frac{1}{2}} X_i^\tr X_i \Gamma_i^{-\frac{1}{2}} - I\| \leq  \frac{3}{4}\}$. Notice that on $\mathcal{E}_i$, for any $v \neq \mathbf{0}\in \mathbb{R}^d$,
    $$v^\tr \frac{1}{T} X_i^\tr X_iv = v^\tr \Gamma_i^{\frac{1}{2}} \Big(\frac{1}{T} \Gamma_i^{-\frac{1}{2}} X_i^\tr X_i \Gamma_i^{-\frac{1}{2}}\Big) \Gamma_i^{\frac{1}{2}} v = (\Gamma_i^{\frac{1}{2}} v)^\tr   \Big(\frac{1}{T} \Gamma_i^{-\frac{1}{2}} X_i^\tr X_i \Gamma_i^{-\frac{1}{2}}\Big) \Gamma_i^{\frac{1}{2}} v>0,$$
    since  $\lambda_{\min}(\frac{1}{T} \Gamma_i^{-\frac{1}{2}} X_i^\tr X_i \Gamma_i^{-\frac{1}{2}}) \geq \frac{1}{4}$ on $\mathcal{E}_i$ and $\Gamma_i^{\frac{1}{2}}$ is invertible so that $\Gamma_i^{\frac{1}{2}} v \neq \mathbf{0}$. Hence $\frac{1}{T} X_i^\tr X_i$ is invertible on $\mathcal{E}_i$. We now derive conditions under which $\mathcal{E}_i$ holds with high probability.

    Let $z_{i,t} := \Gamma_i^{-1/2}x_{i,t}$, which by Assumption \ref{assumptions:initialcovariates} are sub-Gaussian with proxy variance $\sigma_x^2$. 
             Now, for any arbitrary $u \in \mathbb{S}^{d-1}$, 
            let $(W(u)_{i,t})_{t\geq  1} := (u^\tr  z_{i,t}z_{i,t}^\tr u - \|u\|^2)_{t \geq 1}$.  By Lemma \ref{lem:squaresubgaussiansubexp}, we have
            $\mathbb{P}(|W(u)_{i,t}|>q) \leq 2e^{-\frac{q}{C\sigma_x^2}}$ for all $i \in [n]$ and $t \in [T]$ using that $\mathbb{E}[u^\tr  z_{i,t}z_{i,t}^\tr u] = u^\tr  \Gamma_i^{-1/2}\mathbb{E}[x_{i,t}x_{i,t}^\tr] \Gamma_i^{-1/2} u  = u^\tr u$. Moreover, by Lemma \ref{lem:measurablemixing},  since $z_{i,t}$ is exponentially strongly mixing, so is the sequence $(W(u)_{i,t})_{t \in \mathbb{Z}}$  for all $i \in [n]$. Hence, by invoking Lemma \ref{lem:expexpconc} with $K := C\sigma_x^2$
             we have with probability at least $1-\xi$, 
         \begin{align}\label{eq:concentrationWi}
         \Big|\frac{1}{T} \sum_{t=1}^T W(u)_{i,t}\Big| = O\Big(\sigma_x^2\sqrt{\frac{1}{T} \log \frac{1}{\xi}}  + \frac{C_{\sigma_x}}{T} \big(\log \frac{T}{\xi}\big)^2 \Big).
         \end{align}
         Using an $\epsilon$-net argument (Lemma \ref{lemma:epsilonNets}), it follows, with probability at least $1-\xi$,
         \[
         \Big\|\frac{1}{T} \sum_{t=1}^T z_{i,t} z_{i,t}^\tr - I\Big\| = O\Big( \sigma_x^2\sqrt{\frac{1}{T} \big(d + \log \frac{1}{\xi}\big)}  +  \frac{C_{\sigma_x}}{T}\big(d + \log \frac{T}{\xi}\big)^2\Big).
         \]
         In particular, for any $i \in [n]$,  as long as
         $T \geq C(d + \log \frac{nT}{\xi})^2$ for some constant $C>0$,
         \begin{equation}
             \label{eq:prob_Ec}
         \mathbb{P}(\mathcal{E}_i^{\mathsf{c}}) = \mathbb{P}\left( \Big\|\frac{1}{T} \sum_{t=1}^T z_{i,t} z_{i,t}^\tr - I \Big\| > \frac{3}{4}\right) = \mathbb{P}\left( \Big\|\frac{1}{T} \Gamma_i^{-\frac{1}{2}} X_i^\tr X_i \Gamma_i^{-\frac{1}{2}} - I \Big\| > \frac{3}{4}\right) \leq \frac{\xi}{2n}.
         \end{equation}

         Hence, $\mathcal{E}_i$ holds with probability at least $1-\frac{\xi}{2n}$ as long as $T \geq C(d + \log \frac{nT}{\xi})^2$, and we analyze our error bounds on this event.

    We begin with the concentration of $\max_{i \in [n]} \|\hat{\beta}_i - \beta\|$. We first prove concentration of $u^\tr (\hat{\beta}_i - \beta)$ for arbitrary $i \in [n]$ and $u \in \mathbb{S}^{d-1}$, the unit sphere in $\mathbb{R}^d$, and then pass towards our target through union bounds. To that end, for any $i \in [n]$, noting that on $\mathcal{E}_i$ we have $\hat{\beta}_i = (X_i^\tr X_i)^{-1} X_i^\tr Y_i$,
    \begin{equation}
        \begin{split}
            \mathbb{P}(|u^\tr (\hat{\beta}_i - \beta) | > q) &\leq \mathbb{E}_{X_i}\left[\mathbbm{1}_{\mathcal{E}_i} \mathbb{P}\left(|u^\tr (\hat{\beta}_i - \beta) | > q \, \mid \, X_i \right)\right] + \mathbb{P}(\mathcal{E}_i^\mathsf{c}) \\
         &= \mathbb{E}_{X_i} \left[\mathbbm{1}_{\mathcal{E}_i}\mathbb{P}\left(|u^\tr (X_i^\tr X_i)^{-1} X_i^\tr \varepsilon_i| > q \, \mid \, X_i \right) \right] + \mathbb{P}(\mathcal{E}_i^\mathsf{c}) \\ 
    &= \mathbb{E}_{X_i}\Big[\mathbbm{1}_{\mathcal{E}_i} \mathbb{P}\Big( \Big|\frac{u^\tr (X_i^\tr X_i)^{-1} X_i^\tr}{\|u^\tr (X_i^\tr X_i)^{-1} X_i^\tr \|}\varepsilon_i \Big| > \frac{q}{\|u^\tr (X_i^\tr X_i)^{-1} X_i^\tr \|} \, \Big  \mid \, X_i \Big) \Big] + \mathbb{P}(\mathcal{E}_i^\mathsf{c}) \\
    &\leq \mathbb{E}_{X_i}\Big[\mathbbm{1}_{\mathcal{E}_i} \mathbb{P}\Big( |\tilde{u}^\tr \varepsilon_i| > \frac{q}{C_i} \,  \Big\mid \, X_i \Big) \Big] + \mathbb{P}(\mathcal{E}_i^\mathsf{c}),
    \label{eq:sub_gauss1}
    \end{split}
    \end{equation}
   where $C_i := \|(X_i^\tr X_i)^{-1} X_i^\tr \| \geq  \|u^\tr (X_i^\tr X_i)^{-1} X_i^\tr \|$ and $\tilde{u} \in \mathbb{S}^{T-1}$.  Now, considering \eqref{eq:sub_gauss1}, we notice that  $\tilde{u}^\tr \varepsilon_i$ is a $\sigma_\varepsilon^2$-sub-Gaussian random variable by Definition \ref{def:subgaussian} and Assumption \ref{assumptions:initialerrors}, thereby
    allowing us to apply the tail bound of Lemma \ref{suplemma:sumSubGaussians} (with $n=1$) to find 
    \begin{align}\label{Eq:beta_iconcDecomp} 
    \mathbb{P}\left(\left|u^\tr (\hat{\beta}_i - \beta)\right| > q \right) \leq  \mathbb{E}_{X_i}\left[\mathbbm{1}_{\mathcal{E}_i} 2\exp \Big\{ \frac{-q^2}{2\sigma_\varepsilon^2 C_i^2}\Big\}  \right] + \mathbb{P}(\mathcal{E}_i^\mathsf{c}).
    \end{align}
     The key now is to derive a deterministic bound on $C_i$ under $\mathcal{E}_i$.  We have,
    \begin{align}
   \nonumber C_i = \|(X_i^\tr X_i)^{-1} X_i^\tr \| &= \|\Gamma_i^{-\frac{1}{2}} \Gamma_i^{\frac{1}{2}} (X_i^\tr X_i)^{-\frac{1}{2}} (X_i^\tr X_i)^{-\frac{1}{2}} X_i^\tr \| \\
    &\leq \frac{1}{\sqrt{T}}\|\Gamma_i^{-\frac{1}{2}}\| \cdot \Big\|\Gamma_i^\frac{1}{2} \big(\frac{1}{T}X_i^\tr X_i\big)^{-\frac{1}{2}} \Big\| \cdot \|(X_i^\tr X_i)^{-\frac{1}{2}} X_i^\tr\|.
    \label{eq:Ci_bound}
    \end{align}
    First, by Lemma \ref{lem:spectrumbound}, we have that on $\mathcal{E}_i$,
         \begin{equation}
             \label{eq:Ci_bound2}
             \Big\|\Gamma_i^\frac{1}{2}\Big(\frac{1}{T} X_i^\tr X_i\Big)^{-\frac{1}{2}} \Big \| \leq \frac{1}{\sqrt{1 - {3}/{4}}} = 2,
         \end{equation}
    The third term in \eqref{eq:Ci_bound} is equal to $1$ on $\mathcal{E}_i$, because 
    $(X_i^\tr X_i)^{-\frac{1}{2}} X_i^\tr X_i (X_i^\tr X_i)^{-\frac{1}{2}} = I.$
    Thus, using this and \eqref{eq:Ci_bound2} in \eqref{eq:Ci_bound}, we find that on $\xi_i$ we have
    \begin{align}\label{eq:C_iBoundonEi}C_i^2 \leq \frac{1}{T} \Big\|\Gamma_i^\frac{1}{2} \big(\frac{1}{T}X_i^\tr X_i\big)^{-\frac{1}{2}} \Big\|^2 \|\Gamma_{i}^{-\frac{1}{2}}\|^2 \leq \frac{4}{T}  \|\Gamma_{i}^{-\frac{1}{2}}\|^2.\end{align}
Using this bound in \eqref{Eq:beta_iconcDecomp}, we have shown
    \begin{align}
    \mathbb{P}\left(\left|u^\tr (\hat{\beta}_i - \beta)\right| > q \right)
  &\leq 2\mathbb{E}_{X_i} \Big[\mathbbm{1}_{\mathcal{E}_i} \exp\Big\{\frac{-q^2}{2\sigma_\varepsilon^2 C_i^2}\Big\}\Big] + \mathbb{P}(\mathcal{E}_i^{\mathsf{c}}) \leq 2 \exp\Big\{\frac{-T q^2}{8\sigma_\varepsilon^2 \max_{j} \|\Gamma_{j}^{-{1}/{2}}\|^2}\Big\}+ \mathbb{P}(\mathcal{E}_i^{\mathsf{c}}).  \label{Eq:beta_iconcDecomp_new}
    \end{align}

Hence, by  \eqref{eq:prob_Ec} and \eqref{Eq:beta_iconcDecomp_new}, for any $u \in \mathbb{S}^{d-1}$, when $ q \geq 2\sigma_\varepsilon \max_{1 \leq j \leq n} \|\Gamma_j^{-\frac{1}{2}}\|\sqrt{\frac{2}{T} \log \frac{4n}{\xi}}$,
        \begin{align}\label{eq:UnitDirectionBetaIConc}
            \max_{ i \in [n]}\mathbb{P}\left(\left|u^\tr (\hat{\beta}_i - \beta)\right| > q \right)
            &\leq 2\exp \Big\{\frac{-Tq^2}{8\sigma_\varepsilon^2\max_{i}\|\Gamma_i^{-1/2}\|^2} \Big\} + \frac{\xi}{2n} \leq \frac{\xi}{n}.
        \end{align}
        We can now apply a union bound over $[n]$ and Lemma \ref{lemma:epsilonNets} to \eqref{eq:UnitDirectionBetaIConc} to obtain
        \begin{align*}
            &\mathbb{P}\Big( \max_{i \in [n]} \left\|\hat{\beta}_i - \beta \right\| > \frac{4}{3}2\sigma_\varepsilon \max_{i \in [n]}\|\Gamma_i^{-\frac{1}{2}}\| \sqrt{\frac{2}{T}\log \frac{9^d4n}{\xi}}\Big) \\
            & \quad  \leq n\max_{i \in [n]} \Big[\mathbb{P}\Big( \left \|\hat{\beta}_i - \beta \right\| > \frac{4}{3}2\sigma_\varepsilon \max_{i \in [n]}\|\Gamma_i^{-\frac{1}{2}}\| \sqrt{\frac{2}{T}  \log \frac{9^d4n}{\xi}}\Big)  \Big] \leq n \frac{\xi}{n} = \xi,
        \end{align*}
         concluding our proof that with probability at least $1-\xi$,
         \[
         \max_{i \in [n]}\left\|\hat{\beta}_i - \beta \right\| = O\Big(\sigma_\varepsilon\max_{i \in [n]}\|\Gamma_i^{-\frac{1}{2}}\| \sqrt{\frac{1}{T}\big(d + \log \frac{n}{\xi}\big)} \Big).
         \]

        We proceed to the second part of our theorem concerning the concentration of $\bar{\beta} - \beta$. The proof strategy is similar to that used above.
        We notice that $u^\tr (\bar{\beta} - \beta) = u^\tr (\frac{1}{n} \sum_{i=1}^n \hat{\beta}_i - \beta) = \frac{1}{n} \sum_{i=1}^n u^\tr ( \hat{\beta}_i - \beta)$ for any $u \in \mathbb{S}^{d-1}$. Then, as in \eqref{eq:sub_gauss1} and \eqref{Eq:beta_iconcDecomp},
        using Lemma \ref{suplemma:sumSubGaussians} now with $n$ independent $\|u^\top (X_i^\tr X_i)^{-1} X_i^\tr \|^2 \sigma_\varepsilon^2$-sub-Gaussian random variables %
        conditional on $X$,  using that $\varepsilon_i$ is independent of $\varepsilon_j$ given $X$, and defining $\mathcal{E} := \cap_{i=1}^n \mathcal{E}_i$,
        \begin{align*}
        \mathbb{P}\left(\left|u^\tr (\bar{\beta} - \beta) \right| > q \right)  &\leq 2\mathbb{E}_X \Big[\mathbbm{1}_\mathcal{E} \exp\Big\{\frac{-n^2q^2}{2\sigma_\varepsilon^2 \sum_{i=1}^n C_i^2} \Big\} \Big ] + \mathbb{P}(\mathcal{E}^\mathsf{c})  \leq 2\exp \Big\{\frac{-n^2Tq^2}{8\sigma_\varepsilon^2 \sum_{i=1}^n \|\Gamma_i^{-1/2}\|^2} \Big\} + \frac{\xi}{2} \leq \xi,
        \end{align*}
         where the final inequality follows for
         \[
         q = \Omega \Big(\sigma_\varepsilon\sqrt{\frac{1}{n} \sum_{i=1}^n\|\Gamma_i^{-1/2}\|^2} \sqrt{\frac{1}{nT}\log \frac{1}{\xi}} \Big).
         \]
         Therefore, there exists $C'>0$ such that for all $u\in\mathbb{S}^{d-1}$,
         \begin{align} \label{eq:barBetaDecomp} 
         \mathbb{P} \Big( \left|u^\tr (\bar{\beta} - \beta)\right| > C\sigma_\varepsilon\sqrt{\frac{1}{n} \sum_{i=1}^n\|\Gamma_i^{-1/2}\|^2} \sqrt{\frac{1}{nT}\log \frac{1}{\xi}}\Big) \leq \xi.
         \end{align}
         Finally, then,  again applying Lemma \ref{lemma:epsilonNets} to \eqref{eq:barBetaDecomp} we obtain, with probability at least $1-\xi$,
         \[
         \|\bar{\beta} - \beta\| = O \Big(\sigma_\varepsilon\sqrt{\frac{1}{n} \sum_{i=1}^n\|\Gamma_i^{-1/2}\|^2} \sqrt{\frac{1}{nT} \Big(d + \log \frac{1}{\xi}\Big)} \Big).
         \]
\end{proof}

\subsection{Proof of Lemma \ref{lem:suffUI}}
\label{app:suffUI}
\begin{proof}
For any $i \in [n]$, let $\Gamma_{i,T} := \frac{1}{T}X_i^\tr  X_i$ be the symmetric, positive semi-definite matrix of interest. Then, using the definition of the operator norm,
    \begin{align*}
    \mathbb{E}\left[\|\Gamma_{i,T}^{-1}\|^{k + \delta}\right] = \mathbb{E}\big[\lambda_{\min}(\Gamma_{i,T})^{-(k + \delta)}\big] &= \int_{0}^\infty \hspace{-2pt} \mathbb{P}\big(\lambda_{\min}(\Gamma_{i,T})^{-(k + \delta)} > r\big)dr \\
    &= \int_{0}^\infty \hspace{-2pt}  \mathbb{P}\big(\lambda_{\min}(\Gamma_{i,T}) < r^{\frac{-1}{k + \delta}}\big)dr \leq \int_0^\infty \hspace{-2pt}  \mathbb{P}(\lambda_{\min}(\Gamma_{i,T}) \leq r^{\frac{-1}{k+\delta}})dr. 
\end{align*}
Now, through a change of variables,
    \begin{align*}
\mathbb{E}\left[\big\|\Gamma_{i,T}^{-1}\big\|^{k + \delta}\right] \leq \int_{0}^\infty \mathbb{P}\big(\lambda_{\min}(\Gamma_{i,T}) \leq r^{\frac{-1}{k + \delta}}\big)dr &= (k + \delta) 
    \int_{0}^\infty r^{-(k + \delta) - 1}\mathbb{P} \big(\lambda_{\min}(\Gamma_{i,T}) \leq r\big) dr \\
    &\numberthis \label{eq:integralInGammaiT}\leq C\Big(\int_0^{r_0} r^{-(k + \delta) - 1} \mathbb{P}\big(\lambda_{\min}(\Gamma_{i,T}) \leq r \big)dr + \frac{1}{r_0}\Big),
\end{align*}
where the final step is true for any $r_0 >0$. Thus, to complete the proof it suffices to find a $r_0 >0$ and $T_0$ such that for all $i \in [n]$ and $T \geq T_0$, we have 
\begin{equation}
    \label{eq:anticoncentration_lambda_min}
\int_0^{r_0} r^{-(k+ \delta) - 1}\mathbb{P}(\lambda_{\min}(\Gamma_{i,T}) \leq r)dr < \infty.
\end{equation}
To that end, we 
note that for any $u, v \in \mathbb{S}^{d-1}$ and any symmetric, positive semi-definite $\Gamma_{i,T}$,
\[
\big|u^\tr  \Gamma_{i,T} \, u - v^\tr \Gamma_{i,T} \, v\big|  = \big|(u - v)^\tr \Gamma_{i,T} \, u + v^\tr\Gamma_{i,T} \,(u - v)\big|\leq 2\|\Gamma_{i,T}\|\cdot \|u -v \|.
\]
Now let $\mathcal{N}$ be an $\epsilon$-net of $\mathbb{S}^{d-1}$, meaning that for all $u \in \mathbb{S}^{d-1}$, there exists $\tilde{u} \in \mathcal{N}$ such that $\|v - \tilde u\| \leq \epsilon$. Then, let $v_i := \arg \min_{u \in \mathbb{S}^{d-1}} u^\tr \Gamma_{i,T} \, u$, and let $u_i \in \mathcal{N}$ be such that $\|u_i - v_i\| \leq \epsilon$. Then we know that 
\begin{align*}
    \lambda_{\min}(\Gamma_{i,T})=v_i^\tr \Gamma_{i,T} \, v_i \geq u_i^\tr \Gamma_{i,T} \, u_i - 2\|\Gamma_{i,T}\|\cdot \|u_i - v_i\| &\geq u_i^\tr \Gamma_{i,T} \, u_i - 2\epsilon\|\Gamma_{i,T}\| \\
    &\geq \min_{u \in \mathcal{N}}u^\tr \Gamma_{i,T} \, u - 2\epsilon \|\Gamma_{i,T}\|.
\end{align*}
So for any for any $L>0$,
\begin{align*}
\Big\{\lambda_{\min}(\Gamma_{i,T}) \leq r\Big\} &\subseteq \Big\{\min_{u \in \mathcal{N}}u^\tr \Gamma_{i,T} \, u \leq r + 2\epsilon \|\Gamma_{i,T}\| \Big\} \\
&= \Big\{\min_{u \in \mathcal{N}}u^\tr \Gamma_{i,T} \, u \leq r + 2\epsilon \|\Gamma_{i,T}\| \Big\} \cap \Big(\big\{\|\Gamma_{i,T}\| \leq L \big\} \cup \big\{\|\Gamma_{i,T}\| > L \big\}\Big)  \\
&\subseteq \Big\{\min_{u \in \mathcal{N}} u^\tr \Gamma_{i,T} \, u \leq r +2\epsilon L\Big\} \cup \big\{\|\Gamma_{i,T}\| > L \big\}.
\end{align*}
 Then, by a union bound over these events, we obtain
\begin{equation}
\label{eq:new_lambda_bound}
    \mathbb{P}\left(\lambda_{\min}(\Gamma_{i,T}) \leq r \right) \leq \mathbb{P}\big(\min_{v \in \mathcal{N}} v^\tr \Gamma_{i,T} \, v \leq r +2\epsilon L \big) + \mathbb{P}\left(\|\Gamma_{i,T}\| > L\right).
\end{equation}
Recall, we are trying to demonstrate the bound in \eqref{eq:anticoncentration_lambda_min}. Consider the second term on the right side of \eqref{eq:new_lambda_bound}.
By \eqref{eq: xitxitconcentration} and the triangle inequality, with probability at most $\xi$,
\[
\|\Gamma_{i,T} \| > \|\Gamma_i\| + C_\Gamma\sigma_x^2 \sqrt{\frac{1}{T} \big(d + \log \frac{1}{\xi}\big)} + \frac{ C_{\sigma_x, \Gamma}}{T}\big(\log \frac{T}{\xi}\big)^2,
\]
so if we set $L := C_\Gamma +C_\Gamma \sigma_x^2 \sqrt{\frac{1}{T} (d + \log r^{-(k + \delta + \gamma)})} + \frac{C_{\sigma_x,\Gamma}}{T} (\log [Tr^{-(k+\delta+\gamma)}])^2 $, then $\mathbb{P}(\|\Gamma_{i,T}\| > L) \leq r^{k + \delta + \gamma}$. Thus for this choice of $L$, we have from \eqref{eq:new_lambda_bound},
\begin{equation}
\label{eq:new_lambda_bound_2}
    \mathbb{P}\left(\lambda_{\min}(\Gamma_{i,T}) \leq r \right) \leq \mathbb{P}\big(\min_{v \in \mathcal{N}} v^\tr \Gamma_{i,T} \, v \leq r +2\epsilon L \big) + r^{k + \delta + \gamma},
\end{equation}
so to complete the proof, we show an upper bound for the first term on the right side of \eqref{eq:new_lambda_bound_2}.
Finally, set $\epsilon = \frac{r}{2L}$, where we recall that $0 < r \leq r_0$. Now let $r_0 \leq e^{-1}$, then we will aim to upper bound $L$ defined above: first notice
\begin{align*}
    C_\Gamma \sigma_x^2 \sqrt{\frac{1}{T}\Big(d - (k+\delta + \gamma)\log r\Big)} \leq C_\Gamma \sigma_x^2 \sqrt{d + (k+\delta+\gamma)|\log r|} \leq C\left(1 + |\log r|^2\right) \leq C_1 |\log r|^2,
\end{align*}
for constants $C_1, C>0$ not depending on $r$ and $T$, whereas defining $u := (k + \delta + \gamma)|\log r|$
\begin{align*}
    \frac{C_{\sigma_x,\Gamma}}{T} \Big(\log [Tr^{-(k+\delta+\gamma)}]\Big)^2  \leq \frac{C_{\sigma_x,\Gamma}}{T}(\log T + u)^2 \leq C_{\sigma_x,\Gamma}(4 + 2u^2) \leq c(1 + |\log r|^2) \leq C_2|\log r|^2,
\end{align*}
for constants $c, C_2>0$ not depending on $r$ and $T$. It follows that $L^d \leq C_3^d (\log r)^{2d}$ for a constant $C_3>0$ not depending on $r$ and $T$.
Using Corollary 4.2.13 of \cite{vershynin2018high} and the upper bound on $L$, we have  $|\mathcal{N}| \leq ({6L}/{r})^d \leq 6^d C_3^d |\log r|^{2d} r^{-d} =: C'|\log r|^{2d} r^{-d}$.
Thus, using a union bound over $\mathcal{N}$,
\begin{equation}
\label{eq:prob_bound1}
\mathbb{P}\big(\min_{v \in \mathcal{N}} v^\tr \Gamma_{i,T} \, v \leq r +2\epsilon L \big) \leq  |\mathcal{N}| \sup_{v \in \mathcal{N}} \mathbb{P}(v^\tr \Gamma_{i,T} \, v \leq 2r)  \leq  C' |\log r|^{2d} r^{-d} \sup_{v \in \mathcal{N}} \mathbb{P}(v^\tr \Gamma_{i,T} \, v \leq 2r),
\end{equation}
so that it remains to control $\mathbb{P}(v^\tr \Gamma_{i,T} \, v \leq 2r)$ for an arbitrary $v \in \mathbb{S}^{d-1}$. Fixing $v \in \mathbb{S}^{d-1}$, and letting that $w_{i,t} := x_{i,t}^\tr v$, by exploiting the assumed bound on $f_{W_i}$, we have
\begin{align*}
\mathbb{P}\big(v^\tr  \Gamma_{i,T} \,   v \leq 2r\big) &= \mathbb{P}\Big(\sum_{t=1}^T (x_{i,t}^\tr v)^2 < 2rT\Big) = \int_{\|w\| \leq \sqrt{2rT}} f_{W_i}(w) dw \leq  K^T\int_{\|w\| \leq \sqrt{2rT}}1 dw.
\end{align*}
From the above and by the inequality $\Gamma(z) > (\frac{z}{e})^z\sqrt \frac{2\pi}{z}$ (e.g.\ \citealt{olver2010nist}), we see that
\begin{align}
\label{eq:prob_bound2}
\mathbb{P}\big(v^\tr  \Gamma_{i,T} \,   v \leq 2r\big) 
\leq  K^T\text{Vol} (B_T(\sqrt{2rT})) &= K^T \, \frac{\pi^{T/2}(\sqrt{2rT})^T}{\Gamma(T/2 + 1)}  \leq (4\pi K^2er)^{T/2}.
\end{align}
Plugging the bounds from \eqref{eq:prob_bound1}-\eqref{eq:prob_bound2} into \eqref{eq:new_lambda_bound_2}, we have
\begin{align}
\nonumber \mathbb{P}\left(\lambda_{\min}(\Gamma_{i,T}) \leq r \right) 
&\leq C' |\log r|^{2d} r^{-d} (4\pi K^2 er)^{T/2} + r^{k + \delta + \gamma} \\
&= C'(4\pi K^2 e)^d|\log r|^{2d} (4\pi K^2 er)^{T/2 - d} + r^{k + \delta + \gamma}, \label{eq:new_lambda_bound_3}
\end{align}
where $C'$ is a constant that does not depend on $r$ and $T$. We aim to show \eqref{eq:anticoncentration_lambda_min} using the bound in \eqref{eq:new_lambda_bound_3}.
Letting $r \leq r_0:= \min\{\frac{1}{4\pi eK^2}, e^{-1}\}$, we have that $4\pi K^2 er \leq 1$, so $(4\pi K^2 er)^{T/2 - d}$ is decreasing in $T$. Now let $T_0 \geq 2d$ be some threshold, and then for $T \geq T_0$,
\[
(4\pi K^2 er)^{T/2 - d} \leq (4\pi K^2 er)^{T_0/2 - d},
\]
so we obtain
\begin{align}\label{eq:finalGammaLambBound}
    \mathbb{P}(\lambda_{\min}(\Gamma_{i,t}) \leq r) \leq C' (4\pi K^2e)^{T_0/2} |\log r|^{2d} r^{T_0/2 - d} + r^{k+\delta+\gamma},
\end{align}
and notice also that by \eqref{eq:finalGammaLambBound} we obtain
\[
\mathbb{P}(\lambda_{\min}(\Gamma_{i,T}) = 0) = \lim_{r \to 0^+} \mathbb{P}(\lambda_{\min}(\Gamma_{i,T}) \leq r) \leq \lim_{r \to 0^+}\left(C' (4\pi K^2e)^{T_0/2} |\log r|^{2d} r^{T_0/2 - d} + r^{k+\delta+\gamma} \right) = 0,
\]
so that the inverse exists almost surely for all  $T\geq T_0 > 2d$.

To conclude, using \eqref{eq:finalGammaLambBound} in \eqref{eq:integralInGammaiT}, we obtain, for any $T \geq T_0$,
\begin{align*}
    \mathbb{E}[\|\Gamma_{i,T}^{-1}\|^{k+\delta}] &\leq C_1 + C\int_0^{r_0} r^{-(k+\delta)-1} \left(C'(4\pi K^2 e)^d |\log r|^{2d} (4\pi K^2 er)^{T_0/2 - d} + r^{k + \delta + \gamma} \right)dr \\
    &\leq C_1 + C_2 + CC' (4\pi K^2 e)^{T_0 / 2}\int_0^{r_0} r^{-(k+\delta) - 1 + T_0/2 - d}|\log r|^{2d}dr,
\end{align*}
and this latter integral converges if and only if $-(k+\delta) - 1 + T_0/2 - d > -1$; in other words, $T_0 > 2(d + k + \delta)$. Thus, for any finite $T_0$ such that $T_0 > 2(d + k + \delta)$, we have as desired

\[
\sup_{i \in [n], T \geq T_0}\mathbb{E}\Big[\big\|\big(\frac{1}{T} X_i^\tr X_i \big)^{-1} \big\|^{k + \delta}\Big] < \infty.
\]
\end{proof}

\subsection{Proof of Theorem \ref{thm:currentconv}}
\label{app:thm:currentconv}
\begin{proof}        
        Part (a), the privacy guarantee, follows immediately from Theorem \ref{thm:trimest}. 
        
        For part (b), the error bound also follows from Theorem \ref{thm:trimest} by choosing the upper bounds  $B$, $K_1$ and $K_2$ as follows: $B =\Omega( \|\beta\|  +K_2)$ with
        \begin{equation}\label{eq:OLSK1}
        K_1 = O\Big(\sigma_\varepsilon\sqrt{\frac{1}{n} \sum_{i=1}^n\|\Gamma_i^{-1/2}\|^2} \sqrt{\frac{1}{nT} \big(d + \log \frac{1}{\xi}\big)} \Big), \; \text{ and } \; 
        K_2 = O\Big( \sigma_\varepsilon \max_{i \in [n]}\|\Gamma_i^{-1/2}\| \sqrt{\frac{1}{T}\big(d + \log \frac{n}{\xi}\big)} \Big).
        \end{equation}
        Indeed, Lemma \ref{lem:betasconcentration} shows that these choices of $B$, $K_1$ and $K_2$ are such that with probability at least $1-\xi$, they bound $\max_i\|\beta_i\|$, $\|\bar\beta-\beta\|$ and $\max_i\|\beta_i-\beta\|$, respectively. Hence, part (b) follows from the statement of Theorem \ref{thm:trimest} and the high probability guarantees that Lemma \ref{lem:betasconcentration} ensures for our proposed choices  of $B$, $K_1$ and $K_2$.

   Part (c) of the  theorem concerns the asymptotic distribution of the estimator. We prove this by showing that under the conditions stated in the theorem, (i) $\sqrt{nT}(\hat{\beta}_{\mathsf{DP}} - \bar{\beta}) \overset{p}{\to} 0$ and that (ii) $\sqrt{nT}(\bar{\beta} - \beta) \overset{d}{\to} N(0,   \lim_{n \to \infty} \frac{1}{n} \sum_{i=1}^n \Gamma_i^{-1} \Lambda_i \Gamma_i^{-1} )$. We then conclude by Slutsky's.

    We begin by proving (i). First of all, by Assumption \ref{assumption:UI}, given a fixed $(n,T)$ pair where $T \geq T_0$, we have that $\mathbb{P}(\forall i \in [n]: \text{Rank}(X_i) = d) = 1$. Now consider any diverging double sequence $(n_m, T_m) \to \infty$ indexed by $m \in \mathbb{N}$. Let $m_0 = \min\{ m: T_m \geq T_0\}$. Then $\mathbb{P}(\forall m \geq m_0, i \in [n_m] : \text{Rank}(X_i) = d) = 1$, since this is an intersection of countably many almost sure events, and hence remains almost sure. As such, $\mathbb{P}(\forall m \geq m_0, i \in [n_m]: \hat{\beta}_i = \hat{\beta}_{i,\mathsf{OLS}}) = 1$, and so in the remainder we may analyze the asymptotic properties of our estimator when $\hat{\beta}_i = \hat{\beta}_{i,\mathsf{OLS}} = (X_i^\tr X_i)^{-1} X_i^\tr Y_i$.
    
    {Now, notice that by Property 4 of Lemma \ref{lem:propertiesmean} and the value of $K_2$ from \eqref{eq:OLSK1}, we have, filling in $\bar{\beta}$ for $v$ in the statement of Property 4,
    \begin{align*}
    \sqrt{nT} \|\hat{\beta}_{\mathsf{DP}} - \bar{\beta}\| &\leq  \sqrt{nT} \| \bar{\beta} - \bar{\beta}\| + \frac{CK_2}{n\mu}\sqrt{R \log \frac{R}{\xi}} + \frac{CK_2}{n\mu} \sqrt{d + \log \frac{1}{\xi}}  \\
    &= 0 + O\Big(\frac{\sigma_\varepsilon}{\sqrt{n}\mu} \max_{i \in [n]} \|\Gamma_i^{-\frac{1}{2}}\| \sqrt{(d + \log \frac{n}{\xi})(d + R \log\frac{R}{\xi})}\Big).
    \end{align*}}
    Since the $\|\Gamma_i^{-1/2}\|$ are uniformly bounded (Assumption \ref{assumptions:initialcovariates}), we will have that $\sqrt{nT}(\hat{\beta}_{\mathsf{DP}} - \bar{\beta}) \overset{p}{\to} 0$ if $\xi \to 0$ while $\frac{1}{ n\mu^2}(R \log \frac{R}{\xi} + \log\frac{1}{\xi})(\log \frac{n}{\xi})  \to 0.$
    Because $R \geq 1$, under the growth conditions stated in \eqref{eq:growthconditions}, we have that 
    \[\frac{1}{ n\mu^2}\Big(R \log \frac{R}{\xi} + \log\frac{1}{\xi}\Big)\Big(\log \frac{n}{\xi}\Big) \leq \frac{2}{ n}\Big(R \log \frac{R}{\xi} \Big)\Big(\log \frac{n}{\xi}\Big)  = o(1). \]
    Therefore, $\sqrt{nT}(\hat{\beta}_{\mathsf{DP}} - \bar{\beta}) \overset{p}{\to} 0$.

    We now proceed to prove (ii). We have 
    \[
    \bar{\beta} = \frac{1}{n}\sum_{i=1}^n (X_i^\tr   X_i)^{-1} X_i^\tr   Y_i = \beta + \frac{1}{n}\sum_{i=1}^n (X_i^\tr   X_i)^{-1} X_i^\tr   \varepsilon_i.
    \]
    Define $\Gamma_{i,T} := \frac{1}{T} X_i^\tr X_i$, and $S_{i,T} := \frac{1}{\sqrt T}X_i^\tr \varepsilon_i$. Then, we have
    \begin{equation}
    \label{eq:decomp1}
    \sqrt{nT}(\bar{\beta} - \beta) = {\frac{1}{\sqrt n}}\sum_{i=1}^n \Gamma_{i,T}^{-1} S_{i,T} = {\frac{1}{\sqrt n}}\sum_{i=1}^n \Gamma_i^{-1} S_{i,T} + {\frac{1}{\sqrt n}}\sum_{i=1}^n (\Gamma_{i,T}^{-1} - \Gamma_i^{-1}) S_{i,T},
    \end{equation}
    and we first show that the second term on the right side of \eqref{eq:decomp1} is $o_p(1)$. We have
\begin{align}
\label{eq:Ebound1}
  \mathbb{E}\big[\|\frac{{1}}{\sqrt{n}}\sum_{i=1}^n (\Gamma_{i,T}^{-1} - \Gamma_i^{-1})S_{i,T}\|^2\big] &= \frac{1}{n} \sum_{i=1}^n\sum_{j=1}^n \mathbb{E}[S_{i,T}^\tr (\Gamma_{i,T}^{-1} - \Gamma_i^{-1} ) (\Gamma_{j,T}^{-1} - \Gamma_j^{-1}) S_{j,T}]
  \\ &\label{eq:Ebound1_1} = \frac{1}{n}\sum_{i=1}^n \mathbb{E}[\|(\Gamma_{i,T}^{-1} - \Gamma_i^{-1})S_{i,T}\|^2]=\frac{1}{n}\sum_{i=1}^n   \mathbb{E}[ \|\Gamma_{i}^{-1}(\Gamma_{i} - \Gamma_{i,T})\Gamma_{i,T}^{-1}S_{i,T}\|^2],
\end{align}
 where we use the cross-sectional independence of $X_i$ and $\varepsilon_i$ across $i$ (Assumption \ref{assumptions:initialcovariates} and  \ref{assumptions:initialerrors}) and strict exogeneity to obtain our second equality, since under these assumptions $\mathbb{E}[S_{i,T}^\tr (\Gamma_{i,T}^{-1} - \Gamma_i^{-1})(\Gamma_{j,T}^{-1} - \Gamma_j^{-1}) S_{j,T}] = 0$ for $i\neq j$. Then, from the uniform boundedness of $ \|\Gamma_i^{-1}\|$ (Assumption \ref{assumptions:initialcovariates}) and Holder's Inequality, we can proceed from \eqref{eq:Ebound1_1},
 \begin{align}
 \nonumber \sup_{i \geq1} \mathbb{E}[ \|\Gamma_{i}^{-1} & (\Gamma_{i} - \Gamma_{i,T})\Gamma_{i,T}^{-1}S_{i,T}\|^2] \leq  \sup_{i \geq 1}\|\Gamma_i^{-1}\|^2 \cdot \sup_{i\geq1}\mathbb{E}\left[\|\Gamma_{i,T}^{-1}\|^2 \cdot \|\Gamma_{i,T} - \Gamma_i\|^2 \cdot \|S_{i,T}\|^2 \right]  \nonumber \\
 &\leq C\sup_{i\geq1} \left(\mathbb{E}[\|\Gamma_{i,T}^{-1}\|^{2 + \delta}] \right)^{\frac{2}{2+\delta}}\sup_{i\geq1}\left(\mathbb{E} \left[\|\Gamma_{i,T} - \Gamma_{i}\|^{\frac{4(2+\delta)}{\delta}}\right]\right)^\frac{\delta}{2(2+\delta)}\sup_{i\geq1}\left(\mathbb{E}\left[\|S_{i,T}\|^{\frac{4(2+\delta)}{\delta}}\right] \right)^\frac{\delta}{2(2+\delta)}.\label{eq:Ebound2}
\end{align}
Now considering the terms on the right side of \eqref{eq:Ebound2}, by the uniform integrability of  $\Gamma_{i,T}$, Assumption \ref{assumption:UI},  we have $\sup_{i\geq1}\mathbb{E}[\|\Gamma_{i,T}^{-1}\|^{2 + \delta}]  \leq C$ for some constant $C>0$. Moreover, $\sup_{i\geq1}(\mathbb{E}\|\Gamma_{i,T} - \Gamma_i\|^p)^{2/p} = O(T^{-1})$ and $\sup_{i\geq1}(\mathbb{E}\|S_{i,T}\|^p)^{2/p} = O(1)$ by Lemma \ref{auxres:GammaITExpectationBound}. Putting this together in \eqref{eq:Ebound1}-\eqref{eq:Ebound2}, we have
 \begin{align}
 \mathbb{E}\Big[\big\|\frac{{1}}{\sqrt{n}} \sum_{i=1}^n (\Gamma_{i,T}^{-1} - \Gamma_i^{-1})S_{i,T}\big\|^2\Big]  \label{eq:OLSSmallORder}   &\leq \sup_{i \geq1} \mathbb{E}[ \|\Gamma_{i}^{-1}  (\Gamma_{i} - \Gamma_{i,T})\Gamma_{i,T}^{-1}S_{i,T}\|^2]  = O\Big(\frac{1}{T}\Big).
\end{align}
     Since $T \rightarrow \infty$, using \eqref{eq:decomp1}, we have therefore shown that
    \[
    \sqrt{nT}(\bar{\beta} - \beta) = \frac{{1}}{\sqrt{n}}\sum_{i=1}^n \Gamma_i^{-1} S_{i,T} + o_p(1).
    \]

    Thus,  for our asymptotic covariance $V = \lim_{n \to \infty} \frac{1}{n} \sum_{i=1}^n \Gamma_i^{-1} \Lambda_i \Gamma_i^{-1}$, it remains to  prove that
    \[
    \frac{{1}}{\sqrt{n}}\sum_{i=1}^n \Gamma_i^{-1} S_{i,T}=\frac{1}{\sqrt{nT}}\sum_{i=1}^n \Gamma_i^{-1}  X_i^\tr \varepsilon_i = \frac{1}{\sqrt{nT}}\sum_{i=1}^n \sum_{t=1}^T \Gamma_i^{-1} x_{i,t} \varepsilon_{i,t} \overset{d}{\to} N(0, V).
    \]
    Fix an arbitrary $v \in \mathbb{S}^{d-1}$ and let $z_{i,t}(v) := v^\tr \Gamma_i^{-1} x_{i,t} \varepsilon_{i,t} $ and then we show that
    \begin{equation}
    \label{eq:CLT1}
    \frac{1}{\sqrt{nT}}\sum_{i=1}^n \sum_{t=1}^T z_{i,t}(v) \overset{d}{\to} N(0, v^\tr V v),
    \end{equation}
    and conclude our result by Cram\'er-Wold. To that end, let $b_{i,T}(v) := \frac{1}{\sqrt{T}} \sum_{t=1}^T z_{i,t}(v)$, and $S_{n,T} := \frac{1}{\sqrt{n}}\sum_{i=1}^n b_{i,T}(v)$.  With this notation, \eqref{eq:CLT1} becomes $ S_{n,T} \overset{d}{\to} N(0, v^\tr V v)$ as $(n,T) \rightarrow \infty$.

    We will apply  the Lindeberg-Feller CLT for triangular arrays to the array $\{\frac{1}{\sqrt{n}}b_{i,T}(v)\}$ where we take both $T$ and $n$ to infinity. In particular, letting $m$ denote the row of the array, we pick an arbitrary sequence $(n_m, T_m)$ that diverges in both coordinates. Within each row, the elements will be $\frac{1}{\sqrt{n_m}}b_{i, T_m}(v)$ so that the row sums are $S_m(v) := \frac{1}{\sqrt{n_m}}\sum_{i=1}^{n_m} b_{i,T_m}(v) = S_{n_m, T_m}(v)$. Since $n_m$ is diverging, this is a triangular array. In particular, note that within a row, all elements are independent by our assumption of cross-sectional independence. 

    We proceed with finding the asymptotic variance of $S_m(v)$, using cross-sectional independence, the mean-zero property of $S_m(v)$ by exogeneity and mean-zero errors, and second order stationarity of both processes. We have, 
    \begin{align}
   \nonumber V_m(v)^2 := \V(S_m(v)) = \frac{1}{n_m}\sum_{i=1}^{n_m} \V(b_{i,T_m}(v)) &=\frac{1}{n_m}\sum_{i=1}^{n_m} \V\Big(\frac{1}{\sqrt{T_m}}\sum_{t=1}^{T_m}z_{i,t}(v)\Big) \\
   \nonumber &= \frac{1}{n_m}\sum_{i=1}^{n_m} \V\Big(\frac{1}{\sqrt{T_m}}\sum_{t=1}^{T_m}v^\tr \Gamma_i^{-1} x_{i,t} \varepsilon_{i,t}\Big) \\
    &\label{eq:pfbeforeToeplitz} 
    = v^\tr \frac{1}{n_mT_m}\sum_{i=1}^{n_m}\Gamma_i^{-1}  \sum_{t=1}^{T_m} \sum_{s=1}^{T_m} \mathbb{E}[x_{i,t} \varepsilon_{i,t}\varepsilon_{i,s} x_{i,s}^\tr]\Gamma_i^{-1} v,
    \end{align}
    and since $(x_{i,t} \varepsilon_{i,t})_{t\in \mathbb{Z}}$ is assumed to be weakly stationary and by strict exogeneity, we have %
    \begin{align}
    \sum_{t=1}^{T_m} \sum_{s=1}^{T_m}
    \mathbb{E}[x_{i,t} \varepsilon_{i,t}  \varepsilon_{i,s}x_{i,s}^{\tr}] = \sum_{t=1}^{T_m} \sum_{s=1}^{T_m}
    \mathbb{E}[x_{i,0} \varepsilon_{i,0}  \varepsilon_{i, |t-s|}x_{i,|t-s|}^{\tr}]  
    &=\label{eq:pfrewriteToeplitz} \sum_{|\ell| < T_m} (T_m - |\ell|)\mathbb{E}[x_{i,0} x_{i,\ell}^\tr \varepsilon_{i,0} \varepsilon_{i,\ell}],
    \end{align}
    and hence we can apply \eqref{eq:pfrewriteToeplitz} to \eqref{eq:pfbeforeToeplitz} to obtain
    \begin{align*}
    V_m(v)^2&= v^\tr \frac{1}{n_m T_m} \sum_{i=1}^{n_m}\Gamma_i^{-1} \sum_{|\ell| < T_m}(T_m - |\ell|)\mathbb{E}[x_{i,\ell} x_{i,0}^\tr\varepsilon_{i,\ell} \varepsilon_{i,0}] \Gamma_i^{-1} v \\
    &= v^\tr \frac{1}{n_m}\sum_{i=1}^{n_m} \Gamma_i^{-1}  \sum_{\ell = -\infty}^\infty \mathbb{E}[x_{i,\ell} x_{i,0}^\tr\varepsilon_{i,\ell} \varepsilon_{i,0}]  \Gamma_i^{-1} v + v^\tr \frac{1}{n_m} \sum _{i=1}^{n_m} \Gamma_i^{-1} R_{i, T_m} \Gamma_i^{-1} v\\
    &\overset{}{\to} v^\tr \left( \lim_{n \to \infty} \frac{1}{n} \sum_{i=1}^n \Gamma_i^{-1} \Lambda_i \Gamma_i^{-1} \right)v=v^\tr V v,
    \end{align*}
    where we defined $R_{i,T_m} = \sum_{|\ell| < T_m}(1 - \frac{|\ell|}{T_m})\mathbb{E}[x_{i,\ell} x_{i,0}^\tr \varepsilon_{i,\ell} \varepsilon_{i,0}] - \sum_{\ell = -\infty}^\infty \mathbb{E}[x_{i,\ell} x_{i,0}^\tr \varepsilon_{i,\ell} \varepsilon_{i,0}]$ and $\Lambda_i := \sum_{\ell=-\infty}^\infty \mathbb{E}[x_{i,0} x_{i,\ell}^\tr\varepsilon_{i,0}\varepsilon_{i,\ell}]$. In particular, the last limit used
    \begin{align*}
    \|R_{i,T_m}\| &= \Big\|\sum_{|\ell| < T_m}\Big(1 - \frac{|\ell|}{T_m}\Big)\mathbb{E}[x_{i,\ell} x_{i,0}^\tr \varepsilon_{i,\ell} \varepsilon_{i,0}] - \sum_{\ell = -\infty}^\infty \mathbb{E}[x_{i,\ell} x_{i,0}^\tr \varepsilon_{i,\ell} \varepsilon_{i,0}]\Big\| \\ &\leq \sum_{|\ell | > T_m}\|\mathbb{E}[x_{i,\ell} x_{i,0}^\tr \varepsilon_{i,\ell} \varepsilon_{i,0}]\| + \frac{1}{T_m} \sum_{|\ell| < T_m}|\ell| \|\mathbb{E}[x_{i,\ell}x_{i,0}^\tr \varepsilon_{i,\ell} \varepsilon_{i,0}]\| = o(1).
    \end{align*}
    This last equality holds uniformly in $i$ when $T_m \to \infty$ by our mixing assumption. In particular, by Lemma \ref{lem:measurablemixing}, we have that the mean-zero process $(x_{i,t} \varepsilon_{i,t})$ is strongly mixing with strong mixing coefficients decaying as $\alpha_{score}(k) = O(k^{-3-q})$. 
    This combined with Lemma \ref{auxres:davydov} implies that $\sum_{|\ell|=-^\infty}^\infty |\ell| \|\mathbb{E}[x_{i,l} x_{i,0}^\tr \varepsilon_{i,l} \varepsilon_{i,0}]\| <\infty$, so that the second term converges to zero as $T_m \to \infty$. In the same vein, $\sum_{\ell= -\infty}^\infty \|\mathbb{E}[x_{i,l} x_{i,0}^\tr \varepsilon_{i,\ell}\varepsilon_{i,0}]\| < \infty$ immediately implies that $\sum_{|\ell | > T_m}\|\mathbb{E}[x_{i,\ell} x_{i,0}^\tr \varepsilon_{i,\ell} \varepsilon_{i,0}]\|  \to 0$ as $T_m \to \infty$. 
    
     Next we verify the Lindeberg condition, i.e. that, for all $\epsilon>0$,
\begin{equation}
\label{eq:Lindeberg-Feller}
      \lim_{m \to \infty} \sum_{i=1}^{n_m}\frac{1}{n_mV_m(v)^2}\mathbb{E}\left[b_{i,T_m}^2\mathbbm{1}\{|b_{i,T_m}(v)| \geq \sqrt{n_m}\epsilon V_m(v)\}\right] = 0.
\end{equation}
   Using that for any $0< \delta<2$  we have $|x|^2 \mathbbm{1}\{|x| > a\} \leq \frac{|x|^{2+\delta}}{a^\delta }$, we get
    \begin{equation}
    \label{eq:Lindeberg-Feller1}
        \mathbb{E}\left[b_{i,T_m}^2\mathbbm{1}\{|b_{i,T_m}(v)| \geq \sqrt{n_m}\epsilon V_m(v)\}\right] \leq \frac{\mathbb{E}\left[|b_{i,T_m}|^{2+\delta}\right]}{(\sqrt{n_m}\epsilon V_m(v))^\delta}.
    \end{equation}
    Furthermore,  by Lemma \ref{auxres:momentstrongmixing}, we have 
    \begin{equation}
        \label{eq:Lindeberg-Feller2}
    \sup_{i,T}\mathbb{E}[|b_{i,T}(v)|^{2+\delta}] = \sup_{i,T}\mathbb{E}\Big[ \frac{1}{T^{1+\delta/2}} \Big|\sum_{t=1}^T v^\tr   {\Gamma_i^{-1}}x_{i,t} \varepsilon_{i,t}\Big|^{2+\delta}\Big]=\sup_{i,T}\frac{1}{T^{1+\delta/2}} O(T^{1 + \delta/2}) = O(1).
    \end{equation}
   The conditions of Lemma \ref{auxres:momentstrongmixing} are satisfied because by Lemma \ref{lem:measurablemixing}, the mean-zero process $(x_{i,t} \varepsilon_{i,t})$ is strongly mixing with strong mixing coefficients  $\alpha_{score_i}(k)$ satisfying
    \begin{align}\label{eq:scoremixing}
    \sup_{i \geq 1}\sum_{k=1}^\infty (k+1)^{c-2} \alpha_{score_i}(k)^{\frac{\epsilon}{c + \epsilon}} < \infty,
    \end{align}
    for some $\epsilon>0, c \in 2\mathbb{N}, c>2 + \delta$.  In particular, since $\delta <2$, if we take $c = 4$, then we have
    \[
    \sum_{k=1}^\infty (k+1)^{2} \alpha_{score_i}(k)^{\frac{\epsilon}{4+\epsilon}} \leq C\sum_{k=1}^\infty k^2 k^{(-3-q)\frac{\epsilon}{4 + \epsilon}},
    \]
    and this is finite for any $\epsilon > \frac{12}{q}$. 

    We can therefore conclude from \eqref{eq:Lindeberg-Feller},\eqref{eq:Lindeberg-Feller1} and \eqref{eq:Lindeberg-Feller2} that
      \begin{align}\label{eq:lindebergverif}
     \sum_{i=1}^{n_m}\frac{1}{n_mV_m(v)^2}\mathbb{E}\left[b_{i,T_m}^2\mathbbm{1}\{|b_{i,T_m}(v)| \geq \sqrt{n_m}\epsilon V_m(v)\}\right]= \frac{O(1)}{n_m^{\delta/2}\epsilon^\delta V_m(v)^{2+\delta}} \to 0,
    \end{align}
    since $n_m \to \infty$, $\epsilon = O(1)$, and $V_m(v)^2 \to v^\tr ( \lim_{n \to \infty} \frac{1}{n} \sum_{i=1}^n \Gamma_i^{-1} \Lambda_i \Gamma_i^{-1}) v >0$. Thus, we have verified the Lindeberg condition and conclude that $\frac{S_m(v)}{V_m(v)} \overset{d}{\to} N(0,1)$. Since this is true for every sequence $(n_m, T_m)_{m\geq 1}$ for which $n_m \to \infty$, $T_m \to \infty$, it follows that, using Slutsky's Theorem,
    \[
    \frac{1}{\sqrt{nT}}\sum_{i=1}^{n} \sum_{t=1}^{T} z_{i,t}(v)\overset{d}{\to} N(0, v^\tr   V v),
    \]
    as $(n,T) \to \infty$.
    By Cramer-Wold, it finally follows that
    \[
    \frac{1}{\sqrt{nT}}\sum_{i=1}^{n} \Gamma_i^{-1} X_i^\tr \varepsilon_i\overset{d}{\to} N(0,  V),
    \]
    as $(n,T) \to \infty$, and consequently, by Slutsky's theorem
    \[
    \sqrt{nT}(\bar{\beta} - \beta) \overset{d}{\to} N\left(0,  V\right) \quad \mbox{and} \quad \sqrt{nT}(\hat{\beta}_{\mathsf{DP}} - \beta) \overset{d}{\to} N\left(0,  V\right).
    \]
\end{proof}

\subsection{Proof of Theorem \ref{thm:diffconv}}
\label{app:thm:diffconv}
We begin by proving the privacy and utility guarantee of Algorithm \ref{alg:diffmeanest}, which is the workhorse of Section \ref{subsec:twogroupreg}. In particular, Theorem \ref{thm:trimdiffest} will allow us to derive the guarantees of Theorem \ref{thm:diffconv}.
\begin{algorithm}
\caption{DP Treated mean estimation \hfill \textsc{DPTreatMean}$((D_i, z_i)_{i=1}^n,\mu,B,R,\xi)$}
\label{alg:diffmeanest}
\begin{algorithmic}[1]
\STATE \textbf{Input:} Data $(D_i, z_i)_{i=1}^n$, privacy parameter $\mu$, radius $B$, max rounds $R$, failure probability\ $\xi$
\STATE \textbf{Noisy size:} $\tilde z \gets \sum_{i=1}^n z_i + \tfrac{2}{\mu} Z_0$, \; $Z_0 \sim \mathcal N(0,1)$\label{lineDiff:numUsers} 
\STATE \textbf{Set:} $\tau \gets \max\!\Big\{ \tilde z - \frac{2\sqrt{R}}{\mu}\sqrt{2\log\!\tfrac{4R}{\xi}} - \frac{2}{\mu} \sqrt{2\log \frac{8}{\xi}}
,\, 1\Big\}$, $n_{LB} \gets \max\!\Big\{ \tau - \tfrac{2\sqrt{R}}{\mu}\sqrt{2\log\!\tfrac{4R}{\xi}},\, 1\Big\}$
\STATE \label{line:TreatedAlgLine4} \textbf{Init:} $\hat{m}_{r-2} \gets 0$, $\hat{m}_{r-1} \gets 0$ $r \gets 0$, $r^* \gets -1$
\WHILE{$r \le R$ and $r^* = -1$}
  \STATE $\omega_i \gets \mathbbm{1}\{\|D_i - \hat m_{r-1}\| < B z_i / 2^r\}$, \quad
         $\omega \gets \sum_{i=1}^n z_i \omega_i + \tfrac{2\sqrt{R}}{\mu} Z_r$, \; $Z_r \sim \mathcal N(0,1)$
  \IF{$\omega < \tau$}\label{lineDiff:countcheck}
     \STATE $S \gets \{ i : \|D_i - \hat m_{r-2}\| < B z_i / 2^{\,r-1} \}$, \quad $n_{\tilde{S}} \gets \max(|S|, n_{LB})$, \quad $\mathcal{C} \gets \sqrt{3 - \frac{2r}{R}}$,
     \STATE $\hat m \gets \hat m_{r-2}
       + \frac{1}{n_{\tilde{S}}}\sum_{i \in S} (D_i - \hat m_{r-2})
       + \frac{4B}{2^{\,r-1}\mathcal{C}\mu n_{LB}} N_*$, \; $N_* \sim \mathcal N(0,I)$\label{lineDiff:terminationEarly}
           \STATE $r^* \gets r-1$, \quad $B^* \gets \frac{4B}{2^{r^*}\mathcal{C} \mu n_{LB}}$

  \ELSIF{$r=R$}
     \STATE $S \gets \{ i : \|D_i - \hat m_{r-1}\| < B z_i / 2^{\,R} \}$, \quad $n_{\tilde{S}} \gets \max(|S|, n_{LB})$
     \STATE $\hat m \gets \hat m_{r-1}
       + \frac{1}{n_{\tilde{S}}}\sum_{i \in S} (D_i - \hat m_{r-1})
       + \frac{4B}{2^{\,R}\mu n_{LB}} N_*$\label{lineDiff:terminationR}, \quad $N_* \sim N(0, I)$
           \STATE $r^* \gets R$, \quad $B^* \gets \frac{4B}{2^{r^*}\mu n_{LB}}$, \quad 
  \ELSE
     \STATE $S \gets \{ i : \|D_i - \hat m_{r-1}\| < B z_i / 2^{\,r} \}$, \quad $n_{\tilde{S}} \gets \max(|S|, n_{LB})$
     \STATE $\hat m_r \gets \hat m_{r-1}
       + \frac{1}{n_{\tilde{S}}}\sum_{i \in S} (D_i - \hat m_{r-1})
       + \frac{4B\sqrt{R}}{2^{\,r}\mu n_{LB}} N_r$\label{lineDiff:updateCenter}, \quad $N_r \sim N(0,I)$
     \STATE $r \gets r+1$
  \ENDIF
\ENDWHILE
\STATE \textbf{Output:} $(\hat{m}, \hat{m}_{r^* - 1}, r^*, n_{LB}, B^*)$
\end{algorithmic}
\end{algorithm}

\begin{theorem}
\label{thm:trimdiffest}
   Algorithm \ref{alg:diffmeanest} satisfies $\mu$-GDP. Furthermore, assume there exist $B, K_1, K_2>0$ such that $\max_i \|z_iD_i\| \leq B$, $\|\bar{D}_{z=1} - m\| \leq K_1$ and $\max_{i:z_i = 1}\|{D}_i - m\| \leq K_2$, and assume $\sum_{i=1}^n z_i =: \bar{z} \asymp n$. If $R \geq C \log \frac{B}{K_2}$  and  $n\geq C'(\frac{1}{\mu}\sqrt{R(\log\frac{R}{\xi}+ d)})$ for universal constants $C, C'>0$, then Algorithm \ref{alg:diffmeanest}  satisfies, with probability at least $1-\xi$,
    \[
    \|\hat{m}_{\mathsf{DP}} - m\| = O\Big(K_1 +  \frac{K_2}{n\mu}\sqrt{{R \log \frac{R}{\xi}}{}} + \frac{K_2}{n\mu}\sqrt{{d + \log \frac{1}{\xi}}{}}  \Big).
    \]
\end{theorem}

\begin{proof}
    For simplicity we will sometimes use the labels $\mu_1 = \mu_2 = \frac{\mu}{\sqrt{4R}}$ and $\mu_3 = \frac{\mu}{2}$.
    
    The privacy proof is similar to the proof of Proposition \ref{prop:trimprivacy}. Indeed, the final output is a function of the data only through the statistics $\sum_{i=1}^n z_i$, $\sum_{i=1}^n z_i \omega_i$, and $\frac{1}{n_{\tilde{S}}}\sum_{i \in S} (D_i - \hat{m}_{r-j})$, such that if we prove that these statistics are sufficiently private, our privacy results follows by composition (Proposition \ref{prop:composition}). 

    To begin, since $z_i \in \{0,1\}$ it is clear that the sensitivity of $\sum_{i=1}^n z_i$ is bounded by $1$, so that the statistic $\tilde{z}$ satisfies $\mu_3$-GDP. The same is true for the statistics $\sum_{i=1}^n z_i \omega_i$, since $z_i \omega_i \in \{0,1\}$, such that each instance of Line \ref{lineDiff:countcheck} satisfies $\mu_2$-GDP. 
    
    Finally, we consider the sensitivity of
    $\frac{1}{n_{\tilde{S}}}\sum_{i \in S}z_i(D_i - \hat{m}_{r-j}),$
    where $n_{\tilde{S}} = \max(|S|, n_{LB})$, and where $j=1$ in Lines \ref{lineDiff:terminationR} and \ref{lineDiff:updateCenter} and $j = 2$ in Line \ref{lineDiff:terminationEarly}. Let $(x_i, z_i)_{i=1}^n$ and $(x_i', z_i')_{i=1}^n$ be two neighboring datasets where without loss of generality we assume that only the first entry differs between the two datasets. To derive the sensitivity of our statistic, we must consider two cases. The first case to consider is when $z_1 = z_1'$. In this case the analysis is identical to the analysis in Algorithm \ref{alg:trimestTemp}, and so we immediately know that in this case the sensitivity is bounded by $\frac{2B}{2^{r-(j-1)}n_{LB}}$. Hence, we consider in depth the case where $z_1' = z_1 + 1$ (sensitivity is symmetric so we do not need to consider $z_1' = z_1 - 1$). If $\|D_1' - \hat{m}_{r-j}\| > \frac{B}{2^{r}},$
    then $S = S'$, $n_{\tilde{S}} = n_{\tilde{S}}'$, such that the difference is $0$. If, however, 
    $\|D_1' - \hat{m}_{r-j}\| \leq \frac{B}{2^r},$
    then $S \neq S'$. Now, there are two cases to consider. First, if $|S'| \leq n_{LB}$, then we have that $n_{\tilde{S}} = n_{\tilde{S}}'$, such that we can bound the sensitivity as
    \[
    \Big\|\frac{1}{n_{\tilde{S}}}\sum_{i \in S} z_i(D_i - \hat{m}_{r-j}) - \frac{1}{n_{\tilde{S}}'} \sum_{i \in S'} z_i'(D_i' - \hat{m}_{r-j})\Big\| = \Big\|\frac{1}{n_{\tilde{S}}} (D_1' - \hat{m}_{r-j})\Big\|  \leq \frac{B}{2^{r-(j-1)}n_{LB}}.
    \]
    Otherwise, if $|S'| > n_{LB}$, then $n_{\tilde{S}} \neq n_{\tilde{S}}'$. In this case,
    \begin{align*}
   & \Big\|\frac{\sum_{i \in S} z_i(D_i - \hat{m}_{r-j})}{n_{\tilde{S}}} - \frac{\sum_{i \in S'} z_i'(D_i' - \hat{m}_{r-j})}{n_{\tilde{S}}'} \Big\| \\
   &\qquad = \Big\|\frac{(n_{\tilde{S}} + 1)\sum_{i \in S} z_i(D_i - \hat{m}_{r-j}) - n_{\tilde{S}}\sum_{i \in S'}z_i'(D_i' - \hat{m}_{r-j})}{n_{\tilde{S}}(n_{\tilde{S}} + 1)}\Big\| \\  
    &\qquad \leq \Big\|\frac{\sum_{i \in S} z_i(D_i - \hat{m}_{r-j}) - n_{\tilde{S}}(D_1' - \hat{m}_{r-j})}{n_{\tilde{S}}(n_{\tilde{S}} + 1)}\|.
    \end{align*}
    To complete the proof we notice that
        \begin{align*}
     \Big\|\frac{\sum_{i \in S} z_i(D_i - \hat{m}_{r-j}) - n_{\tilde{S}}(D_1' - \hat{m}_{r-j})}{n_{\tilde{S}}(n_{\tilde{S}} + 1)}\Big\| &\leq \Big\|\frac{n_{\tilde{S}}B}{n_{\tilde{S}}(n_{\tilde{S}} + 1) 2^{r-(j-1)}}\Big\| + \Big\|\frac{n_{\tilde{S}}B}{n_{\tilde{S}}(n_{\tilde{S}} + 1)2^{r-(j-1)}}\Big\| \\
    &\leq \frac{2B}{(n_{\tilde{S}} + 1)2^{r-(j-1)}} \leq \frac{2B}{(n_{LB}+1)2^{r-(j-1)}} \leq \frac{2B}{2^{r-(j-1)}n_{LB}},    \end{align*}
    which gives the sensitivity of this quantity, as all other cases follow analogously. 

    Now, if $r^* = R$, then the final output is a function of the data through  one call to Line \ref{lineDiff:numUsers}, one call to Line \ref{lineDiff:terminationR}, $R$ calls to Line \ref{lineDiff:countcheck}, and $R$ calls to Line \ref{lineDiff:updateCenter}, and we established that these satisfy $\mu_3$, $\mu_3$, $\mu_3$, $\mu_2$, and $\mu_1$-GDP respectively. Therefore, it follows that $(\hat{m}, r^*)$ satisfies $\sqrt{R(\mu_1^2 + \mu_2^2) + 2\mu_3^2}$-uGDP. If, however, the algorithm terminates with $r^* < R$, then the algorithm's output is a function of the data through one call to Line \ref{lineDiff:numUsers}, one call to Line \ref{lineDiff:terminationEarly} with privacy parameter $\mu_3\mathcal{C}$, $r^*+1$ calls to Line \ref{lineDiff:countcheck}, and $r^* + 1$ calls to Line \ref{lineDiff:updateCenter}. Therefore it follows that in this case $(\hat{m}, r^*)$ satisfies $\sqrt{(1 + \mathcal{C}^2)\mu_3^2 + (r^* + 1)(\mu_1^2 + \mu_2^2)}$-uGDP, and it is simple to check that this equals $\sqrt{R(\mu_1^2 + \mu_2^2) + 2\mu_3^2}$ for our choice of $\mathcal{C}$.
    
    We now proceed with the proof of the utility guarantee, which is similar to that of the base algorithm, with some minor tweaks to handle the privatization of the inclusion variable $z_i$. Let us first define the high probability event under which our analysis in Theorem \ref{thm:trimest} translates well to Algorithm \ref{alg:diffmeanest}. We let 
    \[
    \mathcal{E}_0 = \Big\{ |Z_0| \leq \sqrt{2 \log \frac{8}{\xi}} \Big\},
    \]
    which by Lemma \ref{suplemma:max_subG} we know holds with probability at least $1-\xi/4$. Furthermore,
    \[
    \mathcal{E}_1 = \Big\{\max_{r \in [R]} |Z_r| \leq \sqrt{2\log\frac{4R}{\xi}} \Big\},
    \]
    which also by Lemma \ref{suplemma:max_subG} we know holds with probability at least $1-\frac{\xi}{2}$. Moreover,
    \[
    \mathcal{E}_2 = \Big\{\max_{r \in 0, \dots, R-1} \|N_r\| \leq \sqrt{d} + \sqrt{2 \log \frac{8(R+1)}{\xi}}\Big\},
    \]
    which by Lemma \ref{suplemma:maxnormals} we know holds with probability at least $1-\frac{\xi}{8}$, and
    \[
    \mathcal{E}_3 = \Big\{\|N_{*} \| \leq \sqrt{d} + \sqrt{2\log \frac{8}{\xi}} \Big\},
    \] which holds with probability at least $1-\frac{\xi}{8}$. Hence, the event \begin{align}\label{eq:DiffMeanHighProbEvent}\mathcal{E} :=  \mathcal{E}_0 \cap \mathcal{E}_1 \cap \mathcal{E}_2 \cap \mathcal{E}_3\end{align} holds with probability at least $1-\xi$, and we analyze Algorithm \ref{alg:diffmeanest} on this event. 
    
    Let $\mathcal{Z} := \{i : z_i = 1\}$, and $\bar{D}_\mathcal{Z} := \frac{1}{\bar{z}} \sum_{i \in \mathcal{Z}}D_i$. Broadly, Algorithm \ref{alg:diffmeanest} coincides with Algorithm \ref{alg:trimestGeneral} applied to $(D_i)_{i \in\mathcal{Z}}$, with the only difference being the replacement of the true sample size $\bar{z}$ by the noisy sample size $\tilde{z}$, and consequent changes to $\tau$ and $n_{LB}$. Note that it is only through changes to $\tau$ and $n_{LB}$ that the output of Algorithm \ref{alg:diffmeanest} will differ from the output of Algorithm \ref{alg:trimestGeneral} with input $(D_i)_{i \in \mathcal{Z}}$. However, on our new high probability event $\mathcal{E}$, we still have that if $\sum_{i=1}^n \omega_i = \bar{z}$, then $\omega \geq \tau$. Then, as in \eqref{eq:Di_bound}, we obtain, for $i\in \mathcal{Z}$,
    \[
    \|D_i - \hat{m}_{r-1}\| \leq \|D_i - \bar{D}_\mathcal{Z}\| + \frac{4B\sqrt{R}}{2^{r-1}\mu n_{LB}}\|N_{r-1}\| ,
    \]
    so, following the argument around \eqref{eq:Di_bound}, we obtain that the algorithm will not terminate as long as $2K_2 \leq B/2^{r+1}$ and $\bar{z} \geq C\frac{1}{\mu}\sqrt{R (\log \frac{R}{\xi}+d)}$. To see this, denote $n_{LB_1}$ the value of $n_{LB}$ when applying Algorithm \ref{alg:trimestTemp} to $(D_i)_{i\in \mathcal{Z}}$, and $n_{LB_2}$ the value of $n_{LB}$ obtained when applying Algorithm \ref{alg:diffmeanest} to $(D_i, z_i)_{i=1}^n$. Then,    \begin{align}\label{eq:LTwoGroupLOneGroup}n_{LB_2} =  n_{LB_1} + \frac{1}{\mu_3} \Big(Z_0  -\sqrt{2\log \frac{8}{\xi}}\Big),\end{align} where we assume $\sum_{i=1}^n z_i$ is large enough that these quantities are larger than $1$ in our high probability analyses. Under our assumptions,  on $\mathcal{E}$ we have 
    \[
    \Big|\frac{1}{\mu_3} \Big(Z_0  -\sqrt{2\log \frac{8}{\xi}}\Big)\Big| \leq \frac{2}{\mu_3} \sqrt{2 \log \frac{8}{\xi}},
    \]
    such that on this high-probability event, the requirement on $\bar{z}$ is of the same order as in  \eqref{eq:alg1ReqonN}  in the setting of Algorithm \ref{alg:trimestGeneral}.  Thus, we obtain that on the event $\mathcal{E}$, Algorithm \ref{alg:diffmeanest} will achieve $B/2^{r^*} \leq 2K_2$ as long as $\bar{z} \geq C\frac{1}{\mu}\sqrt{R (\log \frac{R}{\xi}+d)}$. 

    The final error bound follows from the utility argument in Theorem \ref{thm:trimest} around lines \eqref{eq:tildeScardinality}-\eqref{eq:finalmDPbound}, by noting that as argued, $n_{LB_2}  \asymp n_{LB_1} \geq c\bar{z} \asymp n$ for some constant $c>0$. Thus,
    \[
    \|\hat{m} - m\| = O\Big(K_1+ \frac{K_2}{n \mu}\sqrt{R \log \frac{R}{\xi}}  + \frac{K_2}{n \mu} \sqrt{d + \log \frac{1}{\xi}}   \Big).
    \]
\end{proof}

We are now ready to state the proof of Theorem \ref{thm:diffconv}.
\begin{proof}
By Theorem \ref{thm:trimdiffest}, with the $K_1$ and $K_2$ obtained in the same way as in the proof of Theorem \ref{thm:currentconv} (i.e., using Lemma \ref{lem:betasconcentration}), we know that, with probability at least $1-\xi/2$, 
\begin{align*}
&\|\hat{\beta}_{1,\mathsf{DP}} - (\beta + \theta)\|  = O\Big(  \sigma_\varepsilon\sqrt{\frac{M_1 (d + \log \frac{1}{\xi})}{nT}} +  \sigma_\varepsilon\max_{i: z_i =1}\|\Gamma_i^{-\frac{1}{2}}\| \sqrt{\frac{(R\log\frac{R}{\xi}+d )(d + \log \frac{n}{\xi})}{n^2\mu^2T}}\Big),
\end{align*}
since $\bar{z} = \Omega(n)$,
and similarly, with probability at least $1-\xi/2$,
\begin{align*}
&\|\hat{\beta}_{0,\mathsf{DP}} - (\beta + \theta)\|  = O\Big(  \sigma_\varepsilon\sqrt{\frac{M_0 (d + \log \frac{1}{\xi})}{nT}} +  \sigma_\varepsilon\max_{i: z_i =0}\|\Gamma_i^{-\frac{1}{2}}\|\sqrt{\frac{(R\log\frac{R}{\xi}+d )(d + \log \frac{n}{\xi})}{n^2\mu^2T}}\Big),
\end{align*}
since also $n_0\asymp n$.
Therefore, with probability at least $1-\xi$,
\begin{align*}
&\|\hat{\beta}_{1, \mathsf{DP}} - \hat{\beta}_{0,\mathsf{DP}} - \theta \| = \|\hat{\beta}_{1,\mathsf{DP}} - (\beta + \theta) - (\hat{\beta}_{0,\mathsf{DP}} - \beta)\| \leq \|\hat{\beta}_{1,\mathsf{DP}} - (\beta + \theta)\| + \|\hat{\beta}_{0,\mathsf{DP}} - \beta\| \\
& = O\Big(  \sigma_\varepsilon\sqrt{\frac{(M_0 + M_1)(d + \log \frac{1}{\xi})}{nT}} +  \sigma_\varepsilon\max_{i \in [n]}\|\Gamma_i^{-\frac{1}{2}}\| \sqrt{\frac{(R\log\frac{R}{\xi}+d )(d + \log \frac{n}{\xi})}{n^2\mu^2T}}  \Big).
\end{align*}
Furthermore, by following the precise same steps as in the proof of Theorem \ref{thm:currentconv}, it follows that under our growth conditions, we have
\[
\sqrt{pnT}\left(\hat{\beta}_{1,\mathsf{DP}} - (\beta + \theta) \right) \overset{d}{\to} N\Big(0, (\lim_{n\to \infty} \frac{1}{pn} \sum_{i \in G_1} \Gamma_i^{-1} \Lambda_i \Gamma_i^{-1}) \Big),
\]
and similarly,
\[
\sqrt{(1-p)nT}\left(\hat{\beta}_{0,\mathsf{DP}} - \beta \right) \overset{d}{\to} N\Big(0, (\lim_{n\to \infty} \frac{1}{(1-p)n} \sum_{i \in G_2} \Gamma_i^{-1} \Lambda_i \Gamma_i^{-1}) \Big),
\]
and it follows, since the two are independent,
\[
\sqrt{nT}(\hat{\beta}_{1,\mathsf{DP}} - \hat{\beta}_{0,\mathsf{DP}} - \theta) \overset{d}{\to} N\Big(0, \frac{1}{p}(\lim_{n\to \infty} \frac{1}{pn} \sum_{i \in G_1} \Gamma_i^{-1} \Lambda_i \Gamma_i^{-1}) + \frac{1}{1-p}(\lim_{n\to \infty} \frac{1}{(1-p)n} \sum_{i \in G_2} \Gamma_i^{-1} \Lambda_i \Gamma_i^{-1}) \Big).
\]

The privacy follows from the privacy proof of Theorem \ref{thm:trimdiffest} and by noting that the privacy parameters in the statement of this theorem have been rescaled by $\frac{1}{\sqrt 2}$ to account for composition over the two sets of queries. 

\end{proof}

\section{Proofs of Section \ref{sec:inf}}
\label{app:covariancerobust}

\subsection{Proof of Theorem \ref{thm:covariancerobust}}
\begin{proof}
We begin by proving the privacy result. The quantities $(\hat{\beta}_{\mathsf{DP}}, \hat{m}_{r^*-1},  r^*, B^*)$ were estimated with $\mu_{\mathsf{est}}$-uGDP. 
It will follow that if the resulting covariance statistic, $\hat{V}(\hat{\beta}_{\mathsf{DP}})_{\mathsf{DP}}$, satisfies $\mu_{\mathsf{var}}$-uGDP when treating $(\hat{\beta}_{\mathsf{DP}}, \hat{m}_{r^* - 1}, r^*, B^*)$ as fixed, non-sensitive quantities, then the combined output $(\hat{\beta}_{\mathsf{DP}}, \hat{V}(\hat{\beta}_{\mathsf{DP}})_{\mathsf{DP}})$ satisfies $\sqrt{\mu_{\mathsf{est}}^2 + \mu_{\mathsf{var}}^2}$-uGDP by the composition properties of GDP (Proposition \ref{prop:composition}).  

We now prove that  $\hat{V}(\hat{\beta}_{\mathsf{DP}})_{\mathsf{DP}}$, calculated in Algorithm \ref{alg:covestimation1}, satisfies $\mu_{\mathsf{var}}$-uGDP when treating $(\hat{\beta}_{\mathsf{DP}}, \hat{m}_{r^*-1},  r^*, B^*)$ as arbitrary fixed quantities. Recall $S$ and $n_{\tilde{S}}$ defined in line~\ref{alg:linewithS} of Algorithm~\ref{alg:covestimation1}. Consider two neighboring datasets, $\mathcal{D}$ and $\mathcal{D}'$, where without loss of generality we vary the data of the first user. Then, the Frobenius norm sensitivity of the covariance construction is the smallest upper bound on
    \[
    \Delta:=\Big\|\frac{1}{n_{\tilde{S}}^2}\sum_{i \in S} (\hat{\beta}_i - \hat{\beta}_{\mathsf{DP}})(\hat{\beta}_i - \hat{\beta}_{\mathsf{DP}})^\tr    - \frac{1}{(n_{\tilde{S}}')^2}\sum_{i \in S'}(\hat{\beta}'_i - \hat{\beta}_{\mathsf{DP}})(\hat{\beta}'_i - \hat{\beta}_{\mathsf{DP}})^\tr  \Big\|_F.
    \]
    We would like to find an upper bound for the sensitivity, $\Delta$.
    First consider the case where $1 \in S$ and $1 \in S'$. Then $S = S'$ and $n_{\tilde{S}} = n_{\tilde{S}}'$, and therefore
    \begin{align*}
    \Delta &= \Big\|\frac{1}{n_{\tilde{S}}^2}\sum_{i \in S}\left( (\hat{\beta}_i - \hat{\beta}_{\mathsf{DP}})(\hat{\beta}_i - \hat{\beta}_{\mathsf{DP}})^\tr    - (\hat{\beta}'_i - \hat{\beta}_{\mathsf{DP}})(\hat{\beta}'_i - \hat{\beta}_{\mathsf{DP}})^\tr   \right)\Big\|_F \\ 
    &= \Big\|\frac{1}{n_{\tilde{S}}^2} \left((\hat{\beta}_1 - \hat{\beta}_{\mathsf{DP}})(\hat{\beta}_1 - \hat{\beta}_{\mathsf{DP}})^\tr   - (\hat{\beta}'_1 - \hat{\beta}_{\mathsf{DP}})(\hat{\beta}'_1 - \hat{\beta}_{\mathsf{DP}})^\tr\right)  \Big\|_F \\ 
    &\leq \frac{1}{n_{\tilde{S}}^2}\sqrt{\|\hat{\beta}_1 - \hat{\beta}_{\mathsf{DP}}\|^4 + \|\hat{\beta}'_1 - \hat{\beta}_{\mathsf{DP}}\|^4  } \leq \frac{1}{n_{\tilde{S}}^2}\sqrt{\kappa^4 + \kappa^4  } \leq \frac{\sqrt 2\kappa^2}{n_{LB}^2},
    \end{align*}
    where we have used that $1 \in S, S'$ implies $\|\hat{\beta}_1 - \hat{\beta}_{\mathsf{DP}}\| \leq \kappa$ by the definition of $S, S'$, since $\|\hat{\beta}_1 - \hat{\beta}_{\mathsf{DP}}\| \leq \|\hat{\beta}_1 - \hat{m}_{r^*-1}\| + \|\hat{m}_{r^* - 1} - \hat{\beta}_{\mathsf{DP}}\| \leq \frac{B}{2^{r^*}} + \delta_{\mathsf{DP}} = \kappa$, and similarly for $\hat{\beta}_1'$.

    On the other hand, if $1 \notin S$ and $1 \in S'$, then necessarily $S' = S \cup \{1\}$, since the two datasets are neighboring, and hence $|S'|=1+|S|$. We can therefore bound the sensitivity $\Delta$ as 
    \begin{align}
     \Delta &= \Big\|\frac{1}{n_{\tilde{S}}^2}\sum_{i \in S} (\hat{\beta}_i - \hat{\beta}_{\mathsf{DP}})(\hat{\beta}_i - \hat{\beta}_{\mathsf{DP}})^\tr    - \frac{1}{(n_{\tilde{S}}')^2}\sum_{i \in S'}(\hat{\beta}'_i - \hat{\beta}_{\mathsf{DP}})(\hat{\beta}'_i - \hat{\beta}_{\mathsf{DP}})^\tr  \Big\|_F \nonumber  \\ 
     &= \Big\|\Big(\frac{1}{n_{\tilde{S}}^2} - \frac{1}{(n_{\tilde{S}}')^2}\Big)\sum_{i \in S} (\hat{\beta}_i - \hat{\beta}_{\mathsf{DP}})(\hat{\beta}_i - \hat{\beta}_{\mathsf{DP}})^\tr   - \frac{1}{(n_{\tilde{S}}')^2}(\hat{\beta}'_1 - \hat{\beta}_{\mathsf{DP}})(\hat{\beta}'_1 - \hat{\beta}_{\mathsf{DP}})^\tr  \Big\|_F \nonumber \\ 
     &\leq \Big\|\frac{2n_{\tilde{S}} + 1}{n_{\tilde{S}}^2 (n_{\tilde{S}}')^2}\sum_{i \in S}(\hat{\beta}_i - \hat{\beta}_{\mathsf{DP}})(\hat{\beta}_i - \hat{\beta}_{\mathsf{DP}})^\tr  \Big\|_F + \Big\|\frac{1}{(n_{\tilde{S}}')^2} (\hat{\beta}'_1 - \hat{\beta}_{\mathsf{DP}})(\hat{\beta}'_1 - \hat{\beta}_{\mathsf{DP}})^\tr   \Big\|_F \nonumber\\ 
     &\leq \frac{2n_{\tilde{S}} + 1}{n_{\tilde{S}}^2 (n_{\tilde{S}}')^2}\sum_{i \in S}\|\hat{\beta}_i - \hat{\beta}_{\mathsf{DP}}\|^2 + \frac{1}{(n_{\tilde{S}}')^2}\|\hat{\beta}'_1 - \hat{\beta}_{\mathsf{DP}}\|^2 \nonumber\\
    &\leq  \frac{(2n_{\tilde{S}} + 1) |S| \kappa^2}{n_{\tilde{S}}^2 (n_{\tilde{S}}')^2} + \frac{\kappa^2}{(n_{\tilde{S}}')^2} \leq \frac{n_{\tilde{S}}(2n_{\tilde{S}} + 1)\kappa^2}{n_{\tilde{S}}^4} + \frac{\kappa^2}{n_{LB}^2} \leq \frac{3n_{\tilde{S}}^2 \kappa^2}{n_{\tilde{S}}^4} + \frac{\kappa^2}{n_{LB}^2} \leq  \frac{4\kappa^2}{n_{LB}^2}. \label{eq:SensitivityOlsCov}
    \end{align}
    By symmetry, the sensitivity is the same in the case that $1 \in S$ and $1 \notin S'$. Finally, if $1 \notin S, S'$ the sensitivity is trivially zero. Thus, in any case, the sensitivity is bounded by $\frac{4\kappa^2}{n_{LB}^2}$.

    In what follows, we argue that to guarantee $\mu$-GDP in the covariance estimator, it suffices to add $N(0, \frac{16\kappa^4}{\mu^2 n_{LB}^4})$ to the diagonal estimates and $N(0, \frac{8\kappa^4}{\mu^2 n_{LB}^4})$ to the off-diagonal estimates. To see this, consider the bijective linear mapping $\text{vec}: \mathcal{M}_{symm}^{d \times d} \mapsto \mathbb{R}^{d(d+1)/2}$ given by
    \[
    \text{vec}(M) = (M_{11}, M_{22}, \dots , M_{dd}, \sqrt{2}M_{12}, \dots, \sqrt{2}M_{(d-1)d} ),
    \]
    and note that this mapping preserves distances in the sense that $\|M\|_F = \|\text{vec}(M)\|_2$ for any $M \in \mathcal{M}_{symm}^{d \times d}$. Thus, the $L_2$-sensitivity of our sample covariance matrix, after being mapped to $\mathbb{R}^{d(d+1)/2}$ using $\text{vec}$, is also $\frac{4\kappa^2}{n_{LB}^2}$. By Proposition \ref{prop:gaussianmechanism}, if $Z \sim N(0,I_{d(d+1)/2})$, then
    \begin{align}\label{eq:covVecPrivate}
    \text{vec}(M) + \frac{4\kappa^2}{n_{LB}^2\mu_{\mathsf{var}}}Z
    \end{align}
    satisfies $\mu_{\mathsf{var}}$-GDP. Now, to retrieve the symmetric matrix, we do
    \[
    \text{vec}^{-1}\Big( \text{vec}(M) + \frac{4\kappa^2}{n_{LB}^2\mu_{\mathsf{var}}}Z\Big) = M + \frac{4\kappa^2}{n_{LB}^2\mu_{\mathsf{var}}}\text{vec}^{-1}(Z),
    \]
    and vec$^{-1}(Z)$ is the matrix that maps Z back to the real symmetric $d \times d$ matrices and divides the off-diagonals by $\sqrt{2}$. By postprocessing (Proposition \ref{prop:postprocessing}), this output satisfies $\mu_{\mathsf{var}}$-GDP. Hence, by composition, $(\hat{\beta}_{\mathsf{DP}}, \hat{V}(\hat{\beta}_{\mathsf{DP}})_{\mathsf{DP}})$ satisfies $\sqrt{\mu_{\mathsf{est}}^2 + \mu_{\mathsf{var}}^2}$-GDP.

    We proceed with the convergence analysis of the algorithm. Our analysis will be performed on the event $\tilde{\mathcal{E}} = \tilde{\mathcal{E}}_0 \cap \tilde{\mathcal{E}}_1 \cap \tilde{\mathcal{E}}_2$, where $\tilde{\mathcal{E}}_1$ is the event that the high probability bounds of Lemma \ref{lem:betasconcentration} hold (they hold with probability at least $1-\xi$), $\tilde{\mathcal{E}}_2$ is the event that equations \eqref{eq:trimestE1}-\eqref{eq:trimestE3} hold when applied to $(\hat{\beta}_i)_{i \in [n]}$ under the conditions of Theorem \ref{thm:currentconv}, which by the discussion surrounding these equations also holds with probability at least $1-\xi$, and the event
    
        \begin{equation}
        \label{eq:E0}
    \tilde{\mathcal{E}_0} := \Big\{ \|W\| = O\Big(\frac{\kappa^2}{n_{LB}^2\mu_{\mathsf{var}}} \sqrt{d + \log \frac{1}{\xi}}\Big)\Big\},
    \end{equation}
    which by Lemma \ref{auxres:GOEnorm} holds with probability at least $1-\xi$. Thus, by a union bound $\tilde{\mathcal{E}}$ holds with probability at least $1-3\xi$. Thus, since by \eqref{eq:growthconditions} we have $\xi = o(1)$, this event will occur with probability tending to $1$.

From line~\ref{alg:cov_est} of Algorithm \ref{alg:covestimation1}, we have
    \begin{equation}
    \label{eq:nTV}
    nT\hat{V}(\hat{\beta}_{\mathsf{DP}})_{\mathsf{DP}} = 
    \frac{nT}{n_{\tilde{S}}^2}\sum_{i \in S}(\hat{\beta}_i - \hat{\beta}_{\mathsf{DP}})(\hat{\beta}_i - \hat{\beta}_{\mathsf{DP}})^\tr   + nT(B^*)^2I +nTW,
    \end{equation}
    where diagonal elements $W_{i,i} \overset{iid}{\sim} N(0, \frac{16\kappa^4}{n_{LB}^4\mu_{\mathsf{var}}^2})$ and off-diagonal elements  $W_{i,j} \overset{iid}{\sim} N(0,\frac{8\kappa^4}{n_{LB}^4\mu_{\mathsf{var}}^2})$ for $j >i$ with $W_{i,j} = W_{j,i}$ for $j <i$. In what follows we show that the second and third terms on the right side of \eqref{eq:nTV} are $o(1)$.

     We begin by analyzing the term $nTW$ in \eqref{eq:nTV}. From line~\ref{alg:linewithS} of Algorithm~\ref{alg:covestimation1}, we have 
     \begin{equation}
     \label{eq:kappa_new}
     \kappa = \frac{B}{2^{r^*}} + \delta_{\mathsf{DP}} = \frac{B}{2^{r^*}}+ \|\hat{m}_{r^*-1} - \hat{\beta}_{\mathsf{DP}}\|,
     \end{equation}
     {and $\frac{B}{2^{r^*}}<2K_2$ on $\tilde{\mathcal{E}}$ by Property 2 of Lemma \ref{lem:propertiesmean}.}
    Furthermore, from Line~\ref{line:mhatmhatminus2} of Algorithm \ref{alg:trimestTemp}, and using the definition of $S$,
    \begin{align*}
        \delta_{\mathsf{DP}} = \|\hat{\beta}_{\mathsf{DP}} - \hat{m}_{r^* - 1}\| &= \Big\|\frac{1}{{n}_{\tilde{S}}}\sum_{i \in S} (\hat{\beta}_i - \hat{m}_{r^* - 1}) + \frac{2\sqrt{2} B}{2^{r^*} \tilde{\mathcal{C}}\mu_{\mathsf{est}} n_{LB}}N_{*}\Big\| \\
        &\leq \frac{1}{|S|}\sum_{i \in S} \frac{B}{2^{r^*}} + \frac{2\sqrt{2}B}{2^{r^*} \tilde{\mathcal{C}} \mu_{\mathsf{est}} n_{LB}} \|N_*\| \\
        &\leq \frac{B}{2^{r^*}} + O\Big(\frac{B\sqrt{d + \log \frac{1}{\xi}}}{2^{r^*} \mu_{\mathsf{est}}n_{LB}}\Big) = O(K_2) + O\Big(K_2\frac{\sqrt{d + \log \frac{1}{\xi}}}{\mu_{\mathsf{est}} n_{LB}}\Big) = O(K_2),
    \end{align*}
    %
    {if $n \geq\frac{C}{\mu_{\mathsf{est}}}\sqrt{R(d + \log\frac{R}{\xi})})$, since then by Property 1 of Lemma \ref{lem:propertiesmean} we have $n_{LB} = \Omega(n)$, and where we used the bound \eqref{eq:trimestE3} on $\|N_*\|$ of $\tilde{\mathcal{E}}$}. Thus $\kappa = O(K_2)$ on $\tilde{\mathcal{E}}$.  As a consequence,
    \[
    \frac{4 \kappa^2}{n_{LB}^2 \mu_{\mathsf{var}}} = O\Big(\frac{K_2^2}{n^2\mu_{\mathsf{var}}}\Big) =O\Big(\sigma_\varepsilon^2 \max_{i \in [n]}\|\Gamma_i^{-\frac{1}{2}}\|^2\frac{1}{T n^2\mu_{\mathsf{var}}}\big(d + \log \frac{n}{\xi}\big) \Big),
    \]
    where the second equality follows from Lemma \ref{lem:betasconcentration}.
   Therefore, we may use $\tilde{\mathcal{E}}_0$ in \eqref{eq:E0} in conjunction with \eqref{eq:growthconditions} to obtain that with probability $1-\xi$,
    \begin{align}\label{eq:OLSCovSimpl1}
    \|nTW\|  = O\Big(\sigma_\varepsilon^2 \max_{i \in [n]}\|\Gamma_i^{-\frac{1}{2}}\|^2\frac{ 1}{n\mu_{\mathsf{var}}} \big(d+\log \frac{n}{\xi}\big)\sqrt{d + \log \frac{1}{\xi}} \Big) = o(1),
    \end{align}

    Now we study the term $nT(B^*)^2I$ in \eqref{eq:nTV}.
     {Notice that on $\tilde{\mathcal{E}}$ and under our assumptions $B^* = \frac{2\sqrt{2}B}{2^{r^*}\mathcal{C} \mu_{\mathsf{est}} n_{LB}} = O(\frac{K_2}{\mu_{\mathsf{est}}n})$ by Property 3 of Lemma \ref{lem:propertiesmean}.}  Hence, using similar argument to those above, also on $\tilde{\mathcal{E}}$ and 
     \eqref{eq:growthconditions},
    \begin{align}\label{eq:ntBbound}
    \|nT(B^*)^2 I\| = nT(B^*)^2 = O\Big(nT\cdot \frac{K_2^2}{\mu_{\mathsf{est}}^2 n^2}\Big) = O\Big(\sigma_\varepsilon^2 \max_{i \in [n]}\|\Gamma_i^{-\frac{1}{2}}\|^2\frac{ 1}{n\mu_{\mathsf{est}}} \big(d+\log \frac{n}{\xi}\big) \Big) = o(1).
    \end{align}

     Thus, to study convergence of $\hat{V}(\hat{\beta}_{\mathsf{DP}})_{\mathsf{DP}}$, using \eqref{eq:nTV} and the fact that the second and third terms on the right side of \eqref{eq:nTV} are $o(1)$, we focus finally on the first term on the right side of \eqref{eq:nTV}:
    \begin{equation}
    \label{eq:final_term}
     \frac{nT}{n_{\tilde{S}}^2}\sum_{i \in S}(\hat{\beta}_i - \hat{\beta}_{\mathsf{DP}})(\hat{\beta}_i - \hat{\beta}_{\mathsf{DP}})^\tr   =  \frac{n}{n_{\tilde{S}}}\Big(\frac{T}{n_{\tilde{S}}}\sum_{i \in S}(\hat{\beta}_i - \hat{\beta}_{\mathsf{DP}})(\hat{\beta}_i - \hat{\beta}_{\mathsf{DP}})^\tr \Big).
    \end{equation}
    Notice that by definition of $n_{\tilde{S}} = \max \{ |S|, n_{LB}\}$,
     \begin{align}\label{eq:CovOLSndivByL0}
    1 \leq \frac{n}{n_{\tilde{S}}} \leq \frac{n}{n_{LB}},
    \end{align}
    {and from the definition of $n_{LB}$ in Algorithm~\ref{alg:trimestTemp} and using Property 1 of Lemma \ref{lem:propertiesmean}},
    \begin{align}\label{eq:CovOLSndivByL}
    \frac{n}{n_{LB}} = 1 + \frac{n-n_{LB}}{n_{LB}} = 1 + \frac{\frac{4}{\mu_{\mathsf{est}}}\sqrt{2R \log \frac{4R}{\xi}}}{n_{LB}} = 1 + O\Big( \frac{1}{n\mu_{\mathsf{est}}} \sqrt{R \log \frac{R}{\xi}}\Big).
    \end{align}
    Considering \eqref{eq:CovOLSndivByL}, we see that by \eqref{eq:growthconditions}, $\frac{n}{n_{LB}} \to 1$ as $n\to \infty$ and from \eqref{eq:CovOLSndivByL0}, we have $\frac{n}{n_{\tilde{S}}} \to 1$ as well. Then, by \eqref{eq:final_term} and the Continuous Mapping Theorem, it suffices to find the probability limit of
    \[
    \frac{T}{n_{\tilde{S}}}\sum_{i \in S}(\hat{\beta}_i - \hat{\beta}_{\mathsf{DP}})(\hat{\beta}_i - \hat{\beta}_{\mathsf{DP}})^\tr.
    \]
    First recall from line~\ref{alg:linewithS} of Algorithm~\ref{alg:covestimation1} that $n_{\tilde{S}} = \max(|S|, n_{LB})$ and $S = |\{ i: \|\hat{\beta}_i - \hat{m}_{r^*-1}\| \leq \frac{B}{2^{r^*}} \}|$. {Notice that $n_{\tilde{S}}= |S|$ on $\tilde{\mathcal{{E}}}$:
    indeed, $|\{i: \|\hat{\beta}_i - \hat{m}_{r^* - 1}\| \leq \frac{B}{2^{r^*}}\}| \geq n_{LB}$ by Property 3 of Lemma \ref{lem:propertiesmean}}.
    Therefore, we study the limit of
    \begin{align}\label{eq:OLSCovSimpl2}
     \frac{T}{n_{\tilde{S}}}\sum_{i \in S}(\hat{\beta}_i - \hat{\beta}_{\mathsf{DP}})(\hat{\beta}_i - \hat{\beta}_{\mathsf{DP}})^\tr= \frac{T}{|S|}\sum_{i \in S}(\hat{\beta}_i - \hat{\beta}_{\mathsf{DP}})(\hat{\beta}_{i} - \hat{\beta}_{\mathsf{DP}})^\tr.
    \end{align}
    {Recall from Property 3 of Lemma \ref{lem:propertiesmean} that on $\tilde{\mathcal{E}}$ we have $\hat{\beta}_{\mathsf{DP}} = \bar{\beta}_{S} + B^*Z,$
    where
     $Z \sim N(0,I)$ has bounded contribution, $|S| \geq n_{LB} \geq n/2$, and $B^* = O(\frac{K_2}{n\mu_{\mathsf{est}}})$. Using this, we write \eqref{eq:OLSCovSimpl2}}
 as
 \begin{align}\label{eq:OLSCovSolve1}
     \frac{T}{|S|}\sum_{i \in S}(\hat{\beta}_i - \hat{\beta}_{\mathsf{DP}})(\hat{\beta}_{i} - \hat{\beta}_{\mathsf{DP}})^\tr = \frac{T}{|S|}\sum_{i \in S}(\hat{\beta}_i - \beta)(\hat{\beta}_i - \beta)^\tr -T(\bar{\beta}_{S} - \beta)(\bar{\beta}_{S} - \beta)^\tr + T(B^*)^2 ZZ^\tr.
    \end{align}

     {In what follows, we show that the final two terms on the right side of \eqref{eq:OLSCovSolve1} are $o(1)$. Under our conditions by Property 4 of Lemma \ref{lem:propertiesmean} and \eqref{eq:OLSK1}, we have, on $\tilde{\mathcal{E}}$}
    \begin{equation}
    \begin{split}
    \label{eq:OLSCovSolve2}
    \|T(\bar{\beta}_{S} - \beta)(\bar{\beta}_{S} - \beta)^\tr \| = T\|\bar{\beta}_{S}- \beta\|^2 &= TO\Big(K_1^2 + \frac{K_2^2 R \log \frac{R}{\xi}}{n^2\mu_{\mathsf{est}}^2} \Big) \\
    &= O\Big(\frac{d + \log \frac{1}{\xi}}{n} + \frac{(d + \log \frac{n}{\xi})R \log \frac{R}{\xi}}{n^2\mu_{\mathsf{est}}^2}\Big) = o(1).
    \end{split}
    \end{equation}
   Furthermore,  by our earlier conclusions around \eqref{eq:ntBbound}, on $\tilde{\mathcal{E}}$,
    \begin{align}\label{eq:OLSCovSolve31}
    \|T(B^*)^2Z Z^\tr \| = T(B^*)^2\|Z\|^2 = o(1).
    \end{align}

    Thus, by \eqref{eq:OLSCovSolve1}, to conclude the convergence analysis it suffices to show that
    \begin{align}\label{eq:OLSCovSolve32}
    \frac{T}{|S|}\sum_{i \in S}(\hat{\beta}_i - {\beta})(\hat{\beta}_i - \beta)^\tr \overset{p}{\to} \Big(\lim_{n\to \infty}\frac{1}{n}\sum_{i=1}^n \Gamma_i^{-1} \Lambda_i \Gamma_i^{-1}\Big).
    \end{align}
    Denoting $A_S := \frac{1}{|S|}\sum_{i \in S}(\hat{\beta}_{i} - \beta)(\hat{\beta}_{i} - \beta)^\tr$, first notice that 
    \begin{align}\label{eq:OLSCovSolve4}
    \frac{T}{|S|}\sum_{i\in S}(\hat{\beta}_{i} - \beta)(\hat{\beta}_{i} - \beta)^\tr = \frac{T}{n}\sum_{i=1}^n (\hat{\beta}_{i} - \beta)(\hat{\beta}_{i} - \beta)^\tr + \frac{T(n - |S|)}{n}(A_S - A_{S^{\mathsf{c}}}).
    \end{align}
    In what follows, we argue that $\frac{T(n - |S|)}{n}(A_S - A_{S^{\mathsf{c}}}) \rightarrow 0$. We have
    \[
    T\Big\|\frac{n - |S|}{n}(A_S - A_{S^{\mathsf{c}}})\Big\| = \frac{T|S^{\mathsf{c}}|}{n}\| A_S - A_{S^{\mathsf{c}}}\| \leq \frac{T|S^{\mathsf{c}}|}{n}\left( \|A_s\| + \|A_{S^{\mathsf{c}}}\| \right) \leq \frac{2T|S^{\mathsf{c}}|}{n}\max_{i \in [n]}\|\hat{\beta}_{i} - \beta\|^2,
    \]
    {and by Properties 1 and 3 of Lemma \ref{lem:propertiesmean}, on $\tilde{\mathcal{E}}$},
    \[
    \frac{|S^{\mathsf{c}}|}{n} = \frac{n - |S|}{n}  \leq \frac{n- n_{LB}}{n}= O\Big(\frac{1}{n\mu_{\mathsf{est}}} \sqrt{R \log \frac{R}{\xi}}\Big),
    \]
    whereas by Lemma \ref{lem:betasconcentration}, for large enough $T$, under our conditions and assumptions, on $\tilde{\mathcal{E}}$,
    \[
    \max_{i \in [n]}\|\hat{\beta}_{i} - \beta\| = O\Big(\sqrt{\frac{1}{T}\big(d + \log \frac{n}{\xi}\big)} \Big).
    \]
    Hence, on this event,
    \[
     T\Big\|\frac{n - |S|}{n}(A_S - A_{S^{\mathsf{c}}})\Big\|  \leq \frac{2T|S^{\mathsf{c}}|}{n}\max_{i \in [n]}\|\hat{\beta}_{i} - \beta\|^2 = O\Big(\frac{d + \log \frac{n}{\xi}}{n} \sqrt{R \log \frac{R}{\xi}} \Big).
    \]
    As a result,  by our growth conditions \eqref{eq:growthconditions},
    \begin{align}\label{eq:OLSCovSolve5}
    T\left(\frac{n - |S|}{n}(A_S - A_{S^{\mathsf{c}}})\right) \overset{p}{\to} 0.
    \end{align}

    Now, considering \eqref{eq:OLSCovSolve32}, \eqref{eq:OLSCovSolve4}, and  \eqref{eq:OLSCovSolve5}, it follows that we can consider simply whether
    $
    \frac{T}{n}\sum_{i=1}^n (\hat{\beta}_i - \beta)(\hat{\beta}_i - \beta)^\tr \overset{p}{\to} (\lim_{n\to \infty}\frac{1}{n}\sum_{i=1}^n \Gamma_i^{-1} \Lambda_i \Gamma_i^{-1}).
  $
    First of all, denote $W_{i,T} := \sqrt{T}(\hat{\beta}_i - \beta) = \Gamma_{i,T}^{-1} S_{i,T},$
    where $\Gamma_{i,T} := \frac{1}{T} \sum_{t=1}^T x_{i,T}x_{i,T}^\tr$, and $S_{i,T} := \frac{1}{\sqrt{T}}\sum_{t=1}^T x_{i,t} \varepsilon_{i,t}$. Hence, what remains to be shown is the following: 
    \begin{align}\label{eq:OLSCovSimpl4}
   \frac{T}{n}\sum_{i=1}^nW_{i,T} W_{i,T}^\tr  \overset{p}{\to} \Big(\lim_{n\to \infty}\frac{1}{n}\sum_{i=1}^n \Gamma_i^{-1} \Lambda_i \Gamma_i^{-1}\Big).
    \end{align}

We know by Slutsky's and standard arguments (see e.g.\ the proof of Theorem \ref{thm:currentconv} for the ingredients) that
$
W_{i,T} \overset{d}{\to} N(0,  \Gamma_i^{-1} \Lambda_i \Gamma_i^{-1}),
$
as $T \to \infty$, which implies, by the Continuous Mapping Theorem, that
$
W_{i,T} W_{i,T}^\tr \overset{d}{\to} Z_i Z_i^\tr,
$
where $Z_i \sim N(0, \Gamma_i^{-1} \Lambda_i \Gamma_i^{-1})$. By Assumption \ref{assumption:UI}, this allows us to conclude that $\mathbb{E}[W_{i,T} W_{i,T}^\tr ] \to \Gamma_i^{-1} \Lambda_i \Gamma_i^{-1}$ as $T \to \infty$. Indeed, for some $\gamma \in (0, \frac{\delta}{2})$ with $\delta$ as in Assumption \ref{assumption:UI}, letting $p = \frac{1 + \delta/2}{1 + \gamma} > 1$ and $q = \frac{p}{p-1}$, by H\"older's inequality,
\[
\sup_{T \geq T_0}\mathbb{E}\|W_{i,T}W_{i,T}^\tr \|^{1+ \gamma} \leq \Big(\sup_{T \geq T_0}\mathbb{E}\|\Gamma_{i,T}^{-1}\|^{2(1+\gamma)p} \Big)^{1/p}\Big(\sup_{T \geq T_0}\mathbb{E}\|S_{i,T}\|^{2(1+\gamma)q} \Big)^{1/q}.
\]
Notice, the first term on the right side of the above is bounded by Assumption \ref{assumption:UI} and the second term by Lemma \ref{auxres:GammaITExpectationBound}. We have thereby established  uniform integrability over $W_{i,T}W_{i,T}^\tr$.
We can now prove our final result; namely, we want to show \eqref{eq:OLSCovSimpl4}. Towards this result, notice that 
\begin{equation}
\begin{split}
\label{eq:final_bound}
&\frac{1}{n}\sum_{i=1}^n W_{i,T}W_{i,T}^\tr - \Big(\lim_{n\to \infty}\frac{1}{n}\sum_{i=1}^n \Gamma_i^{-1} \Lambda_i \Gamma_i^{-1}\Big) \\
&\quad = \frac{1}{n}\sum_{i=1}^n \Big(W_{i,T}W_{i,T}^\tr - \mathbb{E}[W_{i,T}W_{i,T}^\tr]\Big) + \frac{1}{n}\sum_{i=1}^n\Big(\mathbb{E}[W_{i,T}W_{i,T}^\tr ] - \Gamma_i^{-1} \Lambda_i \Gamma_i^{-1}\Big) \\
&\quad \quad +  \frac{1}{n} \sum_{i=1}^n \Gamma_i^{-1} \Lambda_i \Gamma_i^{-1} - \Big(\lim_{n\to \infty}\frac{1}{n}\sum_{i=1}^n \Gamma_i^{-1} \Lambda_i \Gamma_i^{-1}\Big).
\end{split}
\end{equation}
Label the terms on the right side of the above $I_1, I_2$ and $I_3$ and we will show that each converges to $0$ in probability to complete the proof.

First notice that $I_1$ of \eqref{eq:final_bound} converges to $0$ in $L^{1+\delta}$ for $\delta<1$ small enough so that Assumption \ref{assumption:UI} holds with $2\delta$. Hence, it converges to zero in probability as well because
\begin{align*}
    \mathbb{E}\Big\|\frac{1}{n}\sum_{i=1}^n (W_{i,T}W_{i,T}^\tr - \mathbb{E}[W_{i,T}W_{i,T}^\tr])\Big\|^{1 + \delta} 
    &\overset{}{\leq} \mathbb{E} \Big\|\frac{1}{n}\sum_{i=1}^n (W_{i,T} W_{i,T}^\tr - \mathbb{E}[W_{i,T} W_{i,T}^\tr])\Big\|^{1 + \delta}_F  \\
    &\overset{}{\leq}  \sum_{j =1}^d \sum_{k = 1}^d \mathbb{E} \Big|\frac{1}{n} \sum_{i=1}^n \Psi_{i,T,j,k}\Big|^{1 + \delta},
\end{align*}
where we defined $\Psi_{i,T,j,k}$ as the $(j,k)$'th entry of the centered $W_{i,T} W_{i,T}^\tr$. Now we show the right side of the above is $o(1)$.
\begin{align*}
     \sum_{j =1}^d \sum_{k = 1}^d \mathbb{E} \Big|\frac{1}{n} \sum_{i=1}^n \Psi_{i,T,j,k}\Big|^{1 + \delta}  &\overset{(1)}{\leq}  \frac{C}{n^{1 + \delta}} \sum_{i=1}^n\sum_{j =1}^d \sum_{k = 1}^d \mathbb{E} |\Psi_{i,T,j,k}|^{1 + \delta} \\
    &\overset{}{\leq}  \frac{C'd^2}{n^\delta} \sup_{T \geq T_0} \sup_{i,j,k} \mathbb{E}|\Psi_{i,t,j,k}|^{1+\delta} \\
    &\overset{}{\leq}  \frac{C''d^2}{n^{\delta}}\sup_{i, T\geq T_0} \mathbb{E}\Big\|W_{i,T}W_{i,T}^\tr - \mathbb{E}[W_{i,T}W_{i,T}^\tr]\Big\|^{1+\delta}\overset{(2)}{=} o(1),
\end{align*}
In the above, $C,C',C''>0$ are constants independent of $n$ and $T$, and  inequality (1)  follows from Lemma \ref{lemma:vonbahresseen}, and equality (2) follows from Assumptions  \ref{assumptions:initialcovariates}, \ref{assumptions:initialerrors}, and \ref{assumption:UI}.

Next, considering term $I_2$ of \eqref{eq:final_bound},
\begin{align*}
   \Big\|\frac{1}{n}\sum_{i=1}^n (\mathbb{E}[W_{i,T} W_{i,T}^\tr ] - \Gamma_i^{-1} \Lambda_i \Gamma_i^{-1}) \Big\| &\leq \sup_{i\geq 1} \Big\|\mathbb{E}[W_{i,T} W_{i,T}^\tr] - \Gamma_i^{-1} \Lambda_i \Gamma_i^{-1}\Big\|.
   \end{align*}
  Therefore, we will show that $\sup_{i \geq 1} \|\mathbb{E}[W_{i,T} W_{i,T}^\tr] - \Gamma_i^{-1} \Lambda_i \Gamma_i^{-1}\| \to 0$ as $T \to \infty$. We can write
   \begin{align}
   \label{eq:single_term}
       \mathbb{E}[W_{i,T} W_{i,T}^\tr] - \Gamma_i^{-1} \Lambda_i \Gamma_i^{-1} = \left(\mathbb{E}[W_{i,T} W_{i,T}^\tr] - \Gamma_i^{-1} \Lambda_{i,T} \Gamma_i^{-1}\right) + \Gamma_i^{-1}(\Lambda_{i,T} - \Lambda_i)\Gamma_i^{-1},
   \end{align}
    for $\Lambda_{i,T} := \mathbb{E}[S_{i,T} S_{i,T}^\tr]$. Consider the second term on the right side of \eqref{eq:single_term}. Recall,
   \begin{align*}
       \Lambda_{i,T} - \Lambda_i =\sum_{h = -(T-1)}^{T-1}\Big(1 - \frac{|h|}{T}\Big)\mathbb{E}[x_{i,0} x_{i,h}^\tr \varepsilon_{i,0} \varepsilon_{i,h}] - \sum_{h=-\infty}^\infty \mathbb{E}[x_{i,0}x_{i,h}^\tr \varepsilon_{i,0} \varepsilon_{i,h}] ,   \end{align*}
       and, denoting $\Gamma_i(h):= \mathbb{E}[x_{i,0} x_{i,h}^\tr \varepsilon_{i,0} \varepsilon_{i,h}]$, we have
       \begin{align*}
           \sup_{i\geq 1} \| \Gamma_i^{-1}(\Lambda_{i,T} - \Lambda_i)\Gamma_i^{-1}\| &\leq \frac{1}{c_\Gamma^2} \sup_{i\geq 1} \|\Lambda_{i,T} - \Lambda_i\| \\
           &\leq \frac{1}{c_\Gamma^2}\Big(\sum_{|h| \geq T}\sup_{i \geq 1} \|\Gamma_i(h)\| + \frac{1}{T} \sum_{|h| <T}|h|\sup_{i\geq1}\|\Gamma_i(h)\|   \Big) \to 0,
       \end{align*}
       since the sub-Gaussianity and mixing constants are uniform in $i \in \mathbb{N}$. In particular, by Lemma \ref{auxres:davydov} and Assumptions \ref{assumptions:initialcovariates}-\ref{assumptions:initialerrors}, there exists a constant $C>0$ such that $\sup_{i \geq 1} \|\Gamma_i(h)\| \leq C|h|^{-3-q/2}$, so that the above sum is $O(T^{-1})$. Consider the first term on the right side of \eqref{eq:single_term}. We will show $\sup_{i \geq 1}\|\mathbb{E}[W_{i,T}W_{i,T}^\tr] - \Gamma_i^{-1} \Lambda_{i,T} \Gamma_i^{-1} \| \to 0$ as $T \to \infty$. Let $\Delta_{i,T} := \Gamma_{i,T}^{-1} - \Gamma_i^{-1}$. Then
       \begin{align*}
           \mathbb{E}[W_{i,T} W_{i,T}^\tr] - \Gamma_i^{-1} \Lambda_{i,T} \Gamma_i^{-1} = \mathbb{E}[\Delta_{i,T} S_{i,T} S_{i,T}^\tr] \Gamma_i^{-1} + \Gamma_i^{-1} \mathbb{E}[S_{i,T} S_{i,T}^\tr \Delta_{i,T}] + \mathbb{E}[\Delta_{i,T} S_{i,T} S_{i,T}^\tr \Delta_{i,T}],
       \end{align*}
       so that
       \begin{align*}
       \sup_{i\geq 1} &\|\mathbb{E}[W_{i,T} W_{i,T}^\tr] - \Gamma_i^{-1} \Lambda_{i,T} \Gamma_i^{-1} \| \\
       &\leq 2\frac{1}{c_\Gamma}\sup_{i\geq1} (\mathbb{E}\|S_{i,T}\|^{2q})^{1/q} \sup_{i \geq 1} (\mathbb{E}\|\Delta_{i,T}\|^p)^{1/p} + \sup_{i\geq 1}(\mathbb{E}\|S_{i,T}\|^{2q'})^{1/q'} \sup_{i\geq 1} (\mathbb{E}\|\Delta_{i,T}\|^{2p'})^{1/p'}
       \end{align*}
       for $p,q,p',q'$ such that $\frac{1}{p} + \frac{1}{q} = 1$ and $\frac{1}{p'} + \frac{1}{q'} = 1$. Now, for any $r>1$ we have
       \[
       \mathbb{E}\|\Delta_{i,T}\|^r = \mathbb{E}\|\Gamma_{i,T}^{-1} - \Gamma_i^{-1}\|^r \leq \|\Gamma_i^{-1}\|^r(\mathbb{E}\|\Gamma_{i,T}^{-1}\|^{ar})^{1/a} (\mathbb{E}\|\Gamma_{i,T} - \Gamma_i\|^{br})^{1/b}, 
       \]
       with $\frac{1}{a} + \frac{1}{b} = 1$. Now, taking $p = q = 2$, $p' = 1 + \delta/4$, $q' = 1 + 4/\delta$, $a = \frac{2+\delta}{2 + \delta/2}$, and $b = 2 + 4/\delta$, where $\delta>0$ is as in Assumption \ref{assumption:UI}, we obtain that $\sup_{i\geq 1} (\mathbb{E} \|S_{i,T}\|^{4})^{1/2} = O(1)$ and $\sup_{i \geq 1}(\mathbb{E}\|S_{i,T}\|^{2 + 8/\delta})^{\frac{1}{1+4/\delta}} = O(1)$ by Lemma \ref{auxres:GammaITExpectationBound}. Similarly, for our choices of $a,b,p,p'$, we have $ap = \frac{2(2+\delta)}{2+\delta/2}$, $bp = 4 + \frac{8}{\delta}$, $2ap' = 2 + \delta$, and $2bp' = 2b(1 + \delta/4)$. Thus by Assumption \ref{assumption:UI} we have $\sup_{i \geq 1}\mathbb{E}\|\Gamma_{i,T}^{-1}\|^{ap} = O(1)$ and $\sup_{i \geq 1} \mathbb{E}\|\Gamma_{i,T}^{-1}\|^{2ap'} = O(1)$. Furthermore, by Lemma \ref{auxres:GammaITExpectationBound}, we obtain $\sup_{i \geq 1} \mathbb{E}\|\Gamma_{i,T} - \Gamma_i\|^{bp} = O(T^{-2 - \frac{4}{\delta}})$ and $\sup_{i \geq 1} \mathbb{E}\|\Gamma_{i,T} - \Gamma_i\|^{2bp'} = O(T^{-b(1 + \delta/4)})$, so that we may conclude that
       $\sup_{i \geq 1} \|\mathbb{E}[W_{i,T} W_{i,T}^\tr] - \Gamma_i^{-1} \Lambda_{i,T} \Gamma_i^{-1} \| = o(1),$
       and hence $I_2 \to 0$ as $(n,T) \to \infty$.

Finally, $I_3$ goes to zero as $n\to \infty$ by assumption and is independent of $T$. 
\end{proof}

\section{Private Covariance Estimation with Two Groups}\label{Appendix:TwoGroupCov}

\begin{theorem}
\label{thm:covarianceDiff}
Let $\hat{\beta}_{1,\mathsf{DP}}$ be estimated as in Theorem \ref{thm:diffconv} with $\mu_{\mathsf{est}}$-uGDP. Let $(\hat{\beta}_{i}, z_i)_{i=1}^n$ and $\hat{\beta}_{1,\mathsf{DP}}$ be the input to Algorithm \ref{alg:covestimationGrouped} with privacy level $\mu_{\mathsf{var}}$ and denote the resulting estimate by $\hat{V}(\hat{\beta}_{1,\mathsf{DP}})_{\mathsf{DP}}$.

\textbf{(a) Privacy.} 
$(\hat{\beta}_{1,\mathsf{DP}}, \hat{V}(\hat{\beta}_{1,\mathsf{DP}})_{\mathsf{DP}})$ satisfies $\sqrt{\mu_{\mathsf{est}}^2 + \mu_{\mathsf{var}}^2}$-uGDP.

\textbf{(b) Convergence.}
If the conditions of Theorem \ref{thm:diffconv} are met, then, as $(n,T) \to \infty$,
    \[
    \left(\hat{V}(\hat{\beta}_{1,\mathsf{DP}})_{\mathsf{DP}} \right)^{-\frac{1}{2}} (\hat{\beta}_{1,\mathsf{DP}} - \beta) \overset{d}{\to} N(0,I_d).
    \]
\end{theorem}
\begin{proof}
    For the privacy component of the  proof, all that is required is to verify the sensitivity of the statistic $\frac{1}{n_{\tilde{S}}^2} \sum_{i \in S} (\hat{\beta}_i - \hat{\beta}_{1,\mathsf{DP}})(\hat{\beta}_i - \hat{\beta}_{1,\mathsf{DP}})^\tr$. Notice that by post-processing (Proposition \ref{prop:postprocessing}), $\hat{\beta}_{1,\mathsf{DP}}$ can be considered a fixed quantity for the purpose of the sensitivity analysis. Consider two neighboring datasets $\mathcal{D}$ and $\mathcal{D}'$ where without loss of generality only the data of the first user differs between the two datasets. Thus, the only difference between the datasets $\mathcal{D}$ and $\mathcal{D}'$ is through a difference between $(\hat{\beta}_1, z_1)$ and $(\hat{\beta}_1', z_1')$.  If $z_1 = z_1'$ then the analysis reverts to the analysis of \eqref{eq:SensitivityOlsCov}, yielding a sensitivity of at most ${4\kappa^2}/{n_{LB}^2}$. If $z_1 = z_1' = 0$, then the sensitivity is $0$. Hence the only new case to consider is when $z_1 = 1$ and $z_1 = 0$ (the opposite case follows by symmetry). In this case $1 \notin S'$ and either $1 \in S$ or $1 \notin S'$. In the former case the sensitivity is $0$, whereas in the latter case the sensitivity is bounded by ${4\kappa^2}/{n_{LB}^2}$ by \eqref{eq:SensitivityOlsCov}. Hence, the sensitivity of this statistic is globally bounded by ${4\kappa^2}/{n_{LB}^2}$, and hence by our discussion surrounding \eqref{eq:covVecPrivate} and by composition (Proposition \ref{prop:composition}), we have that  $(\hat{\beta}_{\mathsf{DP}}, \tilde{V}(\hat{\beta}_{1,\mathsf{DP}})_{\mathsf{DP}})$ is $\mu_{\mathsf{var}}$-uGDP. Since $\hat{V}(\hat{\beta}_{1,\mathsf{DP}})_{\mathsf{DP}}$ is deterministic post-processing of $\tilde{V}(\hat{\beta}_{1,\mathsf{DP}})_{\mathsf{DP}}$, the privacy result follows by Proposition \ref{prop:postprocessing}. 

    The utility proof is nearly identical to the utility proof of Theorem \ref{thm:covariancerobust}. As in \eqref{eq:OLSCovSimpl1}, we have $\|nTW\| = o(1)$ on our high probability event, since  again $\kappa = O(K_2)$ and, as we showed in \eqref{eq:LTwoGroupLOneGroup}, we again have $n_{LB} = \Omega(\sum_{i=1}^n z_i ) = \Omega(n)$. For the same reason, we also have $\|nT(B^*)I\| = o(1)$ on our high probability event, as in \eqref{eq:ntBbound}. Thus, it suffices for us to analyze
    \[
    \frac{Tn}{n_{\tilde{S}}^2}\sum_{i \in S} (\hat{\beta}_i - \hat{\beta}_{1,\mathsf{DP}})(\hat{\beta}_i - \hat{\beta}_{1,\mathsf{DP}})^\tr,
    \]
    and as in \eqref{eq:CovOLSndivByL}, we have that $\frac{n}{S} \to 1$ on our high probability event \eqref{eq:DiffMeanHighProbEvent} as $n \to \infty$. Furthermore, as in Property 3 of Lemma \ref{lem:propertiesmean}, we have that $|S| \geq n_{LB}$, so that $n_{\tilde{S}} = |S|$. Thus, it boils down to deriving the probability limit of
    \[
    \frac{T}{|S|}\sum_{ i \in S}(\hat{\beta}_i - \hat{\beta}_{1,\mathsf{DP}})(\hat{\beta}_i - \hat{\beta}_{1,\mathsf{DP}})^\tr,
    \]
    and since \eqref{eq:OLSCovSolve1}, \eqref{eq:OLSCovSolve2}, \eqref{eq:OLSCovSolve31}, \eqref{eq:OLSCovSolve32}, \eqref{eq:OLSCovSolve4} and \eqref{eq:OLSCovSolve5} all apply in our case, we can equivalently obtain the probability limit of
    \[
    \frac{T}{\sum_{i=1}^n z_i}\sum_{i: z_i = 1}(\hat{\beta}_i - \beta)(\hat{\beta}_i - \beta)^\tr,
    \]
    and the discussion after \eqref{eq:OLSCovSimpl4} yields, in our setting, by simply considering the sample $\{i \in [n]: z_i = 1\}$ instead of $\{i \in [n]\}$,
    \[
     \frac{T}{\sum_{i=1}^n z_i}\sum_{i: z_i = 1}(\hat{\beta}_i - \beta)(\hat{\beta}_i - \beta)^\tr \overset{p}{\to} \lim_{n\to\infty} (\frac{1}{pn}\sum_{i: z_i=1} \Gamma_i^{-1} \Lambda_i \Gamma_i^{-1}),
    \]
    proving our desired result.
\end{proof}
Naturally, Theorem \ref{thm:covarianceDiff} also provides guarantees for the output of the algorithm with input $(\hat{\beta}_i, z_i)_{i=1}^n$, yielding guarantees for $\hat{V}(\hat{\beta}_{0,\mathsf{DP}})_{\mathsf{DP}}$. This allows us to obtain the following corollary. Thus, one can also obtain private, asymptotically valid confidence intervals and Wald tests as in Corollaries \ref{corollaryCI} and \ref{corollaryTest} for the coefficient vector $\theta$. 
\begin{corollary}
    Let $(\hat{\beta}_{1,\mathsf{DP}}, \hat{V}(\hat{\beta}_{1,\mathsf{DP}})_{\mathsf{DP}})$ and $(\hat{\beta}_{0,\mathsf{DP}}, \hat{V}(\hat{\beta}_{0,\mathsf{DP}})_{\mathsf{DP}})$ be as in Theorem \ref{thm:covarianceDiff}. Then $(\hat{\beta}_1, \hat{\beta}_0, \hat{V}(\hat{\beta}_{1,\mathsf{DP}})_{\mathsf{DP}}, \hat{V}(\hat{\beta}_{0,\mathsf{DP}})_{\mathsf{DP}})$ satisfies $\sqrt{2(\mu_{\mathsf{est}}^2 + \mu_{\mathsf{var}}^2)}$-uGDP. If furthermore the conditions of Theorem \ref{thm:covarianceDiff} are met, then,
    \[
    \left(\hat{V}(\hat{\beta}_{1,\mathsf{DP}})_{\mathsf{DP}} + \hat{V}(\hat{\beta}_{0,\mathsf{DP}})_{\mathsf{DP}} \right)^{-\frac{1}{2}} \left(\hat{\beta}_{1,\mathsf{DP}} - \hat{\beta}_{0,\mathsf{DP}} - \gamma\right) \overset{d}{\to} N(0,I_d).
    \]
\end{corollary}

\begin{algorithm}[H]
\caption{Treated User-Level Covariance Matrix Estimation}
\begin{algorithmic}[1]
\label{alg:covestimationGrouped}
\STATE \textbf{Input:} User-level OLS estimates $\hat{\beta}_i$, group membership variables $z_i$, the output $(\hat{\beta}_{1,\mathsf{DP}}, \hat{m}_{r^*-1},  r^*, n_{LB}, B^*) $ of Algorithm \ref{alg:diffmeanest}, and covariance privacy parameter $\mu_{\mathsf{var}}$.
\STATE Let $\delta_{\mathsf{DP}} \gets \|\hat{\beta}_{1,\mathsf{DP}} - \hat{m}_{r^*-1}\|$,  $\kappa \gets \frac{B}{2^{r^*}} + \delta_{\mathsf{DP}}$,  $S \gets \{i \in [n]: \|\hat{\beta}_i - \hat{m}_{r^*-1}\| \leq \frac{B}{2^{r^*}} z_i \}$,  $n_{\tilde{S}} \gets \max(|S|, n_{LB})$
\STATE Let 
        \[
        \tilde{V}(\hat{\beta}_{1,\mathsf{DP}})_{\mathsf{DP}} \gets \frac{1}{n_{\tilde{S}}^2}\sum_{i \in S}(\hat{\beta}_i - \hat{\beta}_{1,\mathsf{DP}})(\hat{\beta}_i - \hat{\beta}_{1,\mathsf{DP}})^\tr  
        + (B^*)^2I 
        +  W,
        \]
        where $W_{i,i} \overset{iid}{\sim} N(0, \frac{16\kappa^4}{n_{LB}^4\mu_{\mathsf{var}}^2} )$ and  $W_{i,j} \overset{iid}{\sim} N(0,\frac{8\kappa^4}{n_{LB}^4\mu_{\mathsf{var}}^2})$ for $j >i$ with $W_{i,j} = W_{j,i}$ for $j <i$.
\STATE Let $\hat{V}(\hat{\beta}_{1,\mathsf{DP}})_{\mathsf{DP}}$ be the projection of $\tilde{V}(\hat{\beta}_{1,\mathsf{DP}})_{\mathsf{DP}}$ onto the cone $\mathcal{M}_{symm}^{d\times d}$. 
\STATE \textbf{Output:} $\hat{V}(\hat{\beta}_{1,\mathsf{DP}})_{\mathsf{DP}}$
\end{algorithmic}
\end{algorithm}

\section{Examples}
\subsection{Example \ref{example:covariateassumption}}\label{subsec:covariateassumption}
\begin{proof}
        In this case  
    \[
    \Gamma =  \begin{bmatrix}
        1 & m^\tr \\ m & mm^\tr
    \end{bmatrix} + \begin{bmatrix}
        0 & \mathbf{0}^\tr \\
        \mathbf{0} & \sum_{j=0}^\infty \Psi_j \Sigma_e\Psi_j^\tr
    \end{bmatrix}. 
    \]
Furthermore, since $e_{i,t} \overset{iid}{\sim}\text{subG}(\sigma_x^2)$, we have $\Gamma^{-\frac{1}{2}}x_{i,t} \sim \text{subG}(\sigma_x^2 \|\Gamma^{-\frac{1}{2}}\|^2 \sum_{j=0}^\infty \|\Psi_j\|^2)$. The remaining conditions are easy to check, but note that 
    \[
    \lambda_{\min}(\Gamma) \geq \frac{\lambda_{\min}(\sum_{j \geq 0} \Psi_j \Sigma_e \Psi_j^\tr)}{1 + \|m\|^2 + \lambda_{\min}(\sum_{j \geq 0} \Psi_j \Sigma_e \Psi_j^\tr)}, \quad \text{ and } \quad \lambda_{\max}(\Gamma) \leq 1 + \|m\|^2 + \max\{1, \Sigma_e \sum_{j\geq 0} \|\Psi_j\|^2\},
    \]
    which are bounded by our example assumptions. Exponential strong mixing is due to \cite[Theorem 1]{mokkadem1988mixing}.
\end{proof}
\subsection{Example \ref{example:ARassumptions}}\label{subsec:exampleARAssumption}
\begin{proof}
        To see this, note that the causality assumption means $
    \sum_{j=0}^\infty |\psi_j| < \infty,$ (e.g.\ \citealt{brockwell1991time})
    so that, for any $\lambda \in \mathbb{R}$,
    \[
    \mathbb{E}[\exp(\lambda \varepsilon_{i,t})] = \prod_{j=0}^\infty \mathbb{E}[\exp(\lambda \psi_j r_{i,t-j})] \leq \prod_{j\geq 0} \exp\Big(\frac{\sigma^2_r\lambda^2 \psi_j^2}{2}\Big) = \exp\Big(\frac{\sigma_r^2 \lambda^2}{2}\sum_{j\geq 0} \psi_j^2 \Big),
    \]
    so it has proxy variance at most $\sigma^2 \|\boldsymbol{\psi}\|_{\ell_2}^2$. Furthermore, \cite{mokkadem1988mixing} proves that causal ARMA models with absolutely continuous innovations are exponentially strongly mixing. Finally, the spectral density of $(\varepsilon_{i,t})_{t\geq1}$ has the form
    \[
    f(\lambda) = \frac{\sigma^2}{2\pi} \frac{1}{|1 - \sum_{k=1}^p \phi_k e^{-ik\lambda}|^2}, \quad \lambda \in [-\pi ,\pi],
    \]
    with $\sigma^2 := \V(r_{i,t})$ \citep{brockwell1991time}, and causality implies that all roots of the denominator lie outside the unit circle, so that, denoting $\Phi(z) = 1 - \sum_{k=1}^p\phi_k z^p$,
    \[
    0 < \frac{\sigma^2}{2\pi\max_{\lambda} |\Phi(e^{-i\lambda})|^2} \leq f(\lambda) \leq \frac{\sigma^2}{2\pi \min_{\lambda}|\Phi(e^{-i\lambda})|^2} <\infty.
    \]
\end{proof}

\subsection{Example \ref{example:UI}}\label{subsec:UIexample}
\begin{proof}
        We have, for any $v \in \mathbb{S}^{d-1}$, that $(v^\tr x_{1,t}, \dots, v^\tr x_{1,T})$ is Gaussian. Furthermore, the covariance matrix $\Sigma_T(v)$ has (see e.g. Theorem 11.8.3 of \citealt{brockwell1991time})
    \[
    \lambda_{\min}(\Sigma_T(v)) \geq \frac{\lambda_{\min}(\Sigma_z)}{\sup_\omega s^2_{\max}(P(e^{-i\omega}))},
    \]
    with $P(e^{-i\omega}) := I - \sum_{k=1}^q A_k e^{-i\omega k}$, which has a bounded spectrum by the causality assumption, and where $s_{\max}(P(e^{-i\omega}))$ is the largest singular value of $P(e^{-i\omega})$
. Thus, since $\lambda_{\min}(\Sigma_z)>0$ (the non-degenerate innovations case), there exists some $c>0$ such that $\frac{\lambda_{\min}(\Sigma_z)}{\sup_\omega s^2_{\max}(P(e^{-i\omega}))} \geq c$, independent of $T$ and $v$. It follows that
    \[
    f_{v^\tr x_{1,t}, \dots, v^\tr x_{1,T}}(w) \leq (2\pi c)^{-T/2} = \Big(\frac{1}{\sqrt{2 \pi c}}\Big)^T,
\]
so we can take $K:= \sqrt{\frac{1}{2\pi c}}$.
\end{proof}

\endgroup

\end{document}